\documentclass[11pt,reqno]{amsart}
\usepackage{xifthen}
\usepackage{subfig,yfonts}
\usepackage{pkgfile}

\def\Tree{{\sf T}}

\newcommand{\Graph}{{\sf G}}
\newcommand{\sfB}{{\sf B}}
\newcommand{\sfK}{{\sf K}}
\newcommand{\sfR}{\sf R}
\newcommand{\sfcK}{\sK}
\newcommand{\loc}{\stackrel{loc}{\longrightarrow}}

\newcommand{\st}{\preccurlyeq}
\DeclareMathOperator*{\argmaxx}{argmax}
\newcommand{\lwc}{\stackrel{\text{\it lwc}}{\longrightarrow}}

\newcommand{\hts}{s}
\newcommand{\htgs}{\gs}
\newcommand{\ws}{\hat s}
\newcommand{\ueta}{\ul \eta}

\renewcommand{\red}{}
\newcommand{\tfa}{\frak{a}}
\def\tTree{{\bf T}}
\newcommand{\tGraph}{{\bf G}}
\newcommand{\tsfB}{{\bf B}}

\newcommand{\sfS}{{\bf S}}
\renewcommand{\purple}{}

\begin{document}
\title[Potts and Random Cluster Measures on Locally Tree-Like Graphs]{
Potts 
and random cluster measures\\ on locally regular-tree-like graphs}

\author[A.\ Basak]{Anirban Basak
}
 \author[A.\ Dembo]{Amir Dembo
 }
 \author[A. \ Sly]{Allan Sly
 }

 \address{$^\star$International Centre for Theoretical Sciences
 \newline\indent Tata Institute of Fundamental Research
 \newline\indent Bangalore 560089, India
}

 \address{$^\dagger$Department of Mathematics and Statistics, Stanford University
\newline\indent Sequoia Hall, 390 Serra Mall, Stanford, CA 94305, USA}

\address{$^{\$}$Department of Mathematics, Princeton University
\newline\indent Fine Hall, Princeton, NJ 08540, USA}

\date{\today}



\maketitle

\begin{abstract}
Fixing $\beta \ge 0$ and an integer $q \ge 2$, consider the ferromagnetic $q$-Potts measures
$\mu_n^{\beta,B}$ on finite graphs $\Graph_n$ on $n$ vertices, with external 
field strength $B \ge 0$ 
and the corresponding random cluster measures $\varphi^{q,\beta,B}_{n}$.
Suppose that as $n \to \infty$ the uniformly sparse 
graphs $\Graph_n$ converge locally to an infinite $d$-regular tree $\Tree_{d}$, 
$d \ge 3$. 
We show that the convergence of the Potts free energy density to its Bethe-Peirles replica symmetric 
prediction (which has been proved in case $d$ is even, or when $B=0$), yields the local weak convergence 
of $\varphi^{q,\beta,B}_n$ and $\mu_n^{\be,B}$ 
to the corresponding free or wired random cluster measure, Potts measure, respectively,
on $\Tree_{d}$.
The choice of free versus wired limit  is according to which has the larger  Potts Bethe functional value, with  mixtures of these two appearing {as limit points on} the critical line $\beta_c(q,B)$
where these two values of the Bethe functional coincide. 
For $B=0$ and $\beta>\beta_c$,  we further  establish a pure-state decomposition by 
showing that conditionally on the same 
dominant color $1 \le k \le q$, the $q$-Potts measures on such edge-expander 
graphs $\Graph_n$ converge locally to the $q$-Potts measure on $\Tree_{d}$ 
with a boundary wired at color $k$.   
\end{abstract}

\section{Introduction}
{For a finite graph} $\Graph=(V,E)$, {parameters $\be, B \in \R$} and an integer $q \ge 2$, 
the Potts measure $\mu_\Graph^{\be, B}(\cdot)$ on $\Graph$ is a probability measure on $[q]^V$, 
given by
\[
\mu_{\Graph}^{\be,B}(\ul \sigma) :=\f{1}{Z_\Graph(\be, B)} \exp \Big\{ \be \sum_{(i,j) \in E} \delta_{\sigma_i, \sigma_j} + B \sum_{i \in V} \delta_{\sigma_i, 1}\Big\}, \quad \ul \sigma \in [q]^V,
\]
where $\delta_{\sigma, \sigma'} := {\bf 1} \left\{ \sigma= \sigma' \right\}$ and $Z_\Graph(\be, B)$ is the normalizing constant (commonly known as the {\em partition function}). 
{Setting $[n]:=\{1,2,\ldots, n\}$, our goal here is to characterize the limits, as $n \to \infty$, of Potts measures in the ferromagnetic regime $\be,B \ge 0$, for \emph{uniformly sparse}} 
graphs $\{\Graph_n:=([n], E_n)\}_{n \in \N}$ that converge
locally 
to the $d$-regular infinite tree $\Tree_d:=(\Tree_d,o)$ 
rooted at a distinguished vertex $o$, as defined next.
\begin{dfn}[Uniform sparsity and local weak convergence of graphs]\label{dfn:loc-graphs}
A sequence of graphs $\Graph_n=([n], E_n)$ (possibly random) is {\em uniformly sparse}, if
\beq\label{eq:unif-sparse}
\lim_{L \to \infty}\limsup_{n \to \infty} \E_n [\Delta_{I_n} {\bf 1} \{\Delta_{I_n} \ge L\}] =0,
\eeq
where $I_n$ is uniform on $[n]$, $\E_n$ denotes the expectation with respect to the (possible) randomness of the graph $\Graph_n$ and the randomly chosen vertex $I_n$, and $\Delta_i$ denotes the degree of a vertex $i \in [n]$. 
Given a graph $\Graph$ (possibly infinite), and $t \in \N$, {let} $\sfB_{v,\Graph}(t)$ be the 
subgraph induced by 
$\{v' \in V(\Graph): {\rm dist}_\Graph(v,v') \le t\}$, {for the graph distance} ${\rm dist}_\Graph(\cdot, \cdot)$ {on} $\Graph$, {writing 
$\sfB_v(t)$ when 
$\Graph$ is clear from the context and $\Tree_d(t):= \sfB_{o,\Tree_d}(t)$.}
A sequence $\Graph_n=([n], E_n)$ is said to \emph{converge locally weakly} to 
$\Tree_d$, 
if
\beq\label{eq:loc-conv}
{\lim_{n \to \infty} \E_n \big[ {\bf 1} (\sfB_{I_n}} (t) \ncong \Tree_d(t)) \big]  = 0, 
\quad 
\quad \text{ for all } t \in \N\,,
\eeq
{and} by a slight abuse of notation, we write $\Graph_n \loc \Tree_d$ {when both} 
\eqref{eq:unif-sparse} and \eqref{eq:loc-conv} hold.
\end{dfn}

It is widely believed that statistical physics models on large locally tree-like graphs are a good proxy for models on the integer lattice $\Z^d$ for large $d$ or for those with long interaction range, and therefore the study of such models on locally tree-like graphs is considered 
a test bed {for}
 {\em mean-field theory}. Based on the {\em cavity method} physicists predict {that}, for a wide class of ferromagnetic spin systems, the asymptotic {\em free energy density} 
{is} given by the {\em Bethe-Peirles prediction}, the maximum of the {\em Bethe free energy functional} (see Definition \ref{dfn:bethe-recursion} below) over all `meaningful' fixed points of the {\em belief propagation equations} 
(cf.~\red{\cite[Chapter 17]{MM09} and the references therein.  See also \cite[Chapter 3]{DM2} for a mathematical 
treatment of this prediction,  in the replica-symmetric regime,  which is expected to hold for 
all ferromagnetic models}).
In the large $n$ limit, the relative probabilities of these meaningful fixed points, commonly termed as the `pure states', are further conjectured to be dictated by their respective Bethe free energies. In particular, if the Bethe free energy is maximized by a unique pure state then the ferromagnetic spin system, in the large $n$ limit, 
{is} governed by that pure state, while if there are multiple maximizers 
{then the system be governed in the limit}
by a mixture of such maximizers. 
{While these} conjectures 
been 
in the physics literature for quite 
some time,
verifying them rigorously poses 
serious mathematical challenges. In this paper,  we prove the latter conjecture for the ferromagnetic Potts model, with an external magnetic field, on locally tree-like graphs when the limiting tree is a regular tree.

{After a series of works, the Bethe-Peirles prediction for the free energy density {of} Ising measures 
(namely, Potts measures with $q=2$), on locally tree-like graphs was rigorously established in \red{\cite{DMS}} in full generality (see the references \red{in \cite{DMSS}}
for a historical account). 
In the context of the limit of Ising measures on locally tree-like graphs, the first success was attained in \cite{MMS}, where they showed that under the assumption 
%
$\Graph_n \loc \Tree_d$ the measures
$\mu_{\Graph_n}^{\be, B}$, for $q=2$, 
locally weakly convergence to the free or wired/plus 
Ising measures on $\Tree_d$ (the choice of limit depends on whether both $B=0$ 
and $\be>\be_c$, with $\be_c$ being the critical inverse temperature, or not)}. Such local weak convergence has been strengthened in \cite{BD} to show that it holds for a much larger collection of graph sequences $\{\Graph_n\}$, covering 
also sequences that converge locally to certain non-regular, possibly random trees
(e.g. to a multi-type Galton-Watson tree of minimum degree $d_{\min} \ge 3$). 

{As mentioned above, this phenomenon should extend} 
to Potts measures with $q \ge 3$. However, as explained in the sequel, 
{proving such a convergence is way} more challenging, especially 
for the two-dimensional region of parameters $\sfR_{\ne}$, 
where a-priori two different scenarios 
are possible for such a limit law. 
{The emergence of the two-dimensional region $\sfR_{\ne}$ is due to the presence \red{of} two distinct phase transitions for the Potts measures with $q \ge 3$ on $\Tree_d$:~the uniqueness/non-uniqueness transition and the disordered/ordered transition, a phenomenon absent in the Ising setting. Moreover, on the critical line $\sfR_c$ {within $\sfR_{\ne}$,
where} the Bethe free energies of both meaningful belief propagation fixed points are the same, the predicted limit is a mixture of both pure states, a behavior also absent in the Ising case. 

We note in passing that the phase diagram of Potts measures, and in particular the local weak
convergence of $\mu_n^{\be,B} := \mu^{\be,B}_{\Graph_n}$, is related to algorithmic questions of much
current interest (see \cite[Section 1.2]{HJP22+} and the references therein, where $B=0$).
Such convergence is also related to the metastability of {the} Glauber dynamics and {of the} Swendsen-Wang chain for $\mu^{\be,B}_n$, {as well as the non-reconstructions of the paramagnetic and ferromagnetic states}
when $(\be,B) \in \sfR_{\ne}$ (for $B=0$, see \cite{CO23} {and the references therein for a list representative works in this area}).

To state our result about such convergence of Potts measures 
$\mu_n^{\be, B}$ when $\Graph_n \loc \Tree_d$, we proceed 
to define the} Potts measures on $\Tree_d$ {that one expects to find in the limit}.
\begin{dfn}[Potts {on trees}, with given boundary conditions]\label{dfn:potts-bounary-B-0}
For each $t \in \N$ and $\ddagger \in [q] \cup \{\free\}$ consider the following Potts measures on $\Tree_d(t)$:
\begin{eqnarray}
\mu^{\be, B}_{ \ddagger,t}(\ul{\sigma})= \mu^{\be, B}_{ \ddagger,\Tree_d, t}(\ul{\sigma}) & \propto & \exp \Big\{ \beta \sum_{(i,j) \in E(\Tree_d(t))} \delta_{\sigma_i, \sigma_j} + B \sum_{i \in V(\Tree_d(t))} \delta_{\sigma_i, 1}\Big\}  \prod_{u \in \partial \Tree_d(t)}  \red{\wt \nu_\ddagger}(du)
\notag
\end{eqnarray}
where $\partial \Tree_d(t):=V(\Tree_d(t)) \setminus V(\Tree_d(t-1))$, $\red{\wt \nu_{\free}}$ is the uniform measure on $[q]$, and $\red{\wt \nu_\ddagger}$ is the Dirac measure at $\ddagger$ for $\ddagger \in [q]$. 
The Potts measure with $\ddagger$ boundary condition
$\mu_{\ddagger}^{\be, B}:= \displaystyle{\lim_{t \uparrow \infty}}
\mu_{\ddagger, \Tree_d, t}^{\be, B}$ 
{exists}
(in the sense of weak convergence 
over $[q]^{V(\Tree_d)}$ 
{with its cylindrical} $\sigma$-algebra), 
for any $\be, B \ge 0$ and $\ddagger \in \{\free, 1\}$, {and} 
\(
\mu_{\ddagger}^{\be,0}
\)
{also} exists for any $\ddagger \in [q]$ and $\be \ge 0$ {(see} 
\cite[Chapter IV, Corollary 1.30]{liggett}, {in case $q=2$, and} 
Remark \ref{rmk:infinite-vol-exist} for an outline of the proof when $q \ge 3$). 
\end{dfn}

{For} $q=2$, if $B>0$ {or $\be \le \be_c(d)
$}
the measures $\mu_{1}^{\be, B}$ and $\mu_{\free}^{\be, B}$ coincide (see \cite[Section 4.1]{DMS}), {while for any} $\be > \be_c(d)$ the measures $\mu_{1}^{\be, 0}$ and $\mu_{\free}^{\be, 0}$ differ 
(see \cite{L89}). {As mentioned earlier, for} $q \ge 3$ the picture is more delicate
{and the} two-dimensional non-uniqueness 
region {$\sfR_{\ne}$ of $(\be, B)$ values} where $\mu_{1}^{\be, B} \ne \mu_{\free}^{\be, B}$
(see \cite[Theorem 1.11]{DMS}), 
{plays} a significant role in {describing our results. Indeed, within this region, the limit is 
determined by the relation between the}
{free energy densities} of $\mu_n^{\be,B}$
\beq\label{eq:Phi-G-def}
\Phi_n(\be, B) := \f{1}{n} \log Z_{\Graph_n} (\be, B),
\eeq
{and the Bethe (free energy) functional at certain Bethe recursion fixed points,
which we define next.}
\begin{dfn}[Bethe functional; Bethe recursion and its fixed points]\label{dfn:bethe-recursion}
{Denoting by} $\cP([q])$ the set of all probabilities on $[q]$, the Bethe functional 
$\Phi: \cP([q]) \mapsto \R$ {is given by} 
\beq\label{eq:BP-fun}
\Phi(\nu) = \Phi^{\be, B}(\nu):=\log \Big\{ \sum_{\sigma \in [q]} e^{B \delta_{\sigma,1}}\left( (e^{\be}-1) \nu(\sigma)+1\right)^d \Big\} - \frac{d}{2} \log \Big\{ (e^{\be}-1)\sum_{\sigma \in [q]} \nu(\sigma)^2+1 \Big\}.
\eeq
Of primary interest to us 
{is the value of $\Phi(\cdot)$ at} $\nu_{\free}^{\be, B}$ and at $\nu_1^{\be, B}$. {The latter  
are the fixed points of the} Bethe recursion ${\sf BP}: \cP([q]) \mapsto \cP([q])$ 
\beq\label{eq:BP-recursion}
({\sf BP} \nu)(\sigma) := \f{1}{z_\nu} e^{B \delta_{\sigma,1}} \left( (e^{\be} -1) \nu(\sigma) +1\right)^{d-1}, \quad \sigma \in [q],
\eeq
({with} $z_\nu$ {denoting} the corresponding normalizing constant), {obtained 
as} the limit of successive iterations of ${\sf BP}$ starting from the uniform probability measure on $[q]$ and from the probability measure supported on spin $1$, respectively (cf.~
\cite[Page 4208, Section 4.3]{DMS} for the existence of both limits, which we also denote as
$\nu_{\free}$ and $\nu_1$,
respectively).
\end{dfn}

\purple{Throughout this paper we {make} the assumption that $\Graph_n \loc \Tree_d$ for some $d \ge 3$. We also need the following result on
Potts free energy densities {$\Phi_n(\be,B)$} on them.} 
\begin{prop}\label{asmp:main}
\purple{Fix $\be, B \ge 0$ and integer $q \ge 2$. 
Suppose $\Graph_n \loc \Tree_d$ for some $d \ge 3$. Then} 
\beq\label{eq:free-energy-limit-assume}
{\lim_{n \to \infty}}
\Phi_n(\be, B)  =  \max\big\{ \Phi(\nu_{\free}^{\be, B}),\Phi(\nu_1^{\be, B})\big\}.
\eeq
\end{prop}
\purple{Proposition \ref{asmp:main} is a culmination of several works. For Ising measures (namely $q=2$), the {identity} \eqref{eq:free-energy-limit-assume}
for any $\be, B \ge 0$ and $d \ge 3$ was proved in 
\cite{DMS}. For $q \geq 3$ this was established for $d \in 2 \N$ in
\cite{DMS, DMSS},  and later for a $d$-regular graph sequence $\{\Graph_n\}_{n \in \N}$ and $B=0$ (see \cite{BBC22+}, following
upon \cite{HJP22+} which required large $q \ge d^{C d}$ and $d \ge 5$).  Very recently, Proposition \ref{asmp:main} was proved in \cite{CH25} in its full generality (as stated above).} 
\red{See also the  
analysis in \cite{GSVY} of the first and second moments of $Z_{\Graph_n}(\be,B)$ 
for uniformly random $d$-regular $\Graph_n$}.

{In view of \eqref{eq:free-energy-limit-assume}, we expect for $B>0$ to arrive at a limit law 
$\mu^{\be,B}_{\free}$, $\mu^{\be,B}_1$ or a mixture thereof, according to the partition of 
$\sfR_{\ne}$ to $\sfR_{\free}$, $\sfR_1$, and $\sfR_{c}$ which correspond, respectively,
to whether $\Phi(\nu_{\free})$ is larger, smaller, or equal to $\Phi(\nu_1)$. Thus, prior to
stating our limit theorem, we first describe the shape of this partition of $\sfR_{\ne}$
(taken mostly from the literature, with a few items supplemented in Appendix \ref{sec:app}).} 
See also Figure \ref{fig1} for a pictorial representation {of this description.}
\begin{prop}[Description of {the} non-uniqueness region]\label{prop:non-unique-regime}
{Fix} $q, d \ge 3$.
\newline
There exist $0 < B_+ = B_+(q,d) < \infty$ and smooth curves $\be_{\free}, \be_c$, and $\be_+$ defined on $[0,B_+]$, with $0< \beta_{\free} (B) < \be_c (B) < \beta_+ (B)$ for $B \in [0,B_+)$ and $\beta_{\free}(B_+)=\beta_c(B_+) = \beta_+(B_+)$ {such that} 
\beq\label{eq:sf R}
\sfR_{\ne} :=\left\{(\be', B') \in (0,\infty)^2:  0<B' < B_+, \be' \in [\be_{\free}(B'), \be_+(B')]\right\}\cup \left\{(\be',0): \be' \ge \be_{\free}(0) \right\},
\eeq
and the following holds:
\begin{enumeratei}
\item If $(\be, B) \in [0,\infty)^2 \setminus \sfR_{\ne}$ then $\nu_{\free}=\nu_1$ 
{and consequently, also}
\begin{align}\label{eq:f-eq-1-B-ge-0}
\mu_{1}^{\be,B} &= \mu_{\free}^{\be, B}, \qquad \text{ whenever } \qquad \quad B >0, \\
\label{eq:f-eq-all-B-eq-0}
\mu_{\ddagger}^{\be,0} &= \mu_{\free}^{\be, 0} \qquad \text{ for all } \ddagger \in [q], \qquad B =0.
\end{align}
\item If $(\be, B) \in \sfR_{\neq}$ then $\nu_{\free}
(1) < \nu_1
(1)$.
{Further, setting}
\beq
\sfR_{\free}:= \left\{ (\be, B) \in \sfR_{\neq}: \be < \be_c(B)\right\},  \sfR_{1}:= \left\{ (\be, B) \in \sfR_{\neq}: \be > \be_c(B)\right\}, \sfR_c:= \sfR_{\neq} \setminus (\sfR_{\free} \cup \sfR_1),
\eeq
{we have that}
\beq\label{eq:free-v-1}
\left\{
\begin{array}{ll}
\Phi(\nu_{\free}^{\be,B}) > \Phi(\nu_1^{\be,B})
& \mbox{ if } (\be, B) \in \sfR_{\free},\\
\Phi(\nu_{\free}^{\be,B}) = \Phi(\nu_1^{\be,B})
& \mbox{ if } (\be, B) \in \sfR_c,\\
\Phi(\nu_{\free}^{\be,B}) < \Phi(\nu_1^{\be,B})
& \mbox{ if } (\be, B) \in \sfR_1.
\end{array}
\right.
\eeq
\end{enumeratei}
\end{prop}
\purple{Let us mention in passing that an explicit description of the critical line $\sfR_c$ can be found in \cite[Eqn.~(1.53)]{CH25} (appeared after the posting of this paper).}

\red{Note that from Proposition \ref{prop:non-unique-regime}, for $B=0$, we have $\mu_{\free}^{\be,0} \ne \mu_1^{\be,0}$
for any $\be \in [\be_{\free}(0), \infty)$. Thus $\be_{\free}(0)$ is the so called uniqueness/non-uniqueness transition threshold. On the other hand, as will be evident from Theorems \ref{thm:main-1} and \ref{thm:main-2} below, $\be_c(0) (> \be_{\free}(0))$ is the disordered/ordered transition threshold. This is in stark contrast to the Ising setting where $\mu_1 \ne \mu_{\free}$ only when $B=0$ and $\be > \be_c$.}

{The resolution of \red{the Ising} case in \cite{MMS,BD} crucially relies on the  
\abbr{FKG} inequality for a stochastic ordering of the edge-correlations for all plausible local marginals 
of $\mu^{\be,0}_n$ between those two extreme candidates. Lacking such 
monotonicity property when $q \ge 3$, we instead couple each Potts measure on $\Graph_n$ with  
the corresponding random cluster measure (\abbr{rcm}),
thereby allowing us to utilize the \abbr{FKG} property of the latter, {and for doing so when} 
$B>0$,
we
first amend our graphs by a \emph{ghost vertex}.}
\begin{dfn}[{Amending graphs} by a ghost vertex]\label{dfn:ghost-vertex}
{From a finite or infinite} graph $\Graph=(V,E)$
{we get} $\Graph^\star$ by adding from every $v \in V$ 
an edge to the additional \emph{ghost vertex} $v^\star$. That is, $\Graph^\star=(V^\star, E^\star)$
for $V^\star:= V \cup \{v^\star\}$ and $E^\star:= E \cup \{(v, v^\star), v \in V\}$. 
With a slight abuse of notation, for any $i \in V$ and $t \in \N$, we take for $\sfB_i^\star(t)$ 
the subgraph $\sfB_{i,\Graph}(t)$ amended by 
$v^\star$ and all edges between $v^\star$ and {$\sfB_{i,\Graph}(t)$.}
That is, $V(\sfB_i^\star(t))=V(\sfB_i(t))\cup \{v^\star\}$ and $E(\sfB_i^\star(t))=
E(\sfB_i(t)) \cup \{(v, v^\star): v \in V(\sfB_i(t))\}$.
We likewise set $\Tree_d^\star(t):=\sfB_o^\star(t)$ for the root $o$ of $\Tree_d$ (whereby 
$\Tree_d^\star$ is also the increasing limit of $\Tree_d^\star(t)$ as $t \uparrow \infty$).
Denoting by $\cG(t)$ the collection of all rooted graphs $(\Graph,o)$ 
of depth $t \in \N$ (namely, with all vertices at distance at most $t$ from the root $o$), and 
correspondingly setting $\cG^\star(t):= \{\sfB_o^\star(t): (\Graph, o) \in \cG(t)\}$, we further 
see 
that $\sfB_i^\star(t) \in \cG^\star(t)$ for any $i \in V(\Graph)$ and $t \in \N$.
\end{dfn}
 
{As promised, we proceed with the definition of the \abbr{rcm} on $\Graph^\star$ that
corresponds to the $q$-Potts measure $\mu^{\be,B}_{\Graph}$, for a finite graph $\Graph$.}
\begin{dfn}[Random cluster measure for a finite graph]\label{dfn:rcm_dfn}
Let $\Graph=(V,E)$ be a finite graph and $\Graph^\star$ be as in Definition \ref{dfn:ghost-vertex}. 
Fixing $q>0$, the 
\abbr{RCM} with external field $B \ge 0$
and percolation parameter $\be \ge 0$, 
is the probability measure on subgraphs of $\Graph^\star$,
given by
\[
\varphi_{\Graph}^{\be,B}(\ul{\eta}) \propto \Big[\prod_{e \in E^\star} p_e^{\eta_e}(1-p_e)^{1-\eta_e}\Big] 
q^{|{\sf C}(\ul \eta)|}, \qquad \ul{\eta} \in \{0,1\}^{E^\star},
\]
where ${\sf C}(\ul \eta)={\sf C}_{\Graph^\star}(\ul \eta)$ denotes the collection of connected components for edge configuration 
$\ul \eta$,  $|A|$ denotes \red{the} size of a finite set $A$ and 
\beq\label{eq:dfn:p-e}
p_e:=\left\{\begin{array}{ll}
1-e^{-\beta} &\mbox{if } e \in E\\
1-e^{-B} & \mbox{if } e \in E^\star\setminus E 
\end{array}
\right. .
\eeq
Summing $\varphi_{\Graph}^{\be,B}(\cdot)$ over 
bonds configurations on $E^\star \setminus E$, yields {an alternative \abbr{RCM} with} external field 
\beq\label{eq:mar-rcm}
\wt \varphi^{\be, B}_\Graph (\ul \eta) \propto \prod_{e \in E} p_e^{\eta_e}(1-p_e)^{1-\eta_e}
\prod_{C \in {\sf C}(\ul \eta)} (1+(q-1)e^{-B|C|}), \qquad \ul{\eta} \in \{0,1\}^{E} 
\eeq
(now without a ghost vertex, see also \cite{BBCK00} for such \abbr{RCM}-s 
in the presence of several external fields).
\end{dfn}

\begin{rmk} \red{For a lighter presentation,  we use the notation $\varphi^{\be,B}_\Graph$ and call this measure
the \abbr{rcm} on $\Graph$,  even though in case $B>0$ it is actually a measure on $\Graph^\star$.  We further}
use the shorthand $\varphi_n^{\be, B}$ for $\varphi_{\Graph_n}^{\be, B}$, where 
$\{\Graph_n\}_{n \in \N}$ is a sequence of graphs under consideration. Anticipating 
our Edwards-Sokal coupling of $\varphi_{\Graph}^{\be, B}$ with 
the {$q$-}Potts measure $\mu_{\Graph}^{\be, B}$ (cf.~Section \ref{sec:subseq-limit}), 
we have opted to index the \abbr{rcm} in Definition \ref{dfn:rcm_dfn} by 
$\be$ instead of $p$. {Indeed, for} $\varphi_\Graph^{\be,B}$ the 
Edwards-Sokal coupling 
yields a nicer conditional distribution of the spin {(Potts)} variables given 
the bond {(\abbr{rcm})} variables {$\ul{\eta}$, than what we get 
if using instead $\wt \varphi^{\be,B}_\Graph$} 
(compare Lemma \ref{lem:coupling-es}(ii) and \cite[Eqn.~(17)]{DMSS}).
{However,} for $B=0$ necessarily $\eta_e=0$ at any edge 
$e \in E^\star\setminus E$, with $\varphi_\Graph^{\be,0}(\cdot)
=\wt \varphi^{\be,0}_{\Graph}(\cdot)$ thus 
{matching the} \emph{standard} 
\abbr{rcm} definition (see \cite[Section 1.2]{G-RC}). 
\end{rmk}

Since {our proof of} the local weak limits for $\mu_n^{\be, B}$ goes via a coupling of those measures with $\varphi_n^{\be, B}$, it also {requires us to identify} the local weak limits of 
$\varphi_n^{\be, B}$. Similarly to the Potts case, {these involve the special choices of free and wired
\abbr{rcm} on $\Tree^\star_d$, which we define next.}
\begin{dfn}[The free and the wired \abbr{rcm}-s on $\Tree_d$]\label{dfn:inf-vol-RC}
Suppose first that $B >0$. Fixing $t \in \N$, let {$\wh \sfB_i^\star(t)$} be the 
graph {$\sfB_i^\star(t)$} amended by all edges of the star graph {$\sfK^\star(\partial \sfB_i (t))$}
for ghost $v^\star$ and the complete graph {$\sfK(\partial \sfB_i(t))$ on $\partial \sfB_i(t)$. 
In case $\Graph=(\red{\Tree_d},o)$, we denote $\wh \sfB_o^\star(t)$ by $\wh \Tree^\star_d(t)$,} calling hereafter the edges of $\sfK^\star(\partial \Tree_d(t))$ as
{\em boundary edges} of $\wh \Tree_d^\star(t)$
(there are in $\wh \Tree_d^\star(t)$ two distinct copies of each edge between $v^\star$ and 
$\partial \Tree_d(t)$, {but} only one of them is a boundary edge).
Let $\gf(\free)=0$ and $\gf(\wired)=1$, setting for $\ddagger \in \{\free, \wired\}$ 
the probability measures
\begin{equation}\label{eq:f-w-rc-t}
\varphi_{\ddagger, t}^{\be, B}(\cdot) := \varphi_{\wh \Tree_d^\star(t)}^{\be, B}(\cdot\,|\, \eta_e =\gf(\ddagger) , \text{ for all } \;\; e \in E(\sfK^\star(\partial\Tree_d(t)))),
\end{equation}
{and define} the wired and the free \abbr{rcm} on $\Tree_d$ with parameters $\be$ and $B$, 
{as their limits}
\beq\label{eq:f-w-rc-limit}
\varphi_{\wired}^{\be, B}:= \lim_{t \uparrow \infty}\varphi_{\wired, t}^{\be, B} \qquad \text{ and } \qquad \varphi_{\free}^{\be, B}:= \lim_{t \uparrow \infty}\varphi_{\free, t}^{\be, B}
\eeq 
\red{(when $B>0$ these are actually measures on $\Tree_d^\star$).}
Similar to {Definition 
\ref{dfn:potts-bounary-B-0}, the limits in \eqref{eq:f-w-rc-limit}} 
are in the sense of weak convergence of probability measures 
{over} $\{0,1\}^{E(\Tree_d^\star)}$ {with its cylindrical} $\sigma$-algebra
{and} the existence of {these} limits is straightforward (see 
the proof of Lemma \ref{lem:rc-fw-limit}). 
{For $B=0$ the ghost $v^\star$ is an isolated vertex, so proceeding without it, we consider
\eqref{eq:f-w-rc-t} for
$\wh \Tree^o_d(t)$ based on $\Tree_d(t)$ and $\sfK(\partial \Tree_d(t))$, in lieu of 
$\wh \Tree_d^\star (t)$, $\Tree_d^\star(t)$ and $\sfK^\star(\partial \Tree_d(t))$, 
respectively. The measures $\varphi_{\ddagger}^{\be, 0}$ 
on $\{0,1\}^{E(\Tree_d)}$ which are thereby defined via \eqref{eq:f-w-rc-limit},}
match the standard notion of 
 free and wired \abbr{rcm} on $\Tree_d$ (as defined in \cite[Chapter 10]{G-RC}). 
\end{dfn} 

We note in passing the results of  \cite{BGJ96} about the asymptotic size and shape of the large
connected components of the \abbr{rcm} for complete graphs $\Graph_n=\sfK_n$,  at 
$\beta=\lambda/n$ and $B=0$.

{Having defined the various parameter domains and candidate limit laws, 
we arrive at the final ingredient for stating our theorem, namely that of} local weak convergence 
of probability measures, {to which end we must first define the 
appropriate spaces of such measures.}
\begin{dfn}[Spaces of probability measures]\label{dfn:prob-space}
{Fixing} finite sets $\cX$ and $\cY$ (throughout this paper $\cX=[q]$ and $\cY=\{0,1\}$),
we equip $\cX^{V(\Tree_d^\star)} \times \cY^{E(\Tree_d^\star)}$ {with the product topology.
We denote by} $\cP_d^\star$ the set of all {Borel} probability measures on 
$\cX^{V(\Tree_d^\star)} \times \cY^{E(\Tree_d^\star)}$, endowed 
with the topology of weak convergence. {Similarly,} for any $t \in \N$, let $\cP_d^{\star,t}$ denote the set of all probability measures on {the finite set} $\cX^{V(\Tree_d^\star(t))} \times \cY^{E(\Tree_d^\star(t))}$, {equipping $\cP_d^\star$ with the cylindrical $\sigma$-algebra that corresponds to 
$\{\cP_d^{\star,t}, t \in \N \}$.}
Next, for a  probability measure $\gm$ on $\cP_d^\star$ and any $t \in \N$, denote by $\gm^t$ the (Borel) 
probability measure on $\cP^{\star, t}_d$ obtained from $\gm$ by projection, {which 
we hereafter call} the \emph{$t$-dimensional marginal} of $\gm$. 
\end{dfn}

{Utilizing Definition \ref{dfn:prob-space}, we next define the 
local weak convergence of probability measures.}

\begin{dfn}[Local weak convergence of probability measures]\label{dfn:lwc-measure}
Given graphs $\Graph_n = {([n],E_n)}$ 
and probability measures $\zeta_n$ on $\cX^{[n]\cup\{v^\star\}} \times \cY^{E_n^\star}$,
for any $t \in \N$ and $i \in [n]$, let ${\sf P}_{\zeta_n}^t(i)$ {denote} the law of the triplet $(\sfB_i^\star(t), \ul \sigma_{\sfB_i^\star(t)}, \ul \eta_{\sfB_i^\star(t)})$, {for} $(\ul \sigma, \ul \eta) \in \cX^{[n]\cup\{v^\star\}} \times \cY^{E_n^\star}$ drawn according the law $\zeta_n$.
Combined with uniformly chosen $I_n \in [n]$
this yields random distributions ${\sf P}_{\zeta_n}^t(I_n)$. {While their first marginal 
can be anywhere in $\cG^\star(t)$, when $\Graph_n \loc \Tree_d$
it must converge 
to $\delta_{\Tree^\star_d(t)}$. We then} say that 
$\{\zeta_n\}$
 {\em converges locally weakly} to a probability measure $\gm$ on $\cP_d^\star$ if {further}
${\sf P}_{\zeta_n}^t(I_n) \Rightarrow \delta_{\Tree_d^\star(t)} \otimes \gm^t$ for {any fixed} $t \in \N$, \red{where the notation $\Rightarrow$ is used to denote the convergence in distribution in weak topology}. 
{In case $\gm=\delta_\nu$ for some} $\nu \in \cP_d^\star$, we say that $\{\zeta_n\}$ 
{\em converge locally weakly in probability} to $\nu$, denoted by $\zeta_n \lwc \nu$. 
Denoting the marginals of $\zeta_n$ in the spin variables $\ul \sigma$ and in the bond variables $\ul \eta$, by $\zeta_{n,\spin}$ and $\zeta_{n,\bond}$, 
the local weak convergence of $\{\zeta_{n,\ddagger}\}$ to $\gm_\ddagger$ {is similarly defined for}
$\ddagger \in \{\spin,\bond\}$, 
where $\gm_{\spin}$ and $\gm_{\bond}$ are distributions over the set of probabilities on 
$\cX^{V(\Tree_d^\star)}$ and $\cY^{E(\Tree_d^\star)}$, respectively, and if 
$\gm_\ddagger = \delta_{\nu_\ddagger}$ we say that $\{\zeta_{n,\ddagger}\}$ converges
locally weakly in probability to $\nu_\ddagger$, denoted by $\zeta_{n, \ddagger} \lwc \nu_\ddagger$. 
{Lastly,} if $\zeta_{n,\spin}(\sigma_{v^\star}=1)=1$ for all $n \in \N$ and 
$\zeta_{n,\spin} \lwc \nu_{\spin}$, {we can and shall view} $\zeta_{n,\spin}$ and 
$\nu_{\spin}$ as probability measures on {$\cX^{[n]}$ and} 
$\cX^{V(\Tree_d)}$, respectively.  
\end{dfn}

Introducing the notation
\beq\label{eq:mu-wired}
\mu_{\wired}^{\be, B}:= \left\{
\begin{array}{ll}
\mu_1^{\be, B} & \mbox{if } B >0,\\
\f1q \displaystyle{\sum_{k=1}^q} \mu^{\be, 0}_k & \mbox{ if } B=0 \,,
\end{array}
\right.
\eeq
our first main result is the rigorous justification of the Bethe-Peirles replica-symmetric 
heuristic for our ferromagnetic Potts models. Namely, that at a non-critical $\be,B \ge 0$ the Potts 
laws $\{\mu_n^{\be,B}, n \gg 1\}$ are locally near $\mu_{\free}^{\be,B}$ or near $\mu_{\wired}^{\be,B}$, 
taking per $(\be,B)$ among these two limits the one with larger  
Bethe functional value.
\begin{thm}\label{thm:main-1}
\purple{Under the assumption $\Graph_n \loc \Tree_d$}, \red{we have {in terms of}
$\sfR_{\neq}$, $\sfR_{\free}$, and $\sfR_1$} of Proposition \ref{prop:non-unique-regime}, the following limits:
\begin{enumeratei}
\item \red{If $(\be, B) \in [0, \infty)^2\setminus \sfR_{\neq}$ then $\mu_n^{\be, B} \lwc \mu_{\free}^{\be, B}=\mu_{\wired}^{\be,B}$ and $\varphi_n^{\be, B} \lwc \varphi_{\free}^{\be, B}= \varphi_{\wired}^{\be,B}$, as $n \to \infty$.}
\item \red{If $(\be, B) \in \sfR_{\free}$ then $\mu_n^{\be, B} \lwc \mu_{\free}^{\be, B}$ and $\varphi_n^{\be, B} \lwc \varphi_{\free}^{\be, B}$, as $n \to \infty$.}
\item If $(\be, B)  \in \sfR_{1}$ then $\mu_n^{\be, B} \lwc \mu_{\wired}^{\be,B}$ and $\varphi_n^{\be, B} \lwc \varphi_{\wired}^{\be, B}$, as $n \to \infty$.
\end{enumeratei}
\end{thm}

Further, only global mixtures of $\mu_{\free}^{\be,B}$ and $\mu_{\wired}^{\be,B}$ 
may emerge as the limit points, at any parameter value on the critical line $\sfR_c$ 
(where $\Phi(\nu_{\free})=\Phi(\nu_{\wired}))$.
\begin{thm}\label{thm:crit-1}
Let $d \in 2 \N$, {$d \ge 3$.} Then, 
for $(\be, B) \in \sfR_c$, {$B > 0$ and any $\Graph_n \loc \Tree_d$,}
all local weak limit points of $\{\mu_n^{\be, B}\}$ and $\{\varphi_n^{\be, B}\}$, as $n \to \infty$,
are supported on
\beq
\cM_{\spin}^{\be, B}:= \{{\bm \al} \mu_{\wired}^{\be,B} + (1-{\bm \al}) \mu_{\free}^{\be, B}, {\bm \al} \in [0,1]\}  \text{ and }  \cM_{\bond}^{\be, B}:= \{{\bm \al} \varphi_{\wired}^{\be,B} + (1-{\bm \al}) \varphi_{\free}^{\be, B}, {\bm \al} \in [0,1]\},
\eeq
respectively.
\end{thm}
\begin{rmk}
As $\wt \varphi^{\be,B}_\Graph$ is the marginal of $\varphi_\Graph^{\be, B}$, both 
Theorems \ref{thm:main-1} and \ref{thm:crit-1} hold {also} for the former \abbr{RCM}-s 
(taking for the free and the wired \abbr{RCM}-s, the marginals of 
$\varphi_{\free}^{\be,B}$ and $\varphi_{\wired}^{\be, B}$ on $\{0,1\}^{E(\Tree_d)}$, respectively). 
\end{rmk}

\begin{rmk}
Recently \cite{S23}, using results on the {\em sofic entropy} for a {\em sofic approximation} of the free group of rank {$d/2$ (for $d \in 2 \N$, $d \ge 4$),} has shown  
that {for Potts measures with parameters $(\be, B)=(\be_c(0),0)$ on random $d$-regular 
graphs $\Graph_n$ chosen according to the uniform-permutation-model, as $n \to \infty$,
any local weak limit point in probability,} must be supported on $\cM_{\spin}^{\be_c(0),0}$.
Although neither Theorem \ref{thm:crit-1} nor \cite[Proposition 3.4]{S23} shows that the limit point in context is a genuine measure on $\cM_{\spin}^{\be, B}$ (i.e.~neither Dirac at $\mu_{\free}^{\be, B}$ nor at $\mu_{\wired}^{\be, B}$), this is believed to be the case for $(\be, B) \in \sfR_c$.
Indeed, \cite[Theorem 2]{HJP22+} confirms the latter prediction when 
$d \ge 5$, $q \ge d^{Cd}$ ($C< \infty$ is some absolute constant), $(\be, B)=(\be_c(0),0)$
and $\Graph_n$ are uniformly random $d$-regular graphs. 
{While} their method {might plausibly} be adapted to all $(\be, B) \in \sfR_c$ {we do}
not pursue this here. 
\end{rmk}

\begin{figure}[htbp]
  \centering
   \begin{minipage}[b]{0.46\linewidth}
   \includegraphics[width=\textwidth]{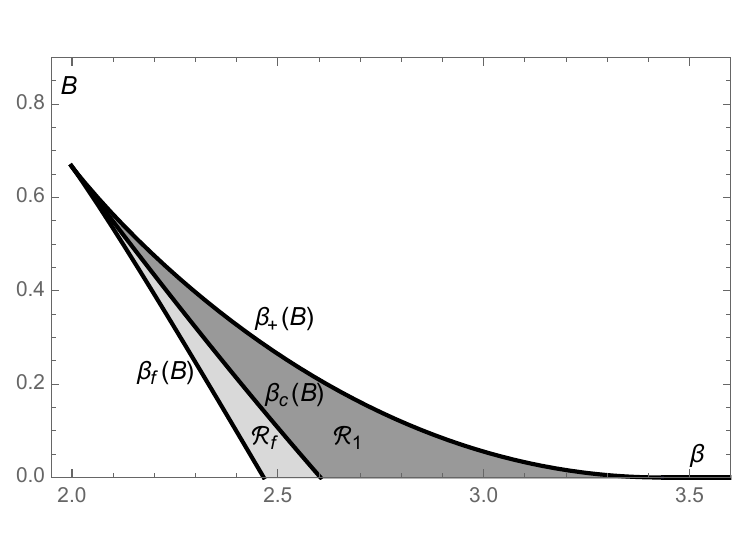}
 \end{minipage}
 \hspace{1cm}
 \begin{minipage}[b]{0.46\linewidth}
  \includegraphics[width=\textwidth]{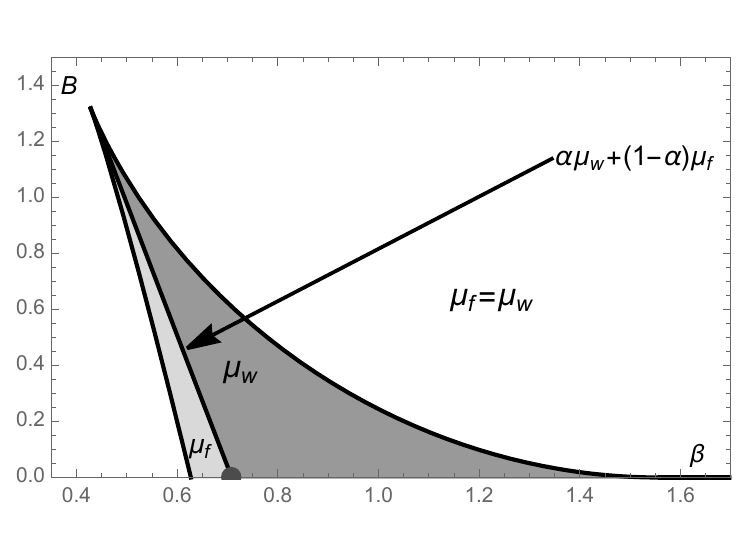}
 \end{minipage}
 \caption{The left panel shows the non-uniqueness regime for the Potts measures on $\Tree_d$  for $q=30$ and $d=3$ \purple{(see Proposition \ref{prop:non-unique-regime} for its definition)}. The lightly shaded region is $\sfR_{\free}$, while the darker region is $\sfR_1$. These two regions are separated by the critical curve $\be_c(B)$, and the non-uniqueness regime $\sfR_{\neq}$ is separated from its complement by the two curves $\be_{\free}(B)$ and $ \be_+(B)$. 
 \newline
 The right panel shows the non-uniqueness regime along with the local weak limits for Potts measures for $q=30$ and $d=10$. Notice the difference in shapes and sizes  of $\sfR_1$ and $\sfR_{\free}$ compared to the left panel. Note that Theorems \ref{thm:main-1} and \ref{thm:crit-1} identify the local weak limit (points) except when $(\be,B)=(\be_c(0),0)$, marked black in the right panel. 
}
  \label{fig1}
\end{figure}

{For any $\be > \be_c(0)$, the $q$-symmetry of the Potts model at $B=0$ gave rise in 
Theorem \ref{thm:main-1}
to a limit $\mu^{\be,0}_{\wired}$ which is a balanced mixture of the $q$ possible pure-states
Potts limit measures $\mu^{\be,0}_k$ on $\Tree_d$. In the Ising setting ($q=2$), this has
been further refined and explained in \cite{MMS}, by showing that for $n \gg 1$ 
the choice between the two pure state limits
matches with high probability
the (random) dominating value $\sK_n(\ul \sigma)$ at the spin configuration 
$\ul \sigma$ on $\Graph_n$.  Towards such a refinement for $q$-Potts
with $q \ge 3$, we proceed to define the Potts measure with fixed dominating spin value.}
\begin{dfn}[Potts with a given dominating spin {value}]\label{dfn:dom-spin}
For a graph $\Graph_n=([n], E_n)$ with spin configuration $\ul \sigma \in [q]^{[n]}$ define
\beq
\sK_n(\ul \sigma):= \text{arg} \max_{k \in [q]} \Big\{ \sum_{i \in [n]} {\bf 1}(\sigma_i =k) \Big\}, \label{eq:sK_n}
\eeq
breaking ties uniformly among the subset of $[q]$ of all maximizer values.\\
For any $\be, B \ge 0$ and $k \in [q]$, we call the probability measure 
\[
\mu_{n, k}^{\be, B} (\cdot):= \mu_n^{\be, B} (\cdot \,|\, \sK_n(\cdot)=k) \,,
\]
{which is supported on $\{\ul \sigma : \sK_n(\ul \sigma)=k \}$, the $q$-Potts with dominating spin $k$.}
\end{dfn}

{As demonstrated in \cite{MMS}, even in the Ising setting, having a pure-state decomposition 
according to the dominating spin value, requires a certain uniform edge-expansion property of $\Graph_n$.
Our next definition states the relevant notion of edge expansion, followed by the statement of the promised
pure-state decomposition according to the dominant spin value.}
\begin{dfn}[Uniformly edge-expander graphs]
A finite graph $\Graph=(V,E)$ is  a $(\delta_1, \delta_2, \lambda)$ edge-expander, if for any set of vertices $S \subset V$ with $\delta_1 |V| \le |S| \le \delta_2 |V|$ we have that $|\partial S| \ge \lambda |S|$, where $\partial S\red{=E(S,S^c)}$ denotes the set of edges between $S$ and $S^c$.  \red{We say that graphs $\{\Graph_n= ([n], E_n)\}_{n \in \N}$ are uniformly edge-expanders if 
\(
\inf\{ n^{-1} |E_n(S,S^c)| : S \subset [n],  \vep n \le |S| \le \frac{n}{2} \} > 0
\) for any fixed $\vep>0$.}
\end{dfn}

\begin{thm}\label{thm:main-2}
Recall the definition of $\be_c(0)$ from Proposition \ref{prop:non-unique-regime}. \purple{Under the assumption $\Graph_n \loc \Tree_d$} 
we have the following limits:

\begin{enumeratei}

\item For $\be \in [0,\be_c(0))$ and any $k \in [q]$ we have $\mu_{n,k}^{\be, 0} \lwc \mu_{\free}^{\be, 0}$.

\item \red{Additionally assume that $\{\Graph_n\}_{n \in \N}$ are uniformly edge-expanders.} 
For $\be > \be_c(0)$ and $k \in [q]$ we have $\mu_{n,k}^{\be,0} \lwc \mu_{k}^{\be, 0}$. 

\end{enumeratei}
\end{thm}

\subsection{Proof outline and organization of the paper}
A key ingredient for the proofs of the main results in \cite{BD, MMS}, the Ising case, is the monotonicity property of edge-correlations, which in turn enables one to argue that, for $B=0$, the sums of the edge-correlations at the root for any Ising Gibbs measure on the limiting tree that is not a mixture of the plus and the minus Ising measure must be strictly smaller than that of the plus (and hence minus) Ising measure (see \cite[Lemma 3.2]{MMS} and \cite[Lemma 3.6]{BD} for precise formulations for the regular and the general cases, respectively). This observation essentially allowed them to restrict the set of all possible local weak limit points to those that are mixtures of the plus and minus Ising measures.

In the absence of such monotonicity property for Potts measures,  we {had to devise an alternate route to 
reach such conclusion. To this end},  we work with the Edwards-Sokal coupling of our $q$-Potts measures and their \abbr{rcm} counterparts.  {Whereas the} \abbr{RCM}-s have the relevant monotonicity properties,  {this came at the cost of having non-quasi-local Hamiltonian.  Thereby,} 
much of the standard theory of Gibbs states does not apply {for the \abbr{rcm}-s,  with the} 
definition and analysis of random cluster Gibbs (\abbr{RCG}) states {posing} technical difficulties (cf.~\cite{BBCK00}).  {Indeed,}  instead of analyzing relevant properties of the \abbr{RCG} states (e.g.~two point connectivity; {as} done for $q=2$ in \cite{BD, MMS}),  {we derive} in Section \ref{sec:subseq-limit}
several salient features of the $t$-dimensional spin and bond marginals of local weak limit points \red{(i.e.~probability measure induced by the spin and the bond configurations in a ball of radius $t$ around the root of the limiting tree $\Tree_d$)} of Edward-Sokal measures on $\Graph_n$.  In particular,  it is shown 
\red{in Lemma \ref{lem:limit-point-properties}(c)-(d),}
that any such $t$-dimensional bond marginal must be a mixture of \abbr{RCM} on $\Tree_d(t)$ with boundary conditions {that are} sandwiched,  in the sense of stochastic ordering,  between the marginals of the free and wired \abbr{rcm}-s on the infinite tree $\Tree_d$,  and moreover any such $t$-dimensional spin marginal {follows specific} conditional distribution given the bond configurations.  
We believe {that} these results 
\red{which apply for $q$-\abbr{rcm} at any $q \in [1,\infty)$, }
should admit broader usage.

\red{Moreover,  we show in Remark \ref{rmk:split-ES} and Lemma \ref{lem:potts-rc-fw} that
the} $t$-dimensional spin marginals of any local weak limit point of Edward-Sokal measures 
on $\Graph_n$,  follow specific conditional distributions given the bond configurations.  
\red{Combining these facts with our results about the bond marginals and} \purple{Proposition} \ref{asmp:main},  
we prove Theorem \ref{thm:main-1} in Section \ref{sec:pf-uncond-limit}. 


The key to the proof of Theorem \ref{thm:main-2}(i) is a coupling of $\mu_{n,k}^{\be,0}$ and $\mu_{n,k'}^{\be, 0}$, $k \ne k' \in [q]$, so that the number of disagreements between the spin configurations induced by $\mu_{n,k}^{\be,0}$ and $\mu_{n,k'}^{\be,0}$ is negligible.
{This is done by a construction of independent uniform colors for the connected 
components of the \abbr{rcm},  which may be of independent interest (cf.~Lemma \ref{lem:sim-unif} and
the proof of \eqref{eq:small-disagreement}).}

Our proof, in Section \ref{subsec-pure-states}, of
Theorem \ref{thm:main-2}(ii),  uses the assumed edge-expansion property,  to argue that $\sK_n$ is well approximated by the dominant color in large neighborhood of a uniformly chosen vertex in
$\Graph_n$.
By Theorem \ref{thm:main-1},  this is in turn further approximated by the {dominant color}
in a large neighborhood of the root in {$\Tree_d$} (under the wired Potts measure),  and thereby  
we complete the proof upon noting that under $\mu_k^{\be,0}$ the latter dominant color
is $k$ (see Lemma \ref{lem:pott-inifinite-prop}).  While at a high level the proof of Theorem \ref{thm:main-2}(ii)
has {a similar} outline as the proof of \cite[Theorem 2.4 II]{MMS},  extra care is needed here,  since 
unlike $q=2$,  one can no longer reduce the identity of the dominant color to the sign of a single variable.

{The most delicate argument of this work is the proof of Theorem \ref{thm:crit-1}.  The challenge here is 
that being along the critical line,  for which there is no Ising analog (due to the difference in the order of phase transitions:~second and first order phase transitions for $q=2$ and $q \ge 3$, respectively),  the limit is no longer an extremal
Potts (or \abbr{rcm}) measure on $\Tree_d$.  Our approach to handle such a scenario,  thus proceeds via 
the following three} key steps.  First, Section \ref{subsec-rcm-messages} shows
that any limit point outside $\cM_{\bond}^{\be_c(B),B}$ yields 
`messages' to the root {(see} Definition \ref{dfn:rcm-msg2}), of Bethe functional value
smaller than $\max\{\Phi(\nu_{\free}^{\be_c(B), B}), \Phi(\nu_{1}^{\be_c(B), B})\}$. For $B>0$, if a local weak limit is not supported on $\cM_{\bond}^{\be_c(B),B}$, the same can be shown to hold when the messages on {$\Tree_d$} are replaced by their local analogs on large neighborhoods {of a uniformly 
chosen vertex in} $\Graph_n$.  For $d \in 2\N$,  this,  together with a randomization over the choices of vertices in $\Graph_n$ and the concentration of sample mean to the population mean,  allow us 
in Section \ref{subsec-d-even}
to procure a sequence of subgraphs $\{\Graph_{n}'\}_{n \in \N}$ so that $\Graph_n' \loc \Tree_d$ with asymptotic Potts free energy density 
exceeding the \abbr{rhs} of \eqref{eq:free-energy-limit-assume},
in contradiction with \cite[Theorem 1]{DMSS}. 

Appendix \ref{sec:app} provides several properties of the infinite volume 
Potts and the Bethe fixed points, {on which we rely in these proofs.}

\section{Existence and properties of limit points}\label{sec:subseq-limit}

{Hereafter we fix $\be,B \ge 0$ and integer $q \ge 2$. For a finite graph $\Graph=(V,E)$, let 
$\Graph^\star$ be as in Definition \ref{dfn:ghost-vertex}, with 
$\be^\star_{(i,j)} :=\be$ for $(i,j) \in E$ and
$\be^\star_{(i,j)} :=B$ for $(i,j) \in E^\star \setminus E$. We consider 
the Potts measure 
\beq\label{dfn:Potts-star}
\mu_{\Graph^\star}^{\be,B}(\ul \sigma) = \frac{1}{Z_{\Graph}(\be,B)} \exp \Big\{ \sum_{(i,j) \in E^\star} 
\be^\star_{(i,j)} \delta_{\sigma_i, \sigma_j} \Big\} \cdot 
\delta_{\sigma_{v^\star},1}, \qquad \ul \sigma \in [q]^{V^\star} \,,
\eeq
proceeding to define the corresponding Edwards-Sokal measure.}
\begin{dfn}[Edwards-Sokal measure on a finite graph]\label{dfn:es_dfn}
{Fix an integer $q \ge 2$. For a finite graph $\Graph$, 
$\Graph^\star$ of Definition \ref{dfn:ghost-vertex}, and $\be,B \ge 0$, set}
$p_e$ as in \eqref{eq:dfn:p-e} {and $\delta_e(\ul{\sigma}):=\delta_{\sigma_i,\sigma_j}$ for 
$e=(i,j) \in E^\star$.
The} Edwards-Sokal probability measure on the joint
spin $\ul{\sigma} \in [q]^{V^\star}$ and 
bond $\ul \eta \in \{0,1\}^{E^\star}$ configuration, is given by
\beq
\varpi_{\Graph^\star}^{\be,B}(\ul{\sigma},\ul{\eta}) \propto \prod_{e \in E^\star} 
\left[ (1-p_e)({1-\eta_e}) + p_e {\eta_e} \delta_e(\ul{\sigma})  
\right] \cdot  \delta_{\sigma_{v^\star},1} \,.
\eeq
\end{dfn}
{Our next lemma, whose elementary proof is omitted, states that for  
integer any $q \ge 2$, such Edwards-Sokal measure
} gives a useful coupling between 
$\mu^{\be,B}_{\Graph^\star}$ and the {\abbr{rcm} of Definition \ref{dfn:rcm_dfn}.}
\begin{lem}\label{lem:coupling-es}
Fix an integer $q \ge 2$, a finite graph $\Graph=(V,E)$, {and} ${B},\be \ge 0$. 
\begin{enumeratei}
\item The marginal of $\varpi_{\Graph^\star}^{\be,B}(\cdot, \cdot)$ in the spin variable 
$\ul \sigma$ is $\mu_{\Graph^\star}^{\be,B}(\cdot)$, whereas the marginal in the bond variable $\ul \eta$ is {the \abbr{rcm}} $\varphi_\Graph^{\be,B}(\cdot)$.

\item For any $\ul \eta \in \{0,1\}^{E^\star}$, with {induced connected components} 
$\sC(\ul \eta)=(C_1,C_2,\ldots,C_k)$, for some $k \ge 1$, the conditional distribution $\varpi_{\Graph^\star}^{\be,B}(\cdot| \ul \eta)$ is {such that}  
the same spin $\sigma'$ is assigned to all vertices in each connected component $C$, independently of 
{all} other components. If $v^\star \notin C$, then the law of $\sigma'$ is uniform on $[q]$, whereas $\sigma'=1$ if $v^\star \in C$.
\end{enumeratei}
\end{lem}

We use throughout $\mu(f)$ or $\mu[f]$ for the expectation of  $\R$-valued function $f$
with respect to a probability measure $\mu$. 
{Further, suppressing the dependence of our} Potts, random cluster, and Edwards-Sokal measures
{on the fixed integer} $q \ge 2$, 
we write for brevity $\varpi_n^{\be, B}:= \varpi_{\Graph^\star_n}^{\be, B}$ 
and $\varphi_n^{\be, B}:= \varphi_{\Graph_n}^{\be, B}$. {As 
the marginal of $\mu_{\Graph^\star}^{\be, B}$ on $[q]^V$ is} 
$\mu^{\be, B}_{\Graph}$, hereafter $\mu_n^{\be, B}$ stands also for
$\mu_{\Graph_n^\star}^{\be, B}$. 
 
The following {immediate} corollary of Lemma \ref{lem:coupling-es} {is crucial in our} proof of Theorem \ref{thm:main-1}. 
\begin{cor}\label{cor:fin-rc-potts}
{For any integer $q \ge 2$, finite graph $\Graph=(V,E)$, and ${B},\be \ge 0$,}
we have that
\beq\label{eq:corr-potts-rc}
\mu_{\Graph}^{\be,B} \left(\sigma_i=\sigma_j\right)  =\Big(1 - \f1q\Big)\varphi_\Graph^{\be,B}\left(i \leftrightarrow j \right) +\f1q, \qquad \forall (i,j) \in E,
\eeq
where $\{ i' \leftrightarrow j'\}$ denotes the event that $i'$ and $j'$ are in the same connected component
{of $\Graph^\star$}, or equivalently, {that} there is an open path {in $\Graph^\star$}
connecting $i'$ and $j'$.
\end{cor}

\begin{proof}
{In view of Lemma \ref{lem:coupling-es}(i), for any $(i,j) \in E^\star$,} 
\[
\mu_{\Graph^\star}^{\be,B} (\sigma_i=\sigma_j) = 
\varpi_{\Graph^\star}^{\be,B} \left[{\bf 1}(\sigma_i=\sigma_j) \right] =
\varphi_\Graph^{\be,B}\left[ 
\varpi_{\Graph^\star}^{\be,B} \left[{\bf 1}(\sigma_i=\sigma_j)  \big| \ul \eta\right] \right].
\] 
Now \eqref{eq:corr-potts-rc} follows from Lemma \ref{lem:coupling-es}(ii)
{(recall also 
that $\mu_{\Graph}^{\be,B}$ is the $[q]^V$-marginal of 
$\mu_{\Graph^\star}^{\be,B}$).}
\end{proof}

{Hereafter, we denote by $\cP(S)$ the set of probability measures on a 
Polish space $S$. Further, adopting the notation of Definition \ref{dfn:inf-vol-RC}, 
for $\dagger \in \{o,\star\}$ let
$\bS^\dagger_t:=[q]^{V(\Tree_d^\dagger(t))}$ and 
$\B_t^\dagger := \{0,1\}^{E(\Tree_d^\dagger(t))}$ denote
the domains of the spins, and the bonds, respectively, in Edwards-Sokal measures on $\Tree^\dagger_d(t)$.
Using $\Xi_t^\dagger := \{0,1\}^{E(\sfK^\dagger(\partial \Tree_d(t)))}$ and
${\wh \B_t^\dagger :=} \{0,1\}^{E(\wh\Tree_d^\dagger(t))}$ for such 
domains of boundary edge bonds, and of the bonds on the tree with its 
edge boundary, respectively, we shall view each bond assignment $\wh{\ul \eta} \in \wh{\B}_t^\dagger$ on $\wh\Tree_d^\dagger(t)$
as the pair $\wh{\ul \eta}=(\ul \eta, {\bm \eta})$ of assignments $\ul \eta \in \B_t^\dagger$ for bonds
in $\Tree_d^\dagger(t)$ and ${\bm \eta} \in \Xi_t^\dagger$ for the boundary edge bonds.}

{We shall} show that the Edwards-Sokal measures on $\Graph_n$
(and hence also the corresponding Potts and random cluster measures), 
admit local weak limits along sub sequences satisfying {certain} key properties. 
{To this end, we proceed to define} the spaces in which such local weak limits {must be}
{(where to lighten our notation we suppress the dependence of these spaces 
on $d$, $q$ and $\be,B \ge 0$).}

\begin{dfn}[Mixtures of Edwards-Sokal and \abbr{rcm}-s {on $\Tree_d$}]\label{dfn:mixture-rc}
Fixing $t \in \N$, let $\gC_\star:=\gC_\star(t)$ denote
the {collection} of all partitions $\sC={\{}C_\star, C_1, C_2, \ldots, C_k{\}}$ of {the labeled set} $\partial \Tree_d(t) \cup \{v^\star\}$ where $v^\star \in C_\star$ and
$\gC_o:=\gC_o(t)$ be all such partitions for which $C_\star=\{v^\star\}$.
Recalling Definition \ref{dfn:inf-vol-RC} 
of $\sfK^\dagger(\partial \Tree_d(t))$ and $\wh\Tree_d^\dagger(t)$, $\dagger \in \{o, \star\}$, 
{we identify each} 
$\sC \in \gC_\dagger$ {with}
the subgraph of $\sfK^\dagger(\partial \Tree_d(t))$ {having} 
edges between $i,j \in \partial \Tree_d(t) \cup \{v^\star\}$ {\em if and only if} $i$ and $j$ belong to {the same
block of} $\sC$.
{Alternatively, each $\sC \in \gC_\dagger$ corresponds to the boundary edge 
assignment ${\bm \eta}(\sC)$ such that 
${\bm \eta}_e = 1$ if and only if the edge $e$ is within a block of $\sC$.
Now, let $\varpi^{\be, B, t}_\sC$} be 
the Edwards-Sokal measure on $\wh\Tree_d^\dagger(t)$ conditioned {on 
${\bm \eta}={\bm \eta}(\sC)$ (namely, on open} bonds in $\sC$ 
and {closed bonds in} $\sfK^\dagger(\partial \Tree_d(t))\setminus \sC$). 
Thus, for {$\sC \in \gC_\dagger$,} $\ul \sigma \in \bS_t^\dagger$, and 
$\ul \eta \in \B_t^\dagger$,
\begin{equation}\label{dfn:ES-wh-cR}
\varpi_{\sC}^{\be, B, t}(\ul \sigma, \ul \eta)  \propto \prod_{e \in E(\Tree^\dagger_d(t))} \left[ 
(1-p_e) (1-\eta_e) +
p_e \eta_e \delta_e(\ul \sigma) \right] \prod_{j=1}^k \prod_{u,v \in C_j} \delta_{\sigma_u, \sigma_v}  \prod_{v \in C_\star} \delta_{\sigma_v,1},
\end{equation}
{restricting for $\dagger=o$ to $B=0$ and eliminating then the (trivial) product over $C_\star$.
We further view $\varpi_{\sC}^{\be, B, t}(\ul \sigma,\ul \eta)$ of \eqref{dfn:ES-wh-cR} also 
as the probability of $(\ul \sigma, \wh{\ul \eta})$ for $\wh{\ul \eta}=(\ul \eta,{\bm \eta}(\sC))$.}
Now define 
\begin{align}\label{eq:def-wh-cR}
{\wt\cR}_\dagger(t)&:=\Big\{\varpi: \varpi= \sum_{\sC \in \gC_\dagger} \rho (\sC)  \varpi_{\sC}^{\be, B, t}  \text{ for some } \rho \in \cP(\gC_\dagger)\Big\},  \qquad \dagger \in \{o,\star\}.
 \end{align} 
{That is, $\wt \cR_\dagger(t)$ denotes} the collection of measures on
{spins and bonds of $\Tree_d^\dagger(t)$ induced by} mixtures of Edwards-Sokal measures 
on $\wh \Tree_d^\dagger(t)$ conditioned to have the edge boundary {bonds ${\bm \eta}(\sC)$.}
{Likewise, the \abbr{rcm} conditioned to such edge boundary, gives the measure on $\B_t^\dagger$}
\beq\label{eq:varphi-sC1}
\varphi^{\be, B, t}_\sC(\ul \eta):= \varphi^{\be, B}_{\wh\Tree_d^\dagger (t)}( (\ul \eta, {\bm \eta}) \,|\,
{\bm \eta} = {\bm \eta} (\sC) 
), 
\eeq
with the corresponding space of mixtures 
\beq\label{eq:def-cR}
 \cR_\dagger(t):=\Big\{\varphi: \varphi= \sum_{\sC \in \gC_\dagger}  \rho (\sC) 
 \varphi^{\be, B, t}_\sC  \text{ for some } \rho \in \cP(\gC_\dagger) \Big\} \,, \qquad 
 \dagger \in \{o,\star\} \,,
 \eeq
 {where for $\dagger=o$ we restrict \eqref{eq:def-wh-cR}-\eqref{eq:def-cR} to $B=0$
 (and eliminate in this case the irrelevant $v^\star$).}
{Further viewing each $\varphi^{\be, B, t}_\sC(\ul \eta)$ also as the 
probability of $\wh{\ul \eta}=(\ul \eta,{\bm \eta}(\sC))$ makes
$\cR_\dagger(t)$ a subset of $\cP(\wh \B_t^\dagger)$, and the 
mixture coefficients are then uniquely determined by $\rho (\sC) = \varphi({\bm \eta} = {\bm \eta}(\sC))$.}

\bigskip
{We next show} that the $t$-marginals of both the free and wired \abbr{rcm}-s on $\Tree_d$ 
reside in the spaces {$\cR_\dagger(t)$ of Definition \ref{dfn:mixture-rc}.}

\begin{lem}\label{lem:rc-fw-limit}
For $\be, B \ge 0$ and $\ddagger \in \{\free, \wired\}$ the \abbr{rcm} 
$\varphi_{\ddagger}^{\be, B}$ on $\Tree_d^\star$ exists {with marginals}
$\varphi_{\ddagger}^{\be, B,t} \in \cR_\star(t)$ {and}
$\varphi_{\ddagger}^{\be, 0,t} \in \cR_o(t)$ for any $t \in \N$.
\end{lem}

\begin{proof} {Fix $\be,B \ge 0$. For the existence of $\varphi_\ddagger^{\be, 0}$ on $\Tree_d$}
see \cite[Theorem 10.67]{G-RC}. {More generally, fix $t \in \N$ and note} that, by definition, 
\beq\label{eq:varphi-t-marg}
\varphi_{\wired, t}^{\be, B}(\cdot)= \varphi_{\wh\Tree_d^\star(t+1)}^{\be, B}(\cdot|\eta_e=1, e \in \sfK^\star(\partial \Tree_d(t+1)) \cup \sfK^\star(\partial \Tree_d(t)) \cup E(\partial\Tree_d(t), \partial \Tree_d(t+1))), 
\eeq
where for two sets of vertices $S$ and $S'$ the notation $E(S,S')$ denotes the collection of edges between $S$ and $S'$. As $\red{\varphi_{\wh \Tree_d^\star(t+1)}^{\be, B}}$ is a strictly positive measure satisfying \abbr{FKG} inequality (see \cite[Proposition 4.3]{DMS}), by \cite[Theorem 2.24]{G-RC} it is monotonic. This observation together with the definition of $\varphi_{\wired, t+1}^{\be, B}$ and \eqref{eq:varphi-t-marg}  entails that $\varphi_{\wired, t}^{\be, B}[f] \ge \varphi_{\wired, t+1}^{\be, B}[f]$ for any increasing function $f$ on 
$\B^\star_s$ with $s \le t$. {As each $\B^\star_s$ is in the linear span of increasing indicator functions,
the} existence of 
$\lim_{t \to \infty}\varphi_{\wired, t}^{\be, B}[f]$ for {any $s \in \N$ and every increasing} $f$
on $\B^\star_s$ implies the existence of the limit $\varphi_{\wired}^{\be, B}$ {(in the sense of 
weak convergence). The existence of $\varphi_{\free}^{\be, B}$ follows similarly, since  
$\varphi_{\free, t}^{\be, B}[f] \le \varphi_{\free, t+1}^{\be, B}[f]$ for any 
increasing function $f$ on $\B^\star_s$ with $s \le t$.}

Turning to prove that $\varphi_\ddagger^{\be,B, t} \in \cR_\star(t)$, we start by showing that
\beq\label{clm:marginal}
\varphi \in \cR_\star(t+1) \qquad \Longrightarrow \qquad {\varphi^t} \in \cR_\star(t) \,,
\eeq 
{where $\varphi^t$ denotes the marginal of $\varphi$ on bonds of $\Tree_d^\star(t)$.} 
Indeed, setting $\wt \Xi_t:=\{0,1\}^{E(\Tree_d^\star(t+1))\setminus E(\Tree_d^\star(t))}$, 
any pair $\ul \eta \in \wt \Xi_t$ and $\sC \in \gC_\star(t+1)$ induces
a partition $\sC'(\ul \eta, \sC) \in \gC_\star(t)$ {and further determines the difference between 
the number of connected components of $\wh{\Tree}^\star_d(t)$ with boundary $\sC'$ and the number of 
connected components of $\wh{\Tree}^\star_d(t+1)$ with boundary $\sC$. Hence, by definition} 
\beq\label{eq:t+1-to-t-pre}
\varphi_{\sC}^{\be,B,t+1}(\ul \eta^0 \,|\, \ul\eta)= \varphi^{\be, B,t}_{\sC'(\ul \eta, \sC)}(\ul \eta^0) \,,
\qquad \forall \ul \eta^0 \in \B^\star_t \,.
\eeq  
Thus, if $\varphi \in \cR_\star(t+1)$, then for some $\rho \in \cP(\gC_\star(t+1))$ and 
any $\ul \eta^0 \in \B^\star_t$,
\beq\label{eq:t+1-to-t}
\varphi (\ul\eta^0) =  \sum_{\sC \in\gC_\star(t+1)} \sum_{\ul \eta \in \wt \Xi_t}
\rho(\sC) \varphi_{\sC}^{\be,B,t+1} (\ul \eta) \varphi^{\be, B, t}_{\sC'(\ul \eta, \sC)}(\ul\eta^0)  
= \sum_{\sC_0 \in \gC_\star(t)} \rho'(\sC_0) \varphi^{\be, B, t}_{\sC_0}(\ul\eta^0) \,,
\eeq
where $\rho'(\sC_0)$ denotes the sum of $\rho(\sC) \varphi_{\sC}^{\be,B,t+1} (\ul \eta)$ 
over all pairs $(\ul \eta, \sC)$ for which $\sC'(\ul \eta, \sC)=\sC_0$. {Having thus established
\eqref{clm:marginal}, recall that $\varphi_{\ddagger, s}^{\be, B} \in \cR_\star(s)$ for any $s \in \N$. 
Hence, by iteratively applying \eqref{clm:marginal}
we deduce that} the $t$-dimensional marginal $\varphi_{\ddagger, s}^{\be, B, t}$ 
of $\varphi_{\ddagger, s}^{\be, B}$ is in {the compact set $\cR_\star(t)$ for any $s \ge t$.} 
Since $\varphi_{\ddagger,s}^{\be, B, t} \Rightarrow \varphi_{\ddagger}^{\be, B, t}$ as $s \to \infty$, 
{we conclude} that $\varphi_{\ddagger}^{\be, B, t} \in \cR_\star(t)$.

{For $B=0$ we have in \eqref{eq:dfn:p-e} that $p_e=0$ for any $e \in E^\star \setminus E$ touching 
$v^\star$. Thus, then} $v^\star$ is an isolated vertex {and in particular $C_\star=\{v^\star\}$.  
So, in this case we can \abbr{wlog} replace $\gC_\star(t)$ by $\gC_o(t)$ throughout the 
preceding proof, while also
replacing our spin and bond domains $\Tree_d^\star(t)$ and $\sfK^\star(\partial \Tree_d(t))$ by 
$\Tree_d (t)$ and $\sfK (\partial \Tree_d(t))$, respectively, to conclude that}  
$\varphi_\ddagger^{\be, 0, t} \in \cR_o(t)$.
\end{proof}

{Having a lattice of partitions $\gC_\dagger(t)$ induces a stochastic ordering on
the
\abbr{rcm} mixtures in $\cR_\dagger(t)$.}

\begin{dfn}[Stochastic ordering on $\cR_\dagger(t)$]\label{dfn:st-order}
Fix $t \in \N$ and $\dagger \in \{o, \star\}$. Recall from Definition \ref{dfn:mixture-rc} 
{our embedding of $\gC_\dagger(t)$ inside $\Xi^\dagger_t$ via $\sC \mapsto {\bm \eta}(\sC)$
and the one-to-one mapping it induces between $\cR_\dagger(t)$ and $\cP(\gC_\dagger(t))$, by 
matching each
\[
\varphi (\wh{\ul \eta}) = \sum_{\sC \in \gC_\dagger(t)} \rho(\sC) \varphi_{\sC}^{\be,B,t} (\wh{\ul \eta}) \,,
\]
with the distribution $\rho(\sC)=\varphi({\bm \eta} = {\bm \eta}(\sC))$ of a  
{\em random edge boundary} ${\bm \eta}$ supported on $\gC_\dagger(t) \subset \Xi^\dagger_t$.}
An edge boundary ${\bm \eta}$ of law $\rho$ is stochastically dominated by ${\bm \eta}'$ of law
$\rho'$, denoted by ${\bm \eta} \st {\bm \eta'}$ (or by $\rho \st \rho'$), if $\rho (f) \le \rho'(f)$ for every function $f$
which is non-decreasing with respect to the {usual} 
partial ordering of $\Xi^\dagger_t$. Equivalently, ${\bm \eta} \st {\bm \eta}'$ if and only if there is a coupling such that ${\bm \eta} \le {\bm \eta}'$ in the sense of partial ordering on $\Xi^\dagger_t$. {For random edge boundaries in $\gC_\dagger(t)$ this is further}
equivalent to a coupling {with such partial order on the} induced partitions $\sC({\bm \eta}) \le \sC({\bm \eta}')$
{(i.e.~where $\sC({\bm \eta})$ is a refinement of $\sC({\bm \eta}')$).}
This notion extends to a stochastic ordering on $\cR_\dagger(t)$ by saying that 
$\varphi \in \cR_\dagger(t)$ is stochastically dominated by $\varphi'
\in \cR_\dagger(t)$ (denoted by $\varphi \preccurlyeq \varphi'$), if and only if 
${\bm \eta} \st {\bm \eta'}$ for the corresponding random edge boundaries.
\end{dfn}

{The proof of Theorem \ref{thm:main-1} crucially relies on our next lemma, which reduces 
the question whether two stochastically ordered measures in $\cR_\dagger(s+1)$ are equal, 
to the evaluation of the corresponding expectations for} a {\em single functional} $\gF_s$ (defined below).
\begin{lem}\label{lem:ordering-gF}
Fix $s \in \N$ and $\dagger \in \{\star, o\}$ {(with $B=0$ if $\dagger=o$).}\\
Suppose $\varphi, \wt \varphi \in \cR_\dagger(s+1)$ are such that $\varphi \st \wt \varphi$. 
\begin{enumerate}
\item[(i)] Then, $\varphi^s \st \wt \varphi^s$ for the $s$-dimensional marginals
of $\varphi$ and $\wt \varphi$. 
\item[(ii)] {If also $\varphi \ne \wt \varphi$,
then $\varphi(\gF_s) < \wt\varphi(\gF_s)$} for 
\begin{equation}\label{eq:def-gF}
\gF_s:= \sum_{i \in \partial \Tree_d(s)} \sum_{j \in \partial i} {\bf 1}(i \bij j),
\end{equation}
{where $\partial i$ denotes the neighborhood of $i$ in $\Tree_d(s+1)$ and 
$\{i \bij j\}$ denotes the event that there exists an open path in $\wh \Tree^\dagger_d(s+1)$
connecting $i$ and $j$.} 
\end{enumerate}
\end{lem}

\begin{proof}
(i). Since $\varphi \st \wt \varphi $ there is a coupling such that {$\sC({\bm \eta}) \le \sC(\wt {\bm \eta})$
for the partitions induced by} the corresponding random edge boundaries
${\bm \eta}$ and $\wt{\bm \eta}$. 
Let ${\bm \eta}_s$ and $\wt{\bm \eta}_s$ denote the random vectors in $\wt \Xi_s$ 
distributed according to the corresponding marginal of $\varphi^{\be, B, s+1}_{\sC({\bm \eta})}$ 
and $\varphi^{\be, B, s+1}_{\sC(\wt{\bm \eta})}$, respectively. Since the \abbr{rcm} $\varphi_{\wh\Tree_d^\star(s+1)}^{\be, B}$ is monotonic (see \cite[Proposition 4.3]{DMS} and \cite[Theorem 2.24]{G-RC}), we have
that 
\beq\label{eq:monotonic-0}
\varphi_{\sC({\bm \eta})}^{\be, B, s+1}(f) \le \varphi_{\sC(\wt{\bm \eta})}^{\be, B, s+1}(f)
\eeq
for any increasing function {$f$ on $\wh \B^\star_{s+1}$.} This implies in turn that 
{$\varphi (g) \le \wt \varphi (g)$ for any function $g$ of $({\bm \eta}_s,{\bm \eta})$
which is increasing on $\wt \Xi_s \times \Xi^\star_{s+1}$. That is,} 
$({\bm \eta}_s, {\bm \eta}) \st (\wt{\bm \eta}_s, \wt{\bm \eta})$. Recall that any given pair 
$\bar{\bm \eta}_s \in \wt \Xi_s$ and $\bar{\bm \eta} \in \Xi^\star_{s+1}$ induces 
an edge boundary $\bar{\bm \eta}' \in \Xi^\star_{s}$. Furthermore, the map $(\bar{\bm \eta}_s, \bar{\bm \eta}) \mapsto \bar{\bm \eta}'$ is increasing, {hence also} ${\bm \eta}' \st  \wt{\bm \eta}'$ for
the random edge boundaries ${\bm \eta}'$ and $\wt{\bm \eta}'$ induced by
the pairs $({\bm \eta}_s, {\bm \eta})$ and $(\wt{\bm \eta}_s, \wt{\bm \eta})$, respectively.  {Recall 
\eqref{eq:t+1-to-t-pre}} that ${\bm \eta}'$ and $\wt {\bm \eta}'$ are {precisely the} 
edge boundaries {of} $\varphi^s$ and $\wt \varphi^s$, respectively. Thus,
$\varphi^s \st \wt \varphi^s$ {as claimed}. In case $B=0$ and $\dagger=o$ {we can and shall 
follow the same argument, while eliminating throughout the isolated ghost vertex $v^\star$.} 

\noindent
(ii). Upon considering \eqref{eq:monotonic-0} for the increasing events $\{i \bij j \}$, we deduce 
that $\varphi(i \bij j) \le \wt \varphi (i \bij j)$ for any $(i,j) \in \Tree^\star_d(s+1)$, hence also 
$\varphi (\gF_s) \le \wt \varphi(\gF_s)$. Now since $\varphi \ne \wt \varphi$ and $\gC_\star(s+1)$
is a finite set, there exist $u, u' \in \sfK^\star (\partial \Tree_d(s+1))$ 
such that under the monotone coupling ${\bm \eta} \le \wt{\bm \eta}$,  
with positive probability $\wt {\bm \eta}_{(u,u')} = 1$ while $\bm \eta_{(u,u')}=0$.
In particular, $u$ and $u'$ are in different blocks of $\sC({\bm \eta})$ and we may
further assume that ${\bm \eta}_{(u,v^\star)}=0$ (or else, exchange $u$ with $u'$).  
It thus suffices to show that any such boundary edges result with 
\beq\label{eq:monotonic-st}
\varphi_{\sC({\bm \eta})}^{\be, B, s+1}(u \bij w) < \varphi_{\sC(\wt{\bm \eta})}^{\be, B, s+1}(u \bij w) \,,
\eeq
where $w \in \partial \Tree_d(s)$ is the parent of $u$. 
Turning to prove \eqref{eq:monotonic-st}, consider the event 
\[
\cA := \{u \stackrel{{\rm in}}{\bij} v^\star\} \cup
\{w \stackrel{{\rm in}}{\bij} \partial \Tree_d(s+1) \setminus \{u'\} \} \,,
\]
where $\{U \stackrel{{\rm in}}{\bij} U'\}$ denotes the event of an {\em open path within} 
$\Tree_d^\star(s+1)$ 
between some vertex of $U$ and some vertex of $U'$. With $\{u \bij w\} \cap \cA$ an increasing event,
for which \eqref{eq:monotonic-0} holds, we arrive at \eqref{eq:monotonic-st} upon showing that
\beq\label{eq:monotonic-stt}
\varphi_{\sC(\wt{\bm \eta})}^{\be, B, s+1}(\{u \bij w\} \cap \cA^c) > 0  \quad  \text{and} \quad
\varphi_{\sC({\bm \eta})}^{\be, B, s+1}(\{u \bij w\} \cap \cA^c) =0 \,.
\eeq
To this end, the event $\cA^c$ implies that apart from boundary edges, $u$ is an 
isolated vertex and there is no open path between $w$ and $\partial \Tree_d(s+1)$
except possibly between $w$ and $u'$. Thus,
\[
\{u \bij w \} \cap \cA^c = \Big( \{u \stackrel{{\rm ex}}{\bij} u' \stackrel{{\rm in}}{\bij} w\}
\cup \{u \stackrel{{\rm ex}}{\bij} v^\star  \stackrel{{\rm in}}{\bij} w \} \Big) \cap \cA^c \,,
\]
where $\{u \stackrel{{\rm ex}}{\bij} u'\}$ denotes the existence of an open path with
{\em only boundary edges} between $u$ and $u'$. The latter event amounts to having
both $u$ and $u'$ in the same block of $\sC$, so our 
observation that $u$ is neither in the same block of $\sC({\bm \eta})$ as $u'$ nor in that of
$v^\star$,
yields the right part of \eqref{eq:monotonic-stt}. Further, with $\wt {\bm \eta}_{(u,u')} = 1$, 
the left side of \eqref{eq:monotonic-stt} holds as soon as
\beq\label{eq:monotonic-1}
\varphi_{\sC(\wt{\bm \eta})}^{\be, B, s+1} (\{w \stackrel{{\rm in}}{\bij} u'\} \cap \cA^c) > 0 \,. 
\eeq
For $B>0$ we get \eqref{eq:monotonic-1} 
by opening only the edges $(w,v^\star)$ and $(u',v^\star)$ of $\Tree_d^\star(s+1)$. In case $B=0$ 
we follow the same reasoning in $\Tree_d(s+1)$ (i.e.~without the isolated $v^\star$), 
now satisfying the event $\{w \stackrel{{\rm in}}{\bij} u'\} \cap \cA^c$ by opening only 
the edges of $\Tree_d(s+1)$ which lie on the unique path from $w$ to $u'$, 
the probability of which is strictly positive whenever $\be>0$. Finally, in case
of $B=\be=0$ there is nothing to prove, for then $\cR_\star(s+1)$ is a singleton
(as all bonds of $\Tree^\star_d(s+1)$ are closed, regardless of
the boundary edge configuration $\sC$).
\end{proof}

{Thanks to our stochastic ordering of $\cR_\dagger(t)$,} all
marginals of local weak limit points of $\{\varphi_n^{\be, B}\}$ are
supported on the set of measures $\cS_\dagger(t)$ and $\cQ_\star(t)$ (the latter for $B >0$), 
as defined and shown below. Namely, such local weak limits 
are all sandwiched between the free and the wired \abbr{rcm}-s on $\Tree_d$, and additionally for $B>0$  their random edge boundary at level $t$ have edges {\em only} between $\partial \Tree_d(t)$ and the ghost $v^\star$. 
\begin{dfn}\label{dfn:btwn-f-w}
Fix $t \in \N$. For $\ddagger \in \{\free, \wired\}$, let $\varphi_\ddagger^{\be, B,t}$ be the 
marginal of $\varphi_\ddagger^{\be, B}$ on {$\B_t^\star$} and define 
 \[
\cS_\dagger(t):= \Big\{ \varphi \in \cR_\dagger(t):  
\varphi^{\be, B, t}_{\free} \st \varphi \st \varphi^{\be, B, t}_{\wired} \Big\}\,,
\]
where $\dagger=\star$ if $B>0$ and $\dagger=o$ if $B=0$. Let $\gD_\star = \gD_\star(t) \subset \gC_\star(t)$ be the collection of all partitions $\sC=\{C_\star, C_1, C_2, \ldots, C_k\}$ of the labeled set $\partial \Tree_d(t) \cup \{v^\star\}$ with $v^\star \in C_\star$ such that $|C_i|=1$ for all $i \ge 1$. 
Define
\[
\cQ_\star(t):=\Big\{\varphi \in \cS_\star(t): \varphi= \sum_{\sC \in \gD_\star}  \rho (\sC) 
 \varphi^{\be, B, t}_\sC  \text{ for some } \rho \in \cP(\gD_\star) \Big\}.
\]
\end{dfn}

As promised, we now show that the finite dimensional marginals of sub-sequential 
local weak limits of {our} Edwards-Sokal measures {must be supported on the} 
spaces {from Definition \ref{dfn:mixture-rc} and those for our \abbr{rcm}-s
must further be supported on the spaces from Definition \ref{dfn:btwn-f-w}.} 
\end{dfn}
\begin{lem}\label{lem:limit-point-properties}
Fix $\be, B \ge 0$, integer $q \ge 2$ and $\Graph_n \loc \Tree_d$.  Set $\dagger=\star$ if $B>0$
and $\dagger=o$ if $B=0$. 
\begin{enumeratea}
\item The sequence of measures $\{\varpi_n^{\be,B}\}$ admits sub-sequential local weak limits,
as do its marginals $\{\varphi_n^{\be,B}\}$ and $\{\mu_n^{\be,B}\}$, in bond and spin variables, 
respectively.
\item Any $t$-dimensional marginal of a local weak limit point $\gm_{\spin,\bond}$  
of $\{\varpi_n^{\be,B}\}$ satisfies $ \gm_{\spin,\bond}^t \in \cP(\wt \cR_\dagger(t))$. 
\item Any $t$-dimensional marginal of a local weak limit point $\gm_{\bond}$ of 
$\{\varphi_n^{\be,B}\}$ satisfies $\gm_{\bond}^t \in \cP(\cS_\dagger (t))$. 

\item For $B >0$, a $t$-dimensional marginal of 
such 
local weak limit point $\gm_{\bond}$ 
satisfies $\gm_{\bond}^t \in \cP(\cQ_\star (t))$.

\end{enumeratea}
\end{lem}

\begin{proof}
(a).
As $\Graph_n \loc \Tree_d$, for any $t \in \N$ and $\vep >0$, there exists an $n_0(\vep)$ such that
\[
\inf_{n \ge n_0(\vep)} {\sf P}_{\varpi_n^{\be,B}}^t(I_n)(\cP_d^{\star,t}) \ge 1-\vep,
\]
where $\cP_d^{\star,t}$ and ${\sf P}_{\varpi_n^{\be,B}}^t(\cdot)$ are as in Definitions \ref{dfn:prob-space} and \ref{dfn:lwc-measure}, respectively. Since $\cP_d^{\star, t}$ is compact,
{from} Prokhorov's theorem it follows that for fixed $t \in \N$ the random probability 
measures $\{{\sf P}_{\varpi_n^{\be,B}}^t(I_n)\}$ admit a sub-sequential limit $\bar\gm_t$
{which is a Borel probability measure on $\cP_d^{\star,t}$.} Upon extracting successive subsequences, from the definition of ${\sf P}_{\varpi_n^{\be,B}}^t(I_n)$ it further follows that for every $t \in \N$ the marginal of $\bar{\gm}_{t+1}$ on $\cP_{d}^{\star,t}$ is $\bar\gm_{t}$.
{Choosing} the diagonal subsequence and using Kolmogorov's extension theorem we {thus}
establish the existence of a sub-sequential local weak limit {point $\bar{\gm}$} of 
$\{\varpi_n^{\be,B}\}$, {with $\bar{\gm}$ a probability measure on $\cP_d^\star$.
As both} $\varphi_n^{\be,B}$ and $\mu_n^{\be, B}$ are marginals of $\varpi_n^{\be,B}$,
the existence of {sub-sequential} local weak limits of $\{\varphi_n^{\be,B}\}$ and 
$\{\mu_n^{\be, B}\}$ {is} now immediate.

\noindent
(b). Fixing $t \in \N$ and $(\ul\sigma^0, \ul \eta^0) \in {\bS^\star_t \times\B_t^\star}$ we have that {for any} $\ell \in [n]$,
\begin{align}\label{eq:ES-decompose}
 & \,{\sf P}_{\varpi_n^{\be,B}}^t(\ell)  (\Tree_d^\star(t), \ul\sigma^0, \ul \eta^0) \notag\\
& = \, 
 {\bm 1} (\sfB_\ell(t) \cong \Tree_d(t)) \, {q_{\ell,n}^t (\ul \sigma^0
 )}
 \prod_{e \in E(\sfB_\ell^\star(t))} \left[ (1-p_e)({1-\eta^0_e}) + p_e {\eta_e^0} \delta_{e} (\ul \sigma^0) \right] \cdot {\bf 1} (\sigma^0_{v^\star}=1) \,, 
\end{align}
where
\begin{align}\label{eq:p-k-ES}
{q_{\ell,n}^t} (\ul \sigma^0)  & := \f{1}{Z_n} 
\sum_{\ul \eta |_{E_n^\star \setminus E(\sfB_\ell^\star(t))} }
\sum_{\ul \sigma : \ul \sigma |_{V(\sfB_\ell^\star(t))} = \ul \sigma^0} 
\prod_{e \in E_n^\star \setminus
E(\sfB_\ell^\star(t))} \left[ (1-p_e)({1-\eta_e}) + p_e {\eta_e} \delta_{e} (\ul \sigma) \right],
\end{align}
depends {only} on $\ul \sigma^0|_{\partial \sfB_\ell(t)}$ and $Z_n$ is 
some normalizing constant. 
{Now, split the outer} sum over $\ul \eta$ in \eqref{eq:p-k-ES} according to which vertices in 
$\partial  \sfB_\ell(t) \cup \{v^\star\}$ are connected among themselves via an open path
induced by $\ul \eta$. {Whenever $\sfB_\ell(t) \cong \Tree_d(t)$,  
this amounts to splitting the sum over $\ul \eta$ per partitions} 
$\sC
\in \gC_\star(t)$ 
{of $\partial \Tree_d(t) \cup \{v^\star\}$ with $v^\star \in C_\star$. Further, $\sigma_{v^\star}=1$
and any open path $\gamma$ in $\ul \eta$ induces the multiplicative factor 
$\prod_{e \in \gamma} \delta_e(\ul \sigma)$ in \eqref{eq:p-k-ES}. Hence, if $\ul \eta$ 
induces the partition $\sC$, then} the sum over the spins in the \abbr{RHS} of \eqref{eq:p-k-ES} be zero 
unless $\gI_\sC(\ul\sigma^0):=\prod_{j=1}^k \prod_{u,v \in C_j} \delta_{\sigma_u^0, \sigma_v^0} \cdot \prod_{v \in C_\star} \delta_{\sigma_v^0,1}=1$. Therefore,
\[
q_{\ell,n}^t (\ul\sigma^0)= \sum_{\sC \in \gC_\star(t)} \vartheta_{\ell,n} (\sC) \gI_\sC(\ul\sigma^0),
\]
for some nonnegative constants $\{\vartheta_{\ell,n}(\sC)\}_{\sC \in \gC_\star(t)}$. 
{Plugging this into \eqref{eq:ES-decompose} results with}
\begin{align}\label{eq:es-pre-limit}
 {\sf P}^t_{\varpi_n^{\be,B}}(\ell)(\Tree_d^\star(t), \ul \sigma^0, \ul \eta^0) = {\bm 1} (\sfB_\ell(t) \cong \Tree_d(t)) 
\wt{\sf P}_n^{\be, B, t}(\ell)  (\ul \sigma^0, \ul \eta^0) \,,
\end{align}
{where for each $\ell \in [n]$,
\begin{align}
\label{eq:wt-sfP}
\wt{\sf P}_n^{\be, B, t}(\ell) (\cdot,\cdot) :=  \sum_{\sC \in \gC_\star(t)}
\varrho_{\ell,n}(\sC) \varpi_{\sC}^{\be, B,t} (\cdot,\cdot) \in \wt \cR_\star(t) \,,
\end{align}
namely, each $\varrho_{\ell,n} \in \cP(\gC_\star(t))$ (and 
when $\sfB_\ell(t) \not\cong \Tree_d(t)$ we can choose an arbitrary $\varrho_{\ell,n}$ of this form).
Recall that the set of probability measures on the finite dimensional simplex
$\cP(\gC_\star(t) \times \bS^\star_t \times \B_t^\star)$ is compact under weak convergence.} 
In particular,  
the probability measures $\wt{\sf P}_n^{\be, B, t}(I_n)$ on the 
{compact subset $\wt \cR_\star(t)$ of $\cP(\bS^\star_t \times \B_t^\star)$}
admit sub-sequential limits in the topology of weak convergence
on $\wt \cR_\star(t)$. As $\Graph_n \loc \Tree_d$, 
by \eqref{eq:es-pre-limit} the 
sub-sequential limit points of $\{{\sf P}^t_{\varpi_n^{\be,B}}(I_n)\}_{n \in \N}$ 
{coincide with those of}  
$\{\wt{\sf P}_n^{\be, B, t}(I_n)\}_{n \in \N}$. {In particular, the $t$-marginal $\gm^t_{\spin,\bond}$ 
of any local weak limit point of $\{\varpi_n^{\be,B}\}$ must also be in $\cP(\wt \cR_\star(t))$.}

{As noted at the end of the proof of Lemma \ref{lem:rc-fw-limit}, if $B=0$
we can omit throughout the isolated vertex $v^\star$ and the preceding argument 
then yields that $\gm^t_{\spin,\bond} \in \cP(\wt \cR_o(t))$.}

\noindent
(c). {First recall that for any $t \in \N$ and $\sC \in \gC_\star(t)$, the bond-marginal
of $\varpi^{\be, B, t}_\sC$ is $\varphi_\sC^{\be, B, t}$.}
Thus, by \eqref{eq:wt-sfP}, for any $\ell \in [n]$, 
\beq\label{eq:wt-sfP1}
\wt{\sf P}_{n, \bond}^{\be,B,t}(\ell)(\cdot):= \sum_{\ul \sigma} \wt {\sf P}_{n}^{\be,B,t}(\ell)(\ul\sigma, 
\cdot)= \sum_{\sC \in \gC_\star(t)} \varrho_{\ell,n}(\sC) \varphi_\sC^{\be, B, t}(\cdot) \in 
\cR_\star(t) \,.
\eeq
{From \eqref{eq:es-pre-limit}, the bond-marginal of 
${\sf P}^t_{\varpi_n^{\be,B}}(\ell)(\Tree_d^\star(t), \cdot, \cdot)$ is 
${\bm 1} (\sfB_\ell(t) \cong \Tree_d(t)) \wt{\sf P}_{n,\bond}^{\be, B, t}(\ell)$. Hence, applying the preceding 
reasoning for limit points of the probability measures $\wt{\sf P}_{n,\bond}^{\be, B, t}(I_n)$
on the compact $\cR_\star(t)$, we conclude that the $t$-marginal $\gm^t_{\bond}$ of any limit point
of the \abbr{rcm}-s $\{\varphi_n^{\be,B}\}$ must be in $\cP(\cR_\star(t))$.
Further, by Definitions \ref{dfn:inf-vol-RC} and \ref{dfn:st-order} we know that 
\[
 \cR_\star(s) = \{\varphi \in \cR_\star(s): 
 \varphi_{\free,s}^{\be, B} \st
 \varphi \st \varphi_{\wired, s}^{\be, B}\} \,, \qquad \forall s \in \N \,.
\]
Next, for $s>t$ let $\cR_\star^t(s)$ be the collection of all $t$-dimensional 
marginals of measures from $\cR_\star(s)$, noting that since $\gm^t_{\bond}$ is 
also the $t$-marginal of $\gm^{s}_{\bond}$ for any $s>t$, necessarily 
\beq\label{eq:sandwich-1}
\gm_{\bond}^t\Big(\bigcap_{s>t} \cR_\star^t(s) \Big)=1 \,, \qquad \forall t \in \N \,.
\eeq
Now, any $\wt \varphi \in \cR_\star^t(s)$ is the $t$-marginal of some $\varphi \in \cR_\star(s)$
and in particular $\varphi_{\free,s}^{\be,B} \st \varphi \st \varphi_{\wired,s}^{\be, B}$. 
In view of \eqref{clm:marginal} and Lemma \ref{lem:ordering-gF}(i), it thus follows that 
\[
\cR_\star^t(s) \subset \{ \wt \varphi \in \cR_\star(t) : 
\varphi^{\be, B, t}_{\free,s} \st \wt\varphi \st \varphi^{\be, B, t}_{\wired,s} \} \,.
\]}
Since $\varphi_{\ddagger, s}^{\be, B, t} \Ra \varphi_{\ddagger}^{\be, B, t}$ when $s \to \infty$,
we thus deduce that 
\[
\bigcap_{s>t} \cR_\star^t(s) \subset \{ \wt \varphi \in \cR_\star(t) : 
\varphi^{\be, B, t}_{\free} \st \wt\varphi \st \varphi^{\be, B, t}_{\wired} \} = \cS_\star(t) 
\]
(see Definition \ref{dfn:btwn-f-w}), which together with \eqref{eq:sandwich-1} implies
that $\gm_{\bond}^t \in \cP(\cS_\star(t))$, as claimed.

As explained at the end of the proof of  part (b), for $B=0$ we omit the isolated $v^\star$ 
and re-run the same argument on the original tree $\Tree_d$, to arrive at
$\gm^t_{\bond} \in \cP(\cS_o(t))$. 

\noindent
(d). Fix any $t \in \N$, $u,v \in \partial \Tree_d(t)$, and $\sC_0\in \gC_\star(t)$ such that $u$ and $v$ belong to a same block of $\sC_0$ {\em not} containing $v^\star$. Fix a $\varphi \in \cP(\{0,1\}^{E(\Tree_d^\star)})$ such that its $(t+s)$-dimensional marginal $\varphi^{t+s} \in \cR_\star(t+s)$ for all integer $s\ge 0$. Recall from Definition \ref{dfn:mixture-rc} that $\varphi^{t+s}$ can be viewed as a measure on the bonds of $\wh \Tree_d^\star(t+s)$. For $s \ge 0$ we let $\Omega_{u,v}(s)$ 
be the event of bond configurations for which the cluster of $u$, induced by the open bonds in $\wh \Tree_d^\star(t+s)\setminus \Tree_d^\star(t)$, contains $v$ but does {\em not} contain the ghost vertex $v^\star$.
Due to consistency of the marginals of $\varphi$ and that \eqref{clm:marginal} holds we therefore have    
\beq\label{eq:bdry-edge-1}
\varphi^t({\bm \eta} = {\bm \eta}(\sC_0)) \le \varphi^{t+s}(\Omega_{u,v}(s)) \text{ for all integer } s \ge 0. 
\eeq
We will show that for any $B >0$ and $s \in \N$
\beq\label{eq:bdry-edge-2}
\varphi^{t+s}(\Omega_{u,v}(s)) \le q^2 \exp(-2Bs). 
\eeq
This together with \eqref{eq:sandwich-1} and \eqref{eq:bdry-edge-1} yield the desired result. Turning to prove \eqref{eq:bdry-edge-2}, since $\varphi^{t+s} \in \cR_\star(t+s)$ there exists some probability measure $\rho \in \cP(\gC_\star(t+s))$ such that
\[
\varphi^{t+s}(\Omega_{u,v}(s)) = \sum_{\sC \in \gC_\star(t+s)} \varphi^{\be,B, t+s}_\sC(\Omega_{u,v}(s)) \rho(\sC). 
\] 
Hence, it suffices to prove \eqref{eq:bdry-edge-2} with $\varphi^{t+s}$ replaced by $\varphi^{\be,B,t+s}_\sC$ with $\sC \ni C \not\owns v^\star$ such that $u',v' \in C$ for some descendants of $u$ and $v$, respectively.
To this end, fix such a $\sC$ and split the bond configurations of $\Tree_d^\star(t+s)$ into three pieces:~$\ul \eta^{(1)} \in \B_t^\star$, $\ul \eta^{(2)} \in \B_{t+s}\setminus \B_t$, and $\ul \eta^{(3)} \in \B_{t+s}^\star \setminus (\B_t^\star \cup \B_{t+s})$. 
Notice that the event $\Omega_{u,v}(s)$ does not depend on $\ul \eta^{(1)}$. Further, {upon ordering
$\partial \Tree_d(t+s)$ in some fixed, non-random manner}
observe that $(\ul \eta^{(2)}, \ul \eta^{(3)}) \in \Omega_{u,v}(s)$ {requires that the following}
events hold:
\begin{enumeratei}
\item There exist open paths $P_u$ and $P_v$, determined by $\ul \eta^{(2)}$, from $u$ and $v$ 
to $\wt u$ and $\wt v \in \partial \Tree_d(t+s)$, respectively, where $\wt u$ and $\wt v$ are 
descendants of $u$ and $v$, respectively. {When} there are multiple {such}
$\wt u$ and $\wt v$ {we fix the first of these as our} $\wt u$ and $\wt v$, respectively.
\item We must have $\ul \eta^{(3,1)} \equiv 0$, where $\ul \eta^{(3)} = (\ul \eta^{(3,1)}, \ul \eta^{(3,2)})$, $\ul \eta^{(3,1)}$ is the bond configuration for the {collection of} edges between $V(P_u \cup P_v)$ 
and $v^\star$, and $\ul \eta^{(3,2)}$ is the rest of the bond configurations. 

\end{enumeratei}

Set $\wh \Omega_{u,v}(s):=\{\ul \eta^{(2)}: (\ul \eta^{(2)}, \ul \eta^{(3)}) \in \Omega_{u,v}(s) \text{ for some } \ul \eta^{(3)}\}$. Note that given any $\ul \eta^{(1)}$, $\ul \eta^{(2)} \in \wh \Omega_{u,v}(s)$, and $\ul \eta^{(3,2)}$ the only indeterminacy in the cluster structure through the choice of $\ul \eta^{(3,1)}$ is whether the clusters of $u$ and $v$ contains the ghost $v^\star$ or not. Therefore, using that $|V(P_u \cup P_v)|=2s$ and that $1-p_e=e^{-B}$ for $e \in \Tree_d^\star(t+s)\setminus \Tree_d(t+s)$, we have
\beq\label{eq:bdry-edge-3}
\varphi^{\be, B, t+s}(\ul \eta^{(3,1)} \equiv 0 |\ul \eta^{(1)}, \ul \eta^{(2)}, \ul \eta^{(3,2)}) \le q^2 \exp(-2Bs).
\eeq
Finally using observations (i) and (ii) above, and taking an average over $(\ul \eta^{(1)}, \ul \eta^{(2)}, \ul \eta^{(3,2)})$ we obtain \eqref{eq:bdry-edge-2} with $\varphi^{t+s}$ replaced by $\varphi^{\be,B,t+s}_\sC$ with $\sC$ as above. This completes the proof. 
 \end{proof}

\begin{rmk}\label{rmk:split-ES}
For $\wh{\ul \eta} \in \wh \B_t^\dagger$ let 
$\gu_\dagger^t(\wh{\ul \eta})$ denote the probability measure on {spins of} 
$\Tree_d^\dagger(t)$, with the same spin value across all vertices in each connected component and 
\abbr{iid} uniform on $[q]$ spin values across different connected components, {except for} the component containing $v^\star$ whose spin is $\sigma_{v^\star}=1$.
{For $\sC \in \gC_\dagger(t)$
let $\gu_\dagger^t(\varphi^{\be,B,t}_\sC)$ denote the 
\emph{random}} probability measure induced by $\gu_\dagger^t(\wh{\ul \eta})$ for random
$\wh{\ul \eta}$ whose law is $\varphi_{\sC}^{\be, B, t}$ {(restricting to $B=0$ when 
$\dagger=o$), and for any $\varphi \in \cR_\dagger(t) \subset \cP(\wh \B_t^\dagger)$} let
\[
\theta_\dagger^t(\varphi) := \sum_{\sC \in \gC_\dagger(t)}  {\varphi({\bm \eta}(\sC))}
\E_{\bond}\big[\gu_\dagger^t (\varphi_\sC^{\be, B, t}) \big],
\]
where $\E_{\bond}[\cdot]$ denotes the expectation with respect to the bond variables. 
\red{Setting $\varphi_\sC^{\be, B, t} (\wh {\ul \eta})=0$ whenever $\wh {\ul \eta} = (\ul \eta, {\bm \eta}')$ is such that ${\bm \eta}(\sC) \ne {\bm \eta}'$,
by} Lemma \ref{lem:coupling-es}(ii),
\begin{align}\label{eq:es-gu-rcm}
\varpi_\sC^{\be, B, t} (\cdot, \wh {\ul \eta}) &= 
{
\varphi_\sC^{\be, B, t} (\wh {\ul \eta}) \gu_\dagger^t (\wh {\ul \eta}) (\cdot)
\,, \quad \forall \sC \in \gC_\dagger(t), \, \wh{\ul \eta} \in \wh \B_t^\dagger},
\end{align}
and therefore
\begin{align}
\label{eq:th-varphi}
\sum_{\ul{\wh \eta}} \varpi_\sC^{\be, B, t} (\cdot, \wh{\ul \eta}) &=
\E_{\bond}\big[\gu_\dagger^t (\varphi_\sC^{\be, B, t}) \big] (\cdot)
\,, \qquad \forall \sC \in \gC_\dagger(t) \,.
\end{align}
Thus, with $\varpi_{\spin}$ and $\varpi_{\bond}$ denoting the spin and bond marginals of 
$\varpi \in \cP(\bS^\dagger_t \times \wh \B_t^\dagger)$, 
\beq\label{eq:def-wt-cR}
{ \varpi \in \wt\cR_\dagger(t) \quad \Longrightarrow \quad }  
\varpi =  \varpi_{\bond} (\wh {\ul \eta}) \gu_\dagger^t (\wh {\ul \eta}), \quad
\varpi_{\spin}= \theta_\dagger^t(\varpi_{\bond}) \quad \text{and}  
\quad \varpi_{\bond} \in \cR_\dagger(t), \quad \dagger \in \{o, \star\}.
\eeq
\end{rmk}

{Utilizing Remark \ref{rmk:split-ES}, we now} relate marginals of the 
free and wired Potts and \abbr{rcm}-s on $\Tree_d$.
\begin{lem}\label{lem:potts-rc-fw}
{Let $\dagger=\star$ if $B>0$ and $\dagger=o$ if $B=0$. Then,} for any $t \in \N$:
\begin{enumeratea}
\item The $t$-dimensional marginal of $\varpi =(\varpi_{\spin}, \varpi_{\bond})\in \wt \cR_\dagger(t+1)$
is $\varpi^t=(\varpi_{\spin}^t, \varpi_{\bond}^t) \in \wt \cR_\dagger(t)$. 
\item The Potts and \abbr{rcm} marginals are such that  
$\mu^{\be, B, t}_\ddagger= \theta_\dagger^t(\varphi^{\be,B,t}_\ddagger)$,
for $\ddagger \in \{\free, \wired\}$.

\item For $\ddagger \in \{\free, \wired\}$ we have
\beq\label{eq:inf-vol-rc-potts-e}
\mu_{\ddagger}^{\be, B} (\sigma_i = \sigma_j) = \Big(1 -\f1q\Big) \varphi_{\ddagger}^{\be,B}(
i \leftrightarrow j) +\f1q, \qquad (i,j) \in E(\Tree_d).
\eeq
\end{enumeratea}
\end{lem}

\begin{proof} 
(a). With $\varpi^t$ denoting the marginal of $\varpi$
to the sub-tree $\Tree^\dagger_d(t)$, clearly one may compute the marginals 
$(\varpi^t)_{\spin}$ and $(\varpi^t)_{\bond}$ on the spins and bonds of
$\Tree^\dagger_d(t)$, respectively, also in the opposite order. 
That is, as stated, $(\varpi^t)_{\spin}=(\varpi_{\spin})^t$ and $(\varpi^t)_{\bond}=(\varpi_{\bond})^t$.
For our remaining claim, that $\varpi^t \in \wt \cR_\dagger(t)$ for any $\varpi \in \wt \cR_\dagger(t+1)$,
it suffices to consider
only the extremal measures of  
$\wt \cR_\dagger(t+1)$. Namely, to fix $\sC \in \gC_\dagger(t+1)$ and
consider the $t$-marginal of $\varpi^{\be,B}_\sC$ (with boundary specification 
$\bm \eta=\bm \eta(\sC)$ on $\sfK^\dagger(\partial \Tree_d(t+1))$). Doing this, we denote by 
$(\ul \sigma^0,\ul \eta^0)$ the configuration on $\Tree_d^\dagger(t)$, so upon splitting the 
spins and bonds of $\Tree_d^\dagger(t+1)$ as $(\ul \sigma^0,\ul \sigma)$ and $(\ul \eta^0,\ul \eta)$, 
our task is to show that for some $\rho \in \cP(\gC_\dagger(t))$, 
\[
\sum_{\ul \eta} \sum_{\ul \sigma} \varpi^{\be,B,t+1}_{\sC} (\ul \sigma^0,\ul \sigma, \ul \eta^0, \ul \eta) 
= \sum_{\sC^0 \in \gC_\dagger(t)} \rho (\sC^0) \varpi_{\sC^0}^{\be,B,t} (\ul \sigma^0,\ul \eta^0) \,,
\qquad \forall \ul \sigma^0, \ul \eta^0 \,.
\]
Applying \eqref{eq:es-gu-rcm} on both sides of this identity and then utilizing
\eqref{eq:t+1-to-t-pre}, it remains only to show that
\beq\label{eq:gu-star-marginal}
\sum_{\ul \sigma} \gu^{t+1}_\dagger (\ul \eta^0,\ul \eta, \sC) (\ul \sigma^0,\ul \sigma) 
= \gu^t_\dagger (\ul \eta^0,\sC'(\ul \eta,\sC) ) (\ul \sigma^0)
\,, 
\qquad \forall (\ul \eta^0,\ul \eta), 
\eeq
where $\sC'=\sC'(\ul \eta,\sC) \in \gC_\dagger(t)$ denotes the partition corresponding
to open components 
of $\partial \Tree_d(t)$ induced by $\ul \eta$ and the given partition $\sC$ of the 
boundary of $\wh \Tree^\dagger_d(t+1)$. Next, recall 
that under $\gu_\dagger^{t+1}(\cdot)$ all spins in each $C_i \in \sC$ must take
the same value, if $C_i$ has an open path to some
$v \in {\partial \Tree_d (t)}$ then the spin of $C_i$ must match $\sigma^0_v$
while otherwise the spin of $C_i$ be chosen uniformly in $[q]$, independently 
of everything else. Further, $(\ul \eta,\sC)$ determines the 
connectivity between ${(\partial \Tree_d (t+1))^\star}$ and $(\partial \Tree_d (t))^\star$
and $\sC'(\ul \eta,\sC)$ is the partition induced on 
$(\partial \Tree_d (t))^\star$ by our requirement that $\sigma^0_v=\sigma^0_u$
whenever $u$ and $v$ connect, via $\ul \eta$, to the same block of $\sC$. 
Combining these observations leads to \eqref{eq:gu-star-marginal}.

\noindent
(b). Fix $\ddagger=\{\free,\wired\}$, $B>0$ and $s \in \N$. By Lemma \ref{lem:coupling-es}(i),
Definition \ref{dfn:mixture-rc}, and Definition \ref{dfn:inf-vol-RC}, we have that 
$\varphi_{\ddagger,s}^{\be, B}$ are the bond marginals
of the Edwards-Sokal measure $\varpi_{\sC}^{\be,B}$ on $\wh{\Tree}^\star_d(s)$
with ${\bm \eta}(\sC) \equiv \gf(\ddagger)$, where $\gf(\cdot)$ is as in \eqref{eq:f-w-rc-t}. We further claim that then 
$\mu_{\ddagger,s}^{\be,B}$ of Definition \ref{dfn:potts-bounary-B-0} is such that 
\beq\label{eq:rcm-to-Potts}
\mu_{\ddagger,s}^{\be,B} (\cdot) = \sum_{\ul \eta} \varpi_{\sC}^{\be,B} (\cdot,\ul \eta)  = 
\theta_\star^s(\varphi_{\ddagger,s}^{\be, B}) \,.
\eeq
Indeed, the \abbr{lhs} of \eqref{eq:rcm-to-Potts} follows from Lemma \ref{lem:coupling-es}(i) 
upon observing that by Lemma \ref{lem:coupling-es}(ii), 
here $\varpi_{\sC}^{\be,B} =\varpi^{\be,B}_{\Tree^\star_d(s)}$ with  
\abbr{iid} uniform spins at $\partial \Tree_d(t)$ if $\ddagger=\free$ 
and all such spins set to $1$ if $\ddagger=\wired$ (whereas
the \abbr{rhs} of \eqref{eq:rcm-to-Potts} is merely an application of Remark \ref{rmk:split-ES}).
Next, in view of part (a) we further deduce from \eqref{eq:rcm-to-Potts} that
$\mu_{\ddagger, s}^{\be, B, t} = \theta_\star^t(\varphi_{\ddagger,s}^{\be, B, t})$ for any $s>t$. Since 
$\mu_{\ddagger, s}^{\be, B, t} \Ra \mu_{\ddagger}^{\be, B, t}$ and $\varphi_{\ddagger, s}^{\be, B, t} \Ra \varphi_{\ddagger}^{\be, B, t}$ when $s \to \infty$, we arrive at 
$\mu_{\ddagger}^{\be, B,t}= \theta_\star^t(\varphi_{\ddagger}^{\be, B,t})$.

For $B=0$ we omit the isolated $v^\star$ and arrive by the preceding reasoning at
$\mu_{\free}^{\be, 0,t} = \theta_o^t(\varphi_{\free}^{\be, 0,t})$. We likewise get 
that $\mu_{\wired}^{\be, 0,t} = \theta_o^t(\varphi_{\wired}^{\be, 0,t})$ upon 
verifying that the \abbr{lhs} of \eqref{eq:rcm-to-Potts} holds in that case, which 
in view of \eqref{eq:mu-wired} and both parts of Lemma \ref{lem:coupling-es}
amounts to checking that  
\[
\frac{1}{q} \sum_{k=1}^q \mu^{\be,0}_{k,s} (\cdot) =  
\mu^{\be,0}_{\Tree_d(s)}(\cdot \, | \, \ul \sigma |_{\partial \Tree_d(s)} \equiv \sigma' \in [q] ) \,.
\]
By symmetry the value of $\mu^{\be,0}_{\Tree_d(s)}( \ul \sigma |_{\partial \Tree_d(s)} \equiv k )$ is
independent of $k$ and as $\mu^{\be,0}_{k,s}$ equals $\mu^{\be,0}_{\Tree_d(s)}$ conditional to 
the latter event (see Definition \ref{dfn:potts-bounary-B-0}), the preceding identity follows.

\noindent
(c). For any $(i,j)\in E(\Tree_d)$ there exists some $t \in \N$ such that $(i,j) \in E(\Tree_d(t))$ 
{in which case} 
$\mu_\ddagger^{\be, B}(\sigma_i=\sigma_j)=\mu_\ddagger^{\be,B,t}(\sigma_i=\sigma_j)$ and $\varphi_\ddagger^{\be,B}(i \bij j) = \varphi_\ddagger^{\be, B,t}(i \bij j)$. Thus, by part (b) and arguing 
as in the proof of Corollary \ref{cor:fin-rc-potts} 
{(now conditionally to $\bm \eta \equiv \gf(\ddagger)$ on $\sfK^\dagger(\partial \Tree_d(t))$),}
{results with} \eqref{eq:inf-vol-rc-potts-e}. 
\end{proof}

\section{{Bethe replica symmetry:} proof of Theorem \ref{thm:main-1}}\label{sec:pf-uncond-limit}
 
{We set as usual $\dagger=\star$ if $B>0$ and $\dagger=o$ if $B=0$. Throughout this 
section also $\ddagger = \free$ if 
$(\be, B) \in \sfR_{\free}$ and $\ddagger=\wired$ if $(\be, B) \in \sfR_1$.}
Our next lemma (the proof of which is postponed to Section \ref{sec:pf-correlation-limit}),
identifies the limit of the internal energy per vertex for
$\mu_n^{\be,B}$,  {which is hence the} key to pinning down the local weak limits {as those 
supported on one specific degenerate measure}.
\begin{lem}\label{lem:correlation-limit}
\purple{Under the assumption $\Graph_n \loc \Tree_d$} {(with} $\ddagger = \free$ if 
$(\be, B) \in \sfR_{\free}$; $\ddagger=\wired$ if $(\be, B) \in \sfR_1${) we have,}
\beq\label{eq:correlation-limit}
\lim_{n \to \infty} \f{1}{n} \sum_{i=1}^n \sum_{j \in \partial i} \mu_n^{\be, B}(\sigma_i = \sigma_j) = 
 \sum_{i \in \partial o} \mu_{\ddagger}^{\be, B} (\sigma_o=\sigma_i)\,.
\eeq
\end{lem}

We proceed with the following three steps of our proof of Theorem \ref{thm:main-1}.

\bigskip
\noindent
{{\bf Step I.} From \abbr{rcm} limits to Potts limits.\\
Suppose $\varphi_{n_k}^{\be,B} \lwc \varphi  \in \cM_{\bond}^{\be,B}$, {where}
$\varphi = {\bm \al} \varphi_{\wired}^{\be,B}+(1-{\bm \al}) \varphi_{\free}^{\be,B}$ for
some $[0,1]$-valued ${\bm \al}$. The local weak limit points
$\gm_{\spin,\bond}$ of the Edwards-Sokal measures $\varpi_{n_k}^{\be,B}$ 
(which exist in view of Lemma \ref{lem:limit-point-properties}(a)), must all
have the same} bond marginal with $\gm_{\bond} (\{ \varphi \})=1$.  Recall 
from Lemma \ref{lem:limit-point-properties}(b) that each $\gm_{\spin,\bond}^t$ is supported on 
$\wt \cR_\dagger(t)$. 
{In view of 
\eqref{eq:def-wt-cR}, this} in turn implies that $\gm_{\spin}^t(\{\theta_\dagger^t(\varphi)\})=1$.  
Applying Lemma \ref{lem:potts-rc-fw}(b) we deduce that 
$\gm_{\spin}^t(\{ {{\bm \al} \mu_{\wired}^{\be,B,t}+(1-{\bm \al}) \mu_{\free}^{\be,B,t}} \})=1$
for all $t \in \N$.  That is, 
$\mu_{n_k}^{\be,B} \lwc {{\bm \al} \mu_{\wired}^{\be,B}+(1-{\bm \al}) \mu_{\free}^{\be,B}}$.
In particular, in the setting of {Theorems \ref{thm:main-1} and \ref{thm:crit-1}}
it suffices to prove only the
stated local weak convergence for the \abbr{rcm}-s $\varphi_n^{\be,B}$.

\bigskip
\noindent
{{\bf Step II.} The uniqueness regime $(\be,B) \in [0,\infty)^{\red{2}} \setminus \sfR_{\ne}$.\\
Here} $\mu_{\wired}^{\be, B}= \mu_{\free}^{\be, B}$
(see Proposition \ref{prop:non-unique-regime}(i)),  
{hence $\varphi_{\wired}^{\be,B}(i \leftrightarrow j) = \varphi_{\free}^{\be,B}(i \leftrightarrow j)$ 
at any edge $(i,j)$ of $\Tree_d$
(see Lemma \ref{lem:potts-rc-fw}(c)), and in particular $\varphi_{\wired}^{\be,B,s+1}(\gF_s) =
\varphi_{\free}^{\be,B,s+1}(\gF_s)$ for all $s \in \N$ (recall \eqref{eq:def-gF}). From Lemma \ref{lem:ordering-gF} it then follows}
that $\varphi_{\free}^{\be,B,t}=\varphi_{\wired}^{\be,B,t}$ {and hence
$\cS_\dagger(t)=\{\varphi_{\free}^{\be,B,t}\}$ for all $t \in \N$ (see Definition \ref{dfn:btwn-f-w}).
In view of Lemma \ref{lem:limit-point-properties}(c), we conclude that}
$\varphi_n^{\be,B} \lwc \varphi_{\free}^{\be,B}$, as claimed.

\bigskip
\noindent
{{\bf Step III.} Non-uniqueness, with $(\be,B) \in \sfR_{\free} \cup \sfR_1$.\\
Consider the uniformly bounded functionals
\begin{equation*}
\gH^{\spin}_t := \f{1}{|\partial \sfB_o(t)|}\sum_{i \in \partial \sfB_o(t)} \f{1}{\Delta_i} \sum_{j \in \partial i} {\bf 1}(\sigma_i = \sigma_j) \quad \text{ and } \quad
\gH^{\bond}_t := \f{1}{|\partial \sfB_o(t)|}\sum_{i \in \partial \sfB_o(t)} \f{1}{\Delta_i} \sum_{j \in \partial i} {\bf 1}(i \bij j) 
\end{equation*}
(see Definitions \ref{dfn:ghost-vertex} and \ref{dfn:inf-vol-RC}), 
where the 
event $\{i \bij j\}$ denotes that there is an open path between $i$ and $j$ in $\wh \sfB_o^\dagger(t+1)$. 
With $\Tree_d$ vertex transitive, $\mu_{\ddagger}^{\be, B}$ is translation invariant, 
so for any $t \in \N$,
\beq\label{eq:trans}
\frac{1}{d} \sum_{j \in \partial o} \mu_{\ddagger}^{\be, B}(\sigma_o=\sigma_j)= \f{1}{|\partial \tTree_d(t)|}\sum_{i \in \partial \tTree_d(t)} \frac{1}{\Delta_i} \sum_{j \in \partial i} 
\mu^{\be,B}_{\ddagger}(\sigma_i=\sigma_j) = \Big(1-\frac{1}{q} \Big) 
\varphi^{\be, B,t+1}_{\ddagger} (\gH^{\bond}_t) + \frac{1}{q} \,,
\eeq
where the right identity is due to Lemma \ref{lem:potts-rc-fw}(c). Consider
the $[0,1]$-valued functions on $V(\Graph_n)$,
\[
\red{\tfa}_n^t(i):= \sum_{k \in \partial \sfB_i(t)} 
\frac{{\bf 1}(\sfB_k(2t) \cong \Tree_d(2t))}{|\partial \sfB_k(t)|}
= \left\{\begin{array}{ll}
0 & \mbox{ if } \sfB_i(t) \ncong \Tree_d(t),\\
1 & \mbox{ if } \sfB_i(3t) \cong \Tree_d(3t).
\end{array}
\right.
\]
For $\Graph_n \loc \Tree_d$ we have that $\E_n [\ga_n^t(I_n)] \to 1$, hence it follows from
\eqref{eq:correlation-limit} and \eqref{eq:trans} that 
\begin{align}
\Big(1-\frac{1}{q} \Big) 
\varphi^{\be, B,t+1}_{\ddagger} (\gH^{\bond}_t) + \frac{1}{q} &=
\lim_{n \to \infty} \f{1}{d n} \sum_{i=1}^n \sum_{j \in \partial i} \ga_n^t (i) \varpi_n^{\be, B}(\sigma_i = \sigma_j) \nonumber \\
&= 
\lim_{n \to \infty} \E_n \Big[
{\bf 1}(\sfB_{I_n}(2t) \cong \Tree_d(2t))
{\sf P}^{t+1}_{\varpi_n^{\be,B}} (I_n) (\gH_t^{\spin}) \Big]
\label{eq:cor-limit1}
\end{align}
(in the last identity we used that $k \in \partial \sfB_i(t)$ if and only if $i \in \partial \sfB_k(t)$). 
From \eqref{eq:cor-limit1} we deduce that for any weak limit point $\gm_{\spin,\bond}$ 
of $\{\varpi_n^{\be,B}\}$
\beq\label{eq:extreme-pre-limit}
\Big(1-\frac{1}{q} \Big) 
\varphi^{\be, B,t+1}_{\ddagger} (\gH^{\bond}_t) + \frac{1}{q} 
= \gm_{\spin,\bond}^{t+1} \big( \varpi(\gH^{\spin}_t) \big), \qquad \forall t \in \N \,.
\eeq
Applying Corollary \ref{cor:fin-rc-potts} on the graph $\wh \Tree^\dagger_d(t+1)$ with boundary 
edges per $\bm \eta(\sC)$, results with
\[
\varpi_{\sC}^{\be,B,t+1}(\gH^{\spin}_t) = 
\Big(1-\frac{1}{q} \Big) 
\varphi^{\be, B,t+1}_{\sC} (\gH^{\bond}_t) + \frac{1}{q}, \qquad \forall t \in \N \,, \quad \forall \sC \in 
\gC_\dagger \,.
\]
Plugging the latter identity into \eqref{eq:extreme-pre-limit} we arrive at
\[
\varphi^{\be, B,t+1}_{\ddagger} (\gH^{\bond}_t) =
\gm_{\spin,\bond}^{t+1} \big( \varphi(\gH^{\bond}_t) \big) = \gm_{\bond}^{t+1} \big(\varphi(\gH_t^{\bond})\big), 
\qquad \forall t \in \N,
\]
which since 
$\gH^{\bond}_t=\frac{1}{c_t} \gF_t$ on $\wh{\Tree}_d^{\dagger}(t+1)$ for any boundary edges
(with $c_t=d |\partial \Tree_d(t)|$), implies that for any local weak limit point $\gm_{\bond}$ of 
$\varphi_n^{\be, B}$, 
\beq\label{eq:extreme-limit}
{\gm_{\bond}^{t+1} \big( \varphi(\gF_t) \big) =
 \varphi^{\be, B,t+1}_{\ddagger}  (\gF_t) }
, \qquad \forall t \in \N \,.
\eeq
By Lemma \ref{lem:limit-point-properties}(c) {and Definition \ref{dfn:btwn-f-w}, we know} that  
$\varphi_{\free}^{\be,B,s+1} \st \varphi \st \varphi_{\wired}^{\be, B, s+1}$
for $\gm_{\bond}^{s+1}$-a.e.~$\varphi$. Combining Lemma \ref{lem:ordering-gF}
with \eqref{eq:extreme-limit} then yields that 
$\gm_{\bond}^{s+1}(\{\varphi_{\ddagger}^{\be, B, s+1}\})=1$ for all $s \in \N$. 
{Since this applies for} any local weak limit point $\gm_{\bond}$, as claimed 
$\varphi_n^{\be, B} \lwc \varphi_{\ddagger}^{\be, B}$ when $n \to \infty$.


\subsection{Proof of Lemma \ref{lem:correlation-limit}}\label{sec:pf-correlation-limit}

{Setting hereafter $\ddagger \in \{\free,1\}$ we shall rely} on the {next two lemmas about} smoothness of the fixed points $\nu_{\ddagger}$ of the ${\sf BP}$ recursion 
and {the corresponding} marginals of two adjacent spins 
under $\mu_{\ddagger}^{\be, B}$ {(proofs of which} 
are postponed to Appendix \ref{sec:app}). 

\begin{lem}[Smoothness of $\nu_\ddagger$]\label{lem:smoothness-fixed-pt}
Let $\partial\sfR_{\neq}$ denote the boundary of the non-uniqueness regime $\sfR_{\neq}$ defined in Proposition \ref{prop:non-unique-regime}. Define
\[
\partial \sfR_{\neq}^{\free} := \left\{ (\be', B') \in \partial \sfR_{\neq}: \be'= \be_{\free}(B')\right\} \quad \text{ and } \quad \partial \sfR_{\neq}^{+} := \left\{ (\be', B') \in \partial \sfR_{\neq}: \be'= \be_{+}(B')\right\}.
\]
\begin{enumeratei}

\item If $(\be_0, B_0) \in \sfR_{\neq}\setminus \partial \sfR_{\neq}^{\free}$ then the map $(\beta, B) \mapsto \nu_1^{\be, B}$ is continuously differentiable at $(\be_0, B_0)$.

\item If $(\be_0, B_0) \in \sfR_{\neq} \setminus \partial \sfR_{\neq}^{+}$ then the same conclusion holds for $\nu_{\free}^{\be, B}$.

\end{enumeratei}
\end{lem}

\begin{lem}[Marginals of $\mu_{\ddagger}^{\be, B}$]\label{lem:marginals}
For $\ddagger \in \{\free,1\}$ and any $\be, B \ge 0$,  
\begin{align}\label{eq:mg-pairs}
\mu_{\ddagger}^{\be,B} (\sigma_i, \sigma_j) & \propto e^{\beta \delta_{\sigma_i, \sigma_j}} \nu_{\ddagger}^{\be, B}(\sigma_i) \nu_{\ddagger}^{\be, B}(\sigma_j), \; \qquad \forall (i, j) \in E(\Tree_d), \quad
\forall \sigma_i, \sigma_j \in [q],\\
\label{eq:mg-single}
\mu_\ddagger^{\be, B}(\sigma_i) &\propto \left((e^\be-1)\nu_\ddagger^{\be, B}(\sigma_i)+1\right)\nu_\ddagger^{\be, B}(\sigma_i), \qquad \forall i \in V(\Tree_d), \quad \forall \sigma_i \in [q].
\end{align}
\end{lem}

\begin{proof}[Proof of Lemma \ref{lem:correlation-limit}]
{First note} that for $(\be_0, B_0) \in \sfR_{\free}$
\beq\label{eq:partials-converge}
\lim_{n \to \infty} \f{1}{n}\sum_{(i,j) \in E_n} \mu_n^{\be_0, B_0}(\sigma_i = \sigma_j)  = \lim_{n \to \infty} \f{\partial }{\partial \be} \Phi_n(\be_0, B_0) = \f{\partial }{\partial \be} \Phi(\nu_{\free}^{\be_0, B_0}).
\eeq
{Indeed, the left equality} in \eqref{eq:partials-converge} follows by a straightforward computation, {whereas for} the right equality recall that the derivatives in $\be$ of
convex functions $\{\Phi_n(\cdot, B)\}_{n \in \N}$ converge to the derivative of their (convex) limit, 
whenever the latter exists. {Now,} as $\red{\sfR_c \subset (\sfR_{\free})^c}$, {it follows} by Proposition \ref{prop:non-unique-regime} and \purple{Proposition} \ref{asmp:main} that for any $(\be_0, B_0) \in \sfR_{\free}$ 
there exists some open neighborhood $U_{\be_0, B_0} \ni (\be_0, B_0)$ such that 
\beq\label{eq:Phi-not-on-c}
\lim_{n \to \infty} \Phi_n(\be, B) = \max\{\Phi(\nu_{\free}^{\be, B}), \Phi(\nu_{1}^{\be, B})\}=\Phi(\nu_{\free}^{\be, B}), \quad \text{ for any } (\be, B) \in U_{\be_0,B_0}. 
\eeq
{From Lemma \ref{lem:smoothness-fixed-pt}(ii) we know that}
$(\be,B) \mapsto \nu_{\free}^{\be, B}$ is differentiable 
at $(\be_0, B_0) \in \sfR_{\free} \subset \sfR_{\neq}\setminus \partial \sfR_{\neq}^+$ and 
{with} $(\nu, \be, B) \mapsto \Phi^{\be, B}(\nu)$ differentiable, by the chain rule 
{the limit 
in \eqref{eq:Phi-not-on-c}}
is differentiable in $\be$ at $(\be_0, B_0)$, yielding the right equality in \eqref{eq:partials-converge}. 
{To complete the proof of \eqref{eq:correlation-limit} for $(\be_0, B_0) \in \sfR_{\free}$
note that the identity}
\beq\label{eq:phi-derivative}
\f{\partial}{\partial \be} \Phi(\nu_{\free}^{\be_0,B_0}) 
= \f{1}{2} \sum_{i \in \partial o} \mu_{\free}^{\be_0, B_0} (\sigma_o=\sigma_i),
\eeq
{is a special case of \cite[Proposition 2.4]{DMS}. Indeed, the positivity and finite mean 
conditions \cite[(H$1$) and (H$2$)]{DMS} apply for the Potts model and $\Tree_d$, respectively, 
their differentiability requirement (H$3^\be$) is covered for $(\be,B) \mapsto \nu_{\free}^{\be, B}$
by Lemma \ref{lem:smoothness-fixed-pt}(ii), and comparing \eqref{eq:mg-pairs} with 
\cite[Eqn.~(1.22)]{DMS} shows that the translation invariance
(hence unimodular), measure $\mu_{\free}^{\be,B}$ is in the space ${\mathcal H}^\star$ of \cite{DMS}.}

{Following the same line of reasoning} we find that for $(\be_0, B_0) \in \sfR_1$ the equality \eqref{eq:correlation-limit} holds with $\mu_{\wired}^{\be,B}$ replaced by $\mu_1^{\be,B}$. 
{Finally, by} construction $\mu_{\wired}^{\be,0}(\sigma_o=\sigma_i)= \mu_k^{\be,0}(\sigma_o=\sigma_i)$ for any $k \in [q]$, and $\mu_1^{\be, B}= \mu_{\wired}^{\be, B}$ for $B >0$ (recall \eqref{eq:mu-wired}). 
\end{proof}

\section{{Dominating spin and pure state:} proof of Theorem \ref{thm:main-2}}\label{sec:pf-cond-thm}

{At $B=0$ the Potts measure is invariant under a global spin color permutation. Hence, 
$\mu_n^{\be,0}(\sK_n(\ul \sigma)=k)=1/q$ for all $k \in [q]$} (see \eqref{eq:sK_n} for 
$\sK_n(\cdot)$), with $\mu_{n,k}^{\be,0}(\cdot)=q \mu_n^{\be,0}(\cdot,\sK_n(\cdot)=k)$.

{In Section \ref{subsec-cond-free} we show that for $\be < \be_c(0)$ and 
any $k_1 \ne k_2 \in [q]$,
\beq\label{eq:small-disagreement}
\lim_{n \to \infty}  \E_n \Big| {\sf P}_{\mu_{n,k_1}^{\be,0}}^t(I_n)
(\ul \sigma_{\sfB_{I_n}(t)}=\ul \sigma) - 
{\sf P}_{\mu_{n,k_2}^{\be,0}}^t(I_n)
(\ul \sigma_{\sfB_{I_n}(t)}=\ul \sigma) \Big| = 0 \,, \qquad \forall \ul \sigma \in \bS_t, \quad 
\forall t \in \N \,. 
\eeq
By Theorem \ref{thm:main-1}\red{(i)-(ii)}, as for such $\beta$
\[
\mu_n^{\be,0}=\frac{1}{q} \sum_{k=1}^q \mu_{n,k}^{\be,0} \lwc \mu_{\free}^{\be,0} \,,
\]
we deduce in view of \eqref{eq:small-disagreement} that
the same must apply also for each $\mu_{n,k}^{\be,0}$, yielding Theorem \ref{thm:main-2}(i).

In contrast, to get $\mu_{n,k}^{\be,0} \lwc \mu_k^{\be,0}$ in Theorem \ref{thm:main-2}(ii)
amounts to showing that for $\be>\be_c(0)$,  
\beq\label{eq:cond-limit-large}
\lim_{n \to \infty} 
\E_n \big[ \mu_{n}^{\be,0} (\ul \sigma_{\sfB_{I_n}(t)}=\ul \sigma, \sK_n(\cdot)=k) \big]
= \frac{1}{q} \mu_{k}^{\be, 0}(\ul{\sigma}_{\Tree_d(t)}= \ul\sigma) \,,  
\qquad \forall \ul \sigma \in \bS_t, \quad \forall t \in \N.
\eeq
To this end, we show in Section \ref{subsec-pure-states} that for such $\be$,}
due to the edge expansion property of $\{\Graph_n\}$, with {high $\mu_n^{\be,0}$-probability}
the dominating color $\sK_n$ coincides with that for a randomly chosen `large' 
local neighborhood. {Employing} the local weak convergence from Theorem {\ref{thm:main-1}\red{(iii)}, we relate the latter functional} to the dominating color 
of $\Tree_d(t)$ {under the wired \abbr{rcm} and the proof is then completed upon identifying 
{the behavior that corresponds to a dominant color in this tree} setting.}

\subsection{Free limit ($\be<\be_c(0)$): Proof of Theorem \ref{thm:main-2}(i)}\label{subsec-cond-free}

We establish \eqref{eq:small-disagreement} by a coupling 
with an arbitrarily small fraction of disagreements between the spins under
$\mu_{n,k_1}^{\be,0}$ and under $\mu_{n,k_2}^{\be,0}$. To this end, we first show
that the free \abbr{rcm} 
on $\Tree_d$, does not have an infinite cluster when $\be < \be_c(0)$.
\begin{lem}\label{lem:phi-f-no-prcolatn}
For any $d\ge 3$, $q \ge 2$ and $\be \in [ 0, \be_{c}(0))$,
\[
\varphi_{\free}^{\be, 0}(o \bij \infty):= \lim_{t \to \infty} \varphi_{\free}^{\be, 0}(o \bij \partial \sfB_o(t)) =0. 
\]
\end{lem}
\begin{proof} Recall \cite[Theorem 10.67]{G-RC} that the \abbr{rcm} $\varphi_{\free}^{\be, 0}$
on $\Tree_d$ with free boundary condition is merely the product measure on $\{0,1\}^{E(\Tree_d)}$, where each edge is open with probability $\pi(\be):= \f{p}{p+q(1-p)}$ for $p=p_e=1-e^{-\be}$. Since the {\em branching number} of $\Tree_d$ is $(d-1)$ (see \cite[Section 2]{L90} for a definition), the stated result follows from \cite[Theorem 6.2]{L90}, upon verifying that
\beq\label{eq:m-beta}
m(\be) := (d-1) \pi(\be) = \f{(d-1)(e^\be -1)}{e^\be + q -1} < 1 \,, \qquad \forall \be \in [0, \be_c(0))\,.
\eeq
By the monotonicity of $y \mapsto \f{y}{y+a}$ whenever $a >0$, it suffices for \eqref{eq:m-beta} to show that $m(\be_c(0)) \le 1$. To this end, note that by Jensen's inequality,
\[
h(x):=\f{x^{1-d/2} (x^{d-1} -1)}{(d-1)(x-1)} = \frac{1}{d-1} \sum_{j=0}^{d-2} x^{j+1-d/2} > 1 \,,
\qquad \forall x > 1 \,.
\]
Now, recall from the discussion after \cite[Theorem 1.1]{BBC22+}, that for any $d,q \ge 3$,
\[
e^{\be_c(0)} = \f{q-2}{(q-1)^{1-2/d} -1} \,.
\]
In view of \eqref{eq:m-beta}, it then follows that $m(\be_c(0))=1/h((q-1)^{2/d})<1$. Similarly, 
$e^{\be_c(0)} =d/(d-2)$ when $q=2$, in which case $m(\be_c(0))=1$. Thus,
$m(\be) < m(\be_c(0)) \le 1$ for all $q \ge 2$, as claimed.
\end{proof}


With $\varphi_n^{\be,0} \lwc \varphi_{\free}^{\be, 0}$
(see Theorem \ref{thm:main-1}\red{(i)-(ii)}), we next show that Lemma \ref{lem:phi-f-no-prcolatn} 
implies having only a small fraction of the vertices of $\Graph_n$ in large open connected 
components of $\varphi_n^{\be,0}$ (whenever $\be < \be_c(0)$).
%
\begin{lem}\label{lem:no-large-cluster} 
Let $\sN_n(r):=\sN_n(r, \ul \eta)$, $r \in \N$, denote the number of open connected 
components of size $r$ in $\Graph_n$ equipped with bond configuration 
$\ul \eta \in \{0,1\}^{E_n}$. For any $\vep >0$ and $\beta \in [0, \beta_{c}(0))$, there exists some 
$\ell_\star = \ell_\star(\vep, \beta) < \infty$ such that 
\[
\limsup_{n \to \infty} \varphi_n^{\be, 0}\Big(\sum_{r \ge \ell_\star} r \sN_n(r) \ge \vep n \Big) \le \vep. 
\]
\end{lem}
\begin{proof}
Fixing $\ell \in \N$, note that the event $\{\sum_{r \ge \ell} r \sN_n(r) \ge \vep n\}$ implies 
having at least $\vep n$ vertices of $\Graph_n$ with open connected components 
of size at least $\ell$. Setting the smallest $t_\ell \in \N$ such that $|\Tree_d(t_\ell)| \ge \ell$, where for a finite graph $\Graph$ the notation $|\Graph|$ is used to denote the cardinality of its vertices,
for any $i \in [n]$ with open connected component of size at least $\ell$ and 
$\sfB_i(t_\ell) \cong \Tree_d(t_\ell)$, there must be an open path within $\sfB_i(t_\ell)$
from $i$ to $\partial \sfB_i(t_\ell)$.
This together with Markov's inequality, Theorem \ref{thm:main-1}\red{(i)-(ii)} for the \abbr{rcm}-s 
$\varphi_n^{\be,0}$, and the fact that $\Graph_n \loc \Tree_d$, entail that
\begin{multline*}
\limsup_{n \to \infty} \varphi_n^{\be, 0}\Big(\sum_{r \ge \ell} r \sN_n(r) \ge \vep n \Big) \\
\le \vep^{-1} \lim_{n \to \infty} \E_n \left[\varphi_n^{\be, 0}(I_n \stackrel{{\rm in}}{\bij}
 \partial \sfB_{I_n}(t_\ell)) \cdot {\bf 1}(\sfB_{I_n}(t_\ell) \cong \Tree_d(t_\ell))\right]=  \vep^{-1} \varphi_{\free}^{\be, 0}( o \leftrightarrow \partial \sfB_o(t_\ell)). 
\end{multline*}
Since $t_\ell \to \infty$ as $\ell \to \infty$, by Lemma \ref{lem:phi-f-no-prcolatn} 
we can choose $\ell$ large enough so that the \abbr{RHS} above is at most $\vep$, 
thereby completing the proof. 
\end{proof}

Recall that, by Lemma \ref{lem:coupling-es}(ii) at $B=0$, given a random cluster bond configuration,
the Potts spin configuration is obtained by assigning a single color to all vertices in each connected component, uniformly at random, and independently across all components. 
Our next `complicated' procedure of generating such independent discrete 
uniform random variables is the key to our promised coupling of $\mu_{n,k}^{\be,0}$
and $\mu_{n,k'}^{\be,0}$.

\begin{lem}\label{lem:sim-unif}
Suppose $\sM:=(\sM_1, \sM_2, \ldots, \sM_q) \stackrel{d}{=}{\rm Mult}_q(M, q^{-1}, q^{-1}, \ldots,q^{-1})$, follows the multinomial distribution with $M$ trials and equal probability
$q^{-1}$ for each of the $q$ categories. Conditioned on $\sM$ choose uniformly at random
a labeled partition of $[M]$
to 
distinguished sets $\{B_k, \wt B_k\}$ of sizes
$|B_k|= \sM_\star := \min_{k} \{\sM_k\}$ and
$|\wt B_k| = \sM_k - \sM_\star$. 
For a uniformly chosen permutation $\ul \gam$ of $[q]$,
set $Y_i=k$ whenever $i \in B_k \cup \wt B_{\gam(k)}$. This yields \abbr{iid} variables
$\{Y_i\}_{i=1}^M$, each following the discrete uniform law $\dU([q])$ on $[q]$. 
\end{lem}
\begin{proof} It suffices to show that for ${\bm Y}:=(Y_1, Y_2, \ldots, Y_M)$, 
any fixed ${\bm k}:=(k_1, k_2, \ldots, k_M) \in [q]^M$ and a fixed permutation
${\bm \gamma}^0$ of $[q]$,
\beq\label{eq:sim-unif}
\P({\bm Y}={\bm k}\,|\, \ul \gamma = {\bm \gam}^0) = q^{-M}.
\eeq
Now, given $\ul \gam={\bm \gam}^0$, 
the event $\Omega_{{\bm k}}:=\{{\bm Y}={\bm k}\}$ induces the partition $\wh B_i:= \{j \in [M] : k_j =i\}$
of $[M]$ with $\wh B_i = B_i \cup \wt B_{{\bm \gam}^0 (i)}$. Setting $|\wh B_i|=y_i$, since
$|B_i|=y_\star := \min_{j}\{y_j\}$ for all $i \in [q]$, we also have $|\wt B_{{\bm \gam}^0 (i)}|=y_i-y_\star$.
That is, given $\ul \gam$, the event $\Omega_{\bm k}$ determines the 
sizes and colors of $\{B_i,\wt B_i\}$, with only 
the choices of $B_i \subset \wh B_i$ indeterminate. In fact, any such choice  of $\{B_i\}$
with $|B_i|=y_\star$ produces a realization of the event $\Omega_{{\bm k}}$. 
There are $\prod_{i=1}^q\binom{y_i}{y_\star}$ such choices
and the probability of observing each realization of $\{B_i,\wt B_i\}_{i=1}^q$ is then
\[
\f{1}{\binom{M}{y_1, y_2, \ldots, y_q} \prod_{i=1}^q \binom{y_i}{y_\star}}\P(\sM_i=y_i, i \in [q]) = \f{1}{\prod_{i=1}^q \binom{y_i}{y_\star}} q^{-M}\,.
\]
Hence, taking a union over the set of all possible choices of $\{B_i\}_{i=1}^q$ 
we arrive at \eqref{eq:sim-unif}. 
\end{proof}

In Lemma \ref{lem:sim-unif} the dominant color is completely determined by the 
$\sum_{k}|\sM_k - \sM_\star|$ colors involving $\ul \gamma$, and   
by standard concentration bounds with high probability there are only $o(M)$ such colors.
We can thus 
produce an identical copy $\{Y_i'\}_{i=1}^M$ with only $o(M)$ discrepancy from $\{Y_i\}_{i =1}^M$ 
and two different specified dominant colors. Indeed, building on Lemmas 
\ref{lem:no-large-cluster} and \ref{lem:sim-unif} we proceed this way 
to establish \eqref{eq:small-disagreement} (and thereby get Theorem \ref{thm:main-2}(i)).

\begin{proof}[Proof of \eqref{eq:small-disagreement}]
Fixing $k_1 \ne k_2 \in [q]$, our proof hinges on producing Potts spin configurations 
$\ul \sigma^1, \ul \sigma^2 \in [q]^n$ 
that with high probability agree up to $o(n)$ sites, while
$\sK_n(\ul \sigma^1)=k_1 \iffa \sK_n(\ul \sigma^2)=k_2$. To this end, equip 
$\Graph_n$ with bond configuration $\ul \eta \in \{0,1\}^{E_n}$ and 
let $W_r \subset [n]$, $r \ge 1$, denote the vertex disjoint union 
of its $\sN_n(r):=\sN_n(r, \ul \eta)$ open connected components of size $r$. 
Next, fix a uniformly random permutation $\ul \gam$ of $[q]$ and set
$\ul \gam'$ such that $\gam'(k)=\gam(k)$ for $k \notin \{k_1, k_2\}$ with 
$\gamma'(k_1)=\gamma(k_2)$ and $\gamma'(k_2)=\gamma(k_1)$. Choosing 
\[
\sM=\sM^{(r)}:=(\sN_n(r,1), \sN_n(r,2), \ldots, \sN_n(r,q)) \stackrel{d}{=}{\rm Mult}_q(\sN_n(r), q^{-1}, q^{-1}, \ldots,q^{-1}) \,,
\]
and the corresponding partition of $[\sN_n(r)]$ as in 
Lemma \ref{lem:sim-unif}, induces via $\ul \gam$ a color coding for 
each open connected component of size $r$. Thereby assigning that same color to all vertices of
such a component, we obtain a spin configuration $\ul \sigma_{W_r}^1$. Following the same 
procedure except for replacing $\ul \gam$ by $\ul \gam'$ yields another spin configuration 
$\ul \sigma_{W_r}^2$. With $\sW_r \subset W_r$ denoting the (random) set of sites 
whose color is independent of the choice of $\ul \gam$, we have by construction that
$\ul \sigma^1_{\sW_r} = \ul \sigma^2_{\sW_r}$ and
\[
|W_r \setminus \sW_r| \le q r \max_{k, k' \in [q]}|\sN_n(r,k) - \sN_{n}(r,k')| \,.
\] 
Repeating this procedure for each $r \in \N$, independently across different $r$'s yields
a random set of sites $\sW \subset [n]$ with color independent of $\ul \gam$ and
spin configurations $\ul\sigma^1, \ul \sigma^2 \in [q]^n$ with 
$\ul\sigma_{\sW}^1 = \ul\sigma_{\sW}^2$ and
\beq\label{eq:d-H-bd}
|\sW^c| \le q \sum_{r \in \N} r \max_{k \ne k' \in [q]}|\sN_n(r,k) - \sN_n(r,k')|. 
\eeq
In view of Lemmas \ref{lem:coupling-es}(ii) and \ref{lem:sim-unif}, for
$\ul \eta$ drawn according the \abbr{rcm} $\varphi_n^{\be, 0}$, the
marginal laws of $\ul\sigma^1$ and $\ul \sigma^2$ are both given by the 
Potts measure $\mu_n^{\be, 0}$. 
Further, these spin configurations are such that $\sK_n(\ul \sigma^1)= k_1 \iffa \sK_n(\ul \sigma^2)=k_2$. Indeed, since the number of sites in $\sW$ with any given color is exactly $q^{-1}|\sW|$ in both 
$\ul\sigma^1$ and $\ul\sigma^2$, the event $\sK_n(\ul \sigma^1)=k_1$ amounts to 
spin configuration $\ul\sigma_{\sW^c}^1$ of dominating color $k_1$. The relation between 
$\ul\gam$ and $\ul \gam'$ dictates that this happens if and only if the same holds 
for $\ul \sigma^2$ with $k_1$ replaced by $k_2$, namely that equivalently 
$\sK_n(\ul\sigma^2)=k_2$.

Armed with these observations, for any $t \in \N$ and $\ul \sigma \in \bS_t$, we deduce that
\begin{multline}\label{eq:couplin-ne-bd}
\f1n\sum_{i=1}^n \left|{\bf 1}(\ul \sigma^1_{\sfB_i(t)}= \ul \sigma, \sK_n(\ul \sigma^1)=k_1) - {\bf 1}(\ul \sigma^2_{\sfB_i(t)}= \ul \sigma, \sK_n(\ul \sigma^2)=k_2)\right|  \\
\le \f1n\sum_{i=1}^n {\bf 1}(\sfB_i(t) \cap \sW^c \ne \emptyset, \sfB_i(2t) \cong \Tree_d(2t)) + \f1n\sum_{i=1}^n {\bf 1}(\sfB_i(2t) \ncong \Tree_d(2t))\red{.} 
\end{multline}
Clearly, ${\bf 1}(\sfB_i(t) \cap \sW^c \ne \emptyset) \le \sum_{j \in \sW^c} {\bf 1}(\sfB_i(t) \ni j)$ and if $j \in \sfB_i(t)$ for $\sfB_i(2t) \cong \Tree_d(2t)$ then also 
$\sfB_j(t) \cong \Tree_d(t)$. Thus, the first term in the \abbr{RHS} of \eqref{eq:couplin-ne-bd} is bounded above by
\begin{equation}\label{eq:couplin-ne-bd-1}
\f1n\sum_{j \in \sW^c} \sum_{i =1}^n {\bf 1}(\sfB_i(2t) \cong \Tree_d(2t), \sfB_i(t) \ni j) 
\le \f1n\sum_{j \in \sW^c}|\sfB_j(t)| {\bf 1}(\sfB_j(t) \cong \Tree_d(t))  \le (d+1)^t \f{|\sW^c|}{n}. 
\end{equation}
Recalling that $\mu_{n,k}^{\be,0}(\cdot)=q \mu_n^{\be,0}(\cdot,\sK_n(\cdot)=k)$ and
$\Graph_n \loc \Tree_d$, by \eqref{eq:couplin-ne-bd}, \eqref{eq:couplin-ne-bd-1} and the triangle inequality, we find that for any $k_1 \ne k_2 \in [q]$ and $t \in \N$,
\[
\limsup_{n \to \infty}  \E_n \Big| {\sf P}_{\mu_{n,k_1}^{\be,0}}^t(I_n)
(\ul \sigma_{\sfB_{I_n}(t)}=\ul \sigma) - 
{\sf P}_{\mu_{n,k_2}^{\be,0}}^t(I_n)
(\ul \sigma_{\sfB_{I_n}(t)}=\ul \sigma) \Big|
\le q(d+1)^t 
\limsup_{n \to \infty} \f{1}{n} \E [|\sW^c|],
\]
where the expectation on the \abbr{RHS} is with respect to both the underlying \abbr{rcm} 
and our two color assignments via Lemma \ref{lem:sim-unif}. Thus, to conclude with 
\eqref{eq:small-disagreement} it suffices to fix $\vep>0$ and show that 
\beq\label{eq:sW-bd}
\limsup_{n \to \infty} n^{-1} \E[|\sW|^c] \le 3 q\vep \,.
\eeq
To this end, set for $\ell_\star(\vep)$ as in Lemma \ref{lem:no-large-cluster} and
$0 < \vep_\star \le \vep \ell_\star(\vep)^{-2}$ the events
\begin{align*}
\cC:= \cC_1 \!\!\! \bigcap_{r < \ell_\star(\vep)} \!\!\! \cC_2(r), \;\; 
\cC_1:= \Big\{ \sum_{r \ge \ell_\star(\vep)} r \sN_n(r) \le \vep n \Big\}, \;\;  \cC_2(r):= \Big\{ \max_{k, k' \in [q]}|\sN_n(r, k) -  \sN_n(r, k')| \le \vep_\star n \Big\}.
\end{align*}
Note that on the event $\cC$ we have from \eqref{eq:d-H-bd} and our choice of $\vep_\star$ that
\[
q^{-1} |\sW^c| \le \sum_{r \in \N} r \max_{k, k' \in [q]} |\sN_n(r,k) - \sN_n(r,k')| \le 
\sum_{r \ge \ell_\star(\vep)} r \sN_n(r) + \vep_\star n \sum_{r < \ell_\star(\vep)} r \le  2\vep n \,.
\]
With $|\sW^c| \le n$, we thus arrive at \eqref{eq:sW-bd} and thereby at \eqref{eq:small-disagreement},
upon showing that 
\beq\label{eq:sW-bd-1}
\limsup_{n \to \infty} \P(\cC_1^c) + \limsup_{n \to \infty} 
\sum_{r<\ell_\star(\vep)}  \P(\cC_2(r)^c) \le q \vep \,.
\eeq
Next note that for some $c=c(q) >0$ and any $\vep_\star>0$, $r \in \N$, by the union bound  
\begin{align*}
\P( \max_{k \ne k' \in [q]} |\sN_n(r,k) -  \sN_n(r,k')| \ge \vep_\star \sN_n(r)|\ul\eta) & \le 
\sum_{k \in [q]} \P(|\sN_n(r,k) -  \sN_n(r)/q| \ge \vep_\star \sN_n(r)/2 \,|\,\ul\eta) \\
& \le 2 q \exp(-c \, \vep_\star^2 \sN_n(r))
\end{align*}
(the last step is a standard Binomial$(m,q^{-1})$ tail bound). 
Further $\cC_2(r)^c \subseteq \{ \sN_n(r) > \vep_\star n \}$, hence 
\[
\sum_{r < \ell_\star(\vep)} \P(\cC_2(r)^c) \le 2 q \ell_\star(\vep) e^{-c \, \vep_\star^3 \, n}\,.
\]
The latter bound goes to zero as $n \to \infty$, so by Lemma \ref{lem:no-large-cluster} 
we have that 
\eqref{eq:sW-bd-1} holds, as claimed. 
\end{proof}

\subsection{Pure states ($\be > \be_c(0)$): Proof of Theorem \ref{thm:main-2}(ii)}\label{subsec-pure-states}
 
For any $\ell \in \N$ we have the following local proxy at $u \in V$  
for the dominant color $\sK_\Graph(\cdot)$ of a graph $\Graph=(V,E)$ (possibly infinite):
\[
\sfcK_{\ell, \Graph}(u):={\rm arg}\max_{k \in [q]} \Big\{ {\sf N}_{\ell, \Graph}(u,k) \Big\}\,, 
\qquad
{\sf N}_{\ell, \Graph}(u,k):= \f{1}{|\sfB_{u, \Graph}(\ell)|}\sum_{v \in \sfB_{u,\Graph}(\ell)} 
{\bf 1}(\sfB_{v,\Graph}(2\ell) \cong \Tree_d(2\ell)) \delta_{\sigma_v, k}  
\,, 
\]
and we break ties uniformly among all maximizer values (as done in Definition 
\ref{dfn:dom-spin}). Further let 
\[
{\sf N}_{\ell, \Graph}^{(1)}(u):= \max_{k \in [q]} \{ {\sf N}_{\ell, \Graph}(u,k) \} 
\quad \text{ and } \quad {\sf N}_{\ell, \Graph}^{(2)}(u):= \max_{k \ne \sfcK_{\ell, \Graph}(u)} \{ {\sf N}_{\ell, \Graph}(u,k) \},
\]
suppressing the dependency on $\Graph$ when it is clear from the context. Our next result identifies the behavior of the dominating color of $\Tree_d(\ell)$, under the Potts measures 
$\{\mu_k^{\be, 0}\}_{k \in [q]}$.  
\begin{lem}\label{lem:pott-inifinite-prop}
Fix $\be \ge \be_{\free}(0)$ and $d \ge 3$. Then, for any $\ell \in \N$,
\beq\label{eq:potts-inf-prop-1}
 \mu_{k}^{\be, 0} \left( {\sf N}_{\ell, \Tree_d}(o,k')\right) = \left\{ \begin{array}{ll}
\mu_{1}^{\be, 0}(\sigma_o=1) & \mbox{if } k=k',\\
\mu_{1}^{\be, 0}(\sigma_o=2)& \mbox{if } k \ne k'.
\end{array}
\right.
\eeq
Further,
\beq\label{eq:potts-inf-prop-2}
\lim_{\ell \to \infty}\Var_{\mu_{k}^{\be, 0}} \left({\sf N}_{\ell, \Tree_d}(o,k') \right) =0, \qquad k, k' \in [q], 
\eeq
and consequently,
\beq\label{eq:potts-inf-prop-3}
\lim_{\ell \to \infty} \mu_{k}^{\be, 0} \left( \sfcK_{\ell, \Tree_d}(o)=k'\right) = \delta_{k,k'}, \qquad k, k' \in [q].
\eeq
\end{lem}

\begin{rmk}\label{rmk:pott-inifinite-prop}
As $\Tree_d$ is a vertex transitive graph, the measures $\{\mu_k^{\be, 0}\}_{k \in [q]}$ 
are translation invariant on $V(\Tree_d)$, so Lemma \ref{lem:pott-inifinite-prop} 
applies even when the root $o$ is replaced by any $j \in V(\Tree_d)$.  
\end{rmk}

\begin{proof}[Proof of Lemma \ref{lem:pott-inifinite-prop}]
We first show \eqref{eq:potts-inf-prop-3} given \eqref{eq:potts-inf-prop-1}-\eqref{eq:potts-inf-prop-2}. Indeed, fixing $k \ne k' \in [q]$ note that
\begin{align}\label{eq:potts-inf-prop-4}
\mu_{k}^{\be, 0} \left( \sfcK_{\ell}(o)=k'\right) \le \mu_{k}^{\be, 0} \left( {\sf N}_{\ell}(o, k)  \le  {\sf N}_{\ell}(o, k')\right). 
\end{align}
Applying \eqref{eq:potts-inf-prop-1}-\eqref{eq:potts-inf-prop-2} and Chebychev's inequality, we see that, for any $\delta >0$,
\[
\limsup_{\ell \to \infty} \mu_{k}^{\be, 0}  \Big( {\sf N}_{\ell}(o, k) \le \mu_{1}^{\be, 0}(\sigma_o=1) -\delta\Big) \le \f{1}{\delta^2} \lim_{\ell \to \infty}\Var_{\mu_{k}^{\be, 0}} \left( {\sf N}_{\ell}(o, k)\right) =0.
\]
By a similar argument
\[
\lim_{\ell \to \infty} \mu_{k}^{\be, 0}  \left( {\sf N}_{\ell}(o, k') \ge \mu_{1}^{\be, 0}(\sigma_o=2) +\delta\right) =0.
\]
From Proposition \ref{prop:non-unique-regime}(ii) and the definition of $\nu_{\free}^{\be, 0}$ it follows that $q^{-1}= \nu_{\free}^{\be, 0}(1) < \nu_1^{\be, 0}(1)$ for $\be \ge \be_{\free}(0)$. 
The definition of $\nu_1$ further yields that $\nu_1^{\be, 0}(2)= (q-1)^{-1} (1- \nu_1^{\be, 0}(1))$ and thus $\nu_1^{\be, 0}(1) > \nu_1^{\be, 0}(2)$ for $\be \ge \be_{\free}(0)$. As the map $\nu \mapsto \Big( 
(e^\be -1)\nu+1\Big) \nu$ is strictly increasing on $(0, \infty)$ it follows from 
\eqref{eq:mg-single} that
$\mu_{1}^{\be, 0}(\sigma_o=1) > \mu_{1}^{\be, 0}(\sigma_o=2)$, for any $\be \ge \be_{\free}(0)$.  
Now setting $\delta:= \f13(\mu_{1}^{\be, 0}(\sigma_o=1) - \mu_{1}^{\be, 0}(\sigma_o=2))$, 
we deduce that for any $k \ne k' \in [q]$,
\beq\label{eq:dom-col-nuk}
\lim_{\ell \to \infty}  \mu_{k}^{\be, 0} \left( {\sf N}_{\ell}(o, k)  -  {\sf N}_{\ell}(o, k') \ge \delta \right) =1.
\eeq
Comparing \eqref{eq:potts-inf-prop-4} with \eqref{eq:dom-col-nuk}
and using that $\sum_{k'=1}^q \mu_{k}^{\be, 0} \left( \sfcK_{\ell}(o)=k'\right)=1$, results with  \eqref{eq:potts-inf-prop-3}.

Setting $\pi_k(1)=k$, $\pi_k(k)=1$ and $\pi(k')=k'$ whenever $k' \ne 1$, $k' \ne k$,
the proof of \eqref{eq:potts-inf-prop-1} is immediate by the translation invariance
of $\{\mu_{k}^{\be, 0}\}_{k \in [q]}$ and having that for any finite set $W \subset V(\Tree_d)$,
\beq\label{eq:mu-1-k}
\mu_{1}^{\be,0} (\sigma _i =k_i', i \in W) = \mu_{k}^{\be,0} (\sigma _i =\pi_{k}(k'_i), i \in W), \qquad k, k_i' \in [q]. 
\eeq
Turning to establish \eqref{eq:potts-inf-prop-2}, by \eqref{eq:mu-1-k} we can set 
$k=1$, further noting that by the translation invariance of $\mu_1^{\be,0}$ 
and the vertex transitivity of $\Tree_d$,
\beq\label{eq:cov-1}
\Cov_{\mu_{1}^{\be,0}}\Big(\delta_{\sigma_v, k'}, \delta_{\sigma_u,k'}\Big) = c_{k'}({\rm dist} (u,v) ), \qquad 
\forall u,v \in V(\Tree_d).
\eeq
By definition $c_{k'}(s) \le 1$ and thus, for any $t \in \N$,
\[
\Var_{\mu_{1}^{\be, 0}} \left({\sf N}_{\ell, \Tree_d}(o,k') \right) = \frac{1}{|\Tree_d(\ell)|^2}
\sum_{u,v \in \Tree_d(\ell)} c_{k'}({\rm dist}(u,v)) \le \frac{|\Tree_d(t)|}{|\Tree_d(\ell)|} + \max_{s \ge t} \{
c_{k'}(s)\} \,.
\]
As $|\Tree_d(\ell)| \to \infty$, it remains only to show that $c_{k'} (t) \to 0$ when $t \to \infty$.
This in turn follows from the fact that the marginal of $\mu_{1}^{\be,0}$ on any fixed ray in $\Tree_d$ 
is a time homogeneous Markov chain of finite state space ($[q]$), and strictly positive transition 
probabilities. Indeed, generalizing the proof of Lemma \ref{lem:marginals} along the lines of Remark 
\ref{rmk:infinite-vol-exist}, shows that such marginal must be a (possibly inhomogeneous)
Potts measure on $\N$, hence a Markov chain with strictly positive transition matrices, which 
thanks to the translation transitivity of $\mu^{\be,0}_1$ must also be time homogeneous.
\end{proof}

Setting for $\delta, \eta >0$ and $i \in [n]$, 
\begin{equation*}
\cA_{\ell, n}^{\delta} := \left\{ \left|\left\{(i,j) \in E_n: \sfcK_\ell(i) \ne \sfcK_\ell(j)\right\} \right| 
\ge \delta n \right\}  \quad \text{ and } \quad
\cB_{\ell, n}^{\eta}(i) := 
\Delta_i {\bf 1} \left\{{\sf N}_{\ell, \Graph_n}^{(1)}(i) - {\sf N}_{\ell, \Graph_n}^{(2)}(i) \le \eta\right\},
\end{equation*}
we proceed to show that as $n,\ell \to \infty$, under $\mu_n^{\be,0}$ both 
$\cA_{\ell, n}^{\delta}$ and $\cB_{\ell, n}^{\eta}(I_n)$ become negligible. This observation together with the assumed edge expansion property of $\Graph_n$ will ensure that $\sK_\ell(I_n)$ and $\sK_n$ are same on an event with arbitrarily large probability for all large $n$ and $\ell$.
\begin{lem}\label{lem:cr-AB}
For any $\beta > \be_{c}(0)$, $\delta>0$ and small $\eta=\eta(\beta)>0$, under  \purple{the assumption $\Graph_n \loc \Tree_d$}, 
\beq\label{eq:cr-AB}
\lim_{\ell \to \infty} \limsup_{n \to \infty} \mu_n^{\be,0}(\cA_{\ell, n}^\delta) = 0 \quad \text{ and } \quad \lim_{\ell \to \infty} \limsup_{n \to \infty} \E_n[\mu_n^{\be,0}(\cB_{\ell, n}^\eta(I_n))] =0. 
\eeq
\end{lem}
\begin{proof} Fixing $\ell \in \N$, by Markov's inequality, the local weak convergence of Theorem \ref{thm:main-1}\red{(iii)} and the
uniform integrability of $\Delta_{I_n}$ 
\begin{align}\label{eq:cr-AB-eq1}
\limsup_{n \to \infty} \mu_n^{\be,0}(\cA_{\ell, n}^\delta) \le \delta^{-1} & \limsup_{n \to \infty}  \E_n\Big[ \sum_{j \in \partial I_n} \mu_n^{\be,0}(\sfcK_\ell(I_n) \ne \sfcK_\ell(j))\Big] \nonumber \\
&= (q\delta)^{-1}  \sum_{k=1}^q \sum_{j \in \partial o} \mu_k^{\be, 0}(\sfcK_\ell(o) \ne \sfcK_\ell(j)).
\end{align}
As $\{\sfcK_\ell(o) \ne \sfcK_\ell(j)\}\subset \{\sfcK_\ell(o) \ne k\} \cup \{\sfcK_\ell(j) \ne k\}$, 
the proof of the first assertion in \eqref{eq:cr-AB} completes upon combining \eqref{eq:cr-AB-eq1} and  \eqref{eq:potts-inf-prop-3} (see also Remark \ref{rmk:pott-inifinite-prop}).  

Turning to the second assertion of \eqref{eq:cr-AB}, denote  
\[
\cB_{\ell, n}^{\eta}(i) := \sum_{k=1}^q \cB_{\ell, n}^{\eta, k}(i), \quad \text{ where } \quad 
\cB_{\ell, n}^{\eta, k}(i):= 
\Delta_i {\bf 1} \left\{\sfcK_{\ell, \Graph_n}(i)= k ,\;
{\sf N}_{\ell, \Graph_n}^{(1)}(i) - {\sf N}_{\ell, \Graph_n}^{(2)}(i) \le \eta \right\}.
\]
Hence, by Theorem \ref{thm:main-1}\red{(iii)} and 
the uniform integrability of $\Delta_{I_n}$ we have that
\beq\label{eq:bn-eta-lim}
\lim_{n \to \infty} \E_n \left[ \mu_n^{\be,0}\left(\cB_{\ell, n}^\eta(I_n) \right)\right] = \frac{d}{q} \sum_{j, k=1}^q \mu_j^{\be, 0}\left(\sfcK_\ell(o)=k, {\sf N}_{\ell,\Tree_d}^{(1)}(o) - {\sf N}_{\ell,\Tree_d}^{(2)}(o) \le \eta \right). 
\eeq
By \eqref{eq:dom-col-nuk} and the union bound, for small enough $\eta=\eta(\be)>0$ 
and all $k \in [q]$
\[
\lim_{\ell \to \infty} \mu_k^{\be, 0}\left(\sfcK_\ell(o)=k, {\sf N}_{\ell,\Tree_d}^{(1)}(o) - 
{\sf N}_{\ell,\Tree_d}^{(2)}(o) \ge \eta \right) = 1, 
\]
which in combination with \eqref{eq:bn-eta-lim} yields the second assertion of \eqref{eq:cr-AB}.
\end{proof}

\begin{proof}[Proof of \eqref{eq:cond-limit-large}]
Recall from Theorem \ref{thm:main-1}\red{(iii)} that for any $\be > \be_c(0)$, $k \in [q]$, {$\ell \ge t$}, and 
$\ul \sigma \in \bS_t$,
\beq\label{eq:cll-aux2}
\lim_{n \to \infty} \E_n \big[ {\sf P}_{\mu_n^{\be,0}}^{\ell}(I_n)(\ul\sigma_{\sfB_{I_n}(t)} =\ul\sigma, \sfcK_\ell(I_n)=k) \big] = \f1q\sum_{k'=1}^q \mu_{k'}^{\be, 0}(\ul\sigma_{\Tree_d(t)} =\ul\sigma, \sfcK_{\ell, \Tree_d}(o)=k).
\eeq
Furthermore, for any such $k$, $\ell \ge t,$ and $\ul \sigma$,
\begin{align}
\E_n 
\big| \mu_{n}^{\be,0} (\ul \sigma_{\sfB_{I_n}(t)}=\ul \sigma, \sK_n=k) 
 -  {\sf P}_{\mu_n^{\be,0}}^{\ell}(I_n)  & (\ul\sigma_{\sfB_{I_n}(t)} =\ul\sigma, \sfcK_\ell(I_n)=k)
 \big| 
 \nonumber  \\ &
  \le 
 \sum_{k=1}^q \sum_{k' \ne k}
  \E_n[\mu_n^{\be,0} (\sfcK_\ell(I_n)=k', \sK_n=k)],
 \label{eq:cll-aux1}
\end{align}
and from \eqref{eq:potts-inf-prop-3} of Lemma \ref{lem:pott-inifinite-prop}, 
for any $k,k' \in [q]$,
\beq\label{eq:cll-aux3}
\lim_{\ell \to \infty} \mu_{k'}^{\be, 0}(\ul\sigma_{\Tree_d(t)} =\ul\sigma, \sfcK_{\ell, \Tree_d}(o)=k) = \delta_{k,k'} \cdot \mu_{k}^{\be, 0}(\ul\sigma_{\Tree_d(t)} =\ul\sigma). 
\eeq
In view of \eqref{eq:cll-aux2}-\eqref{eq:cll-aux3} and the triangle inequality, 
\eqref{eq:cond-limit-large} is a direct consequence of 
\beq\label{eq:cll-claim}
\lim_{\ell \to \infty} \limsup_{n \to \infty} \E_n [\mu_n^{\be,0} (\sfcK_\ell(I_n)=k',\sK_n=k)] =0, 
\quad \forall k' \ne k \,.
\eeq
\red{To prove \eqref{eq:cll-claim} we use the edge-expansion property to show that,  outside an event of negligible probability, there exists a color $k \in [q]$ such that up to at most $o(n)$ vertices $v \in \Graph_n$, 
the $\ell$-neighborhood around $v$ has dominating color $k$ with a bounded away from zero 
difference between the fraction of vertices of that neighborhood possessing color $k$ and the fraction of 
those possessing any other color.  Combining this with $\Graph_n \loc \Tree_d$,  we deduce that 
$\sK_n= \sK_\ell(I_n)$ with probability approaching one,  as $\ell \to \infty$,  yielding \eqref{eq:cll-claim}.
Turning to carry out the details}, 
fix $\eta=\eta(\be)$ as in Lemma \ref{lem:cr-AB} and set
 \[
 S_{k, \ell}:= \{i \in [n]: \sfcK_\ell(i)=k, {\sf N}_\ell^{(1)}(i) - {\sf N}_\ell^{(2)}(i)  > \eta \}, \qquad k \in [q].
\]
We next show that the assumed edge expansion property of $\Graph_n$ allows us to restrict
attention when $\ell$ is large and $\vep<1/(2q)$, to the event
\beq\label{dfn:wtSkl}
\cS_\ell := \bigcup_{k=1}^q \cS_{k,\ell} \quad \text{ where } \quad 
\cS_{k, \ell}:=\{\max_{k' \ne k} \{|S_{k', \ell}|\} < \vep n, \; |S_{k, \ell}| \ge (1-q\vep) n\} \,.
\eeq
Indeed, with $\sum_{k} |S_{k,\ell}| \le n$, there is at most one $k \in [q]$ with $|S_{k,\ell}| > n/2$ 
and consequently,
\[
\cS_\ell^c \subset  \Big\{ \sum_{k=1}^q |S_{k,\ell}| \le (1-\vep) n \Big\} \bigcup_{k=1}^q 
\Big\{ |S_{k,\ell}| \in [\vep n,n/2] \Big\} \,.
\]
Now, fix $\vep \in (0,\frac{\eta}{4q})$, let $\lambda = \lambda_{\vep} \wedge 2$
for $\lambda_\vep>0$ of Theorem \ref{thm:main-2}(ii). The event $|S_{k, \ell}| \in [\vep n,n/2]$ implies 
by the edge expansion property of $\Graph_n$ that $|\partial S_{k, \ell}| \ge \lambda \vep n
:=2 \delta n$. 
Further, noting that 
\[
\sum_{i=1}^n \Delta_i {\bf 1}\big({\sf N}_\ell^{(1)}(i) -{\sf N}_\ell^{(2)}(i) \le \eta\big) +
\left|\left\{(i,j) \in E_n: \sfcK_\ell(i) \ne \sfcK_\ell(j)\right\}\right| \ge 
\max_{k} \{ |\partial S_{k, \ell}| \} \,,
\]
it follows from the preceding that under 
$\cS_\ell^c$ either the event 
$\big\{ \sum_{i=1}^n \cB_{\ell, n}^{\eta}(i) \ge n \delta\big\}$ or $\cA^\delta_{\ell,n}$ holds.
Thus, from Lemma \ref{lem:cr-AB} we deduce that, as claimed earlier, 
\beq\label{eq:cS-ell}
\lim_{\ell \to \infty} \limsup_{n \to \infty} \mu_n^{\be,0}(\cS^c_\ell) =0.
\eeq
Having to contend only with the (disjoint) events 
$\cS_{k,\ell}$ of \eqref{dfn:wtSkl}, we 
arrive at \eqref{eq:cll-claim} upon showing that for any $k \ne k' \in [q]$ and $\ell \in \N$,
\beq\label{eq:wtcS-kl}
\E_n\left[\mu_n^{\be,0}(\cS_{k, \ell},\sfcK_\ell(I_n)=k')\right] 
\le q \vep \quad \text{ and } \quad \limsup_{n \to \infty}\mu_n^{\be,0}(\cS_{k', \ell}, \sK_n =k) =0.
\eeq
To this end, as $|S_{k, \ell}| \ge (1-q\vep) n$ on $\cS_{k,\ell}$, it is immediate that for any $k' \ne k$,
\[
 {\bf 1}(\cS_{k, \ell}) \frac{1}{n} \sum_{i=1}^n {\bf 1}(\sfcK_\ell(i) =k')  \le q \vep \,,
\]
which upon taking the expectation gives the \abbr{lhs} of \eqref{eq:wtcS-kl}. 
Next, ${\sf N}_{\red{\ell}}(i,k) \le 1$, hence on $\cS_{k', \ell}$ 
\[
\wh{\sf N}_\ell (k',k) := \frac{1}{n} \sum_{i=1}^n [ {\sf N}_\ell(i,k')  -  {\sf N}_\ell(i,k) ] 
\ge
\frac{1}{n} \sum_{i \in S_{k', \ell}} [ {\sf N}^{(1)}_\ell(i)  -  {\sf N}^{(2)}_\ell(i)] - \vep q
\ge (1-\vep q) \eta  -\vep q \ge \f{\eta}{2}\,
\]
(due to our choice of $\vep q <\eta/4<1/4$). Thus,
\beq\label{eq:wtcS-kl-1}
\mu_n^{\be,0}\big(\wh{\sf N}_\ell (k',k) {\bf 1}\{ \cS_{k', \ell}, \sK_n =k\} \big) \ge \frac{\eta}{2} \mu_n^{\be,0}(\cS_{k', \ell},\sK_n =k).
\eeq 
Moreover, note that for any $k' \ne k$,
\[
\wh{\sf N}_\ell (k',k) =\f1n \sum_{v=1}^n {\bf 1}\big( \sfB_v(2\ell) \cong \Tree_d(2\ell)\big) 
[ \delta_{\sigma_v, k'} -\delta_{\sigma_v,k} ] \,.
\]
Hence, with $\Graph_n \loc \Tree_d$, 
\beq\label{eq:Nl-conv}
\lim_{n \to \infty} \sup_{{\ul \sigma} \in [q]^n} \Big|\wh {\sf N}_\ell(k',k) - \wh {\sf N} (k',k) \Big| = 0,
\qquad \text{where} \qquad 
\wh {\sf N} (k',k) :=  \f1n \sum_{v=1}^n (\delta_{\sigma_v, k'}  - \delta_{\sigma_v, k}) \,.
\eeq
Clearly, $\wh{\sf N} (k',k) {\bf 1}(\sK_n =k) \le 0$, hence \eqref{eq:Nl-conv} entails that 
\beq\label{eq:wtcS-kl-2}
\limsup_{n \to \infty} \mu_n^{\be,0}\big(\wh{\sf N}_\ell (k',k) {\bf 1} \{\cS_{k', \ell}, \sK_n =k\} \big) \le 0.
\eeq
From \eqref{eq:wtcS-kl-1} and \eqref{eq:wtcS-kl-2} we get the \abbr{rhs} of \eqref{eq:wtcS-kl},
thus completing the proof of Theorem \ref{thm:main-2}(ii). 
\end{proof}

\section{The critical line: {proof of Theorem \ref{thm:crit-1}}}\label{sec:crit-line}

Recall Lemma \ref{lem:limit-point-properties}(a) on the existence of sub-sequential 
local weak limit points for both $\mu_n^{\be,B}$ and $\varphi_n^{\be,B}$ and that  
as in {\bf Step I} of the proof of Theorem \ref{thm:main-1}, fixing 
$d, q \ge 3$, $B>0$ and $(\be,B) \in \sfR_c$, it suffices to show that such limit 
point $\gm_{\bond}$ of $\varphi_n^{\be,B}$ is supported on $\cQ_{\rm m}:=\cM_{\bond}^{\be,B}$. 
Recalling 
Definition \ref{dfn:btwn-f-w} of $\cQ_\star(t)$ and Lemma \ref{lem:limit-point-properties}(d) 
that {such} $\gm_{\bond}$ must be supported on 
$\cQ_\star:=\{\varphi \in \cP(\B_\infty^\star): \varphi^t \in \cQ_\star(t) \text{ for all } t \in \N\}$,
we
show that $\gm_{\bond}(\cQ_\star \setminus \cQ_{\rm m}) >0$  
contradicts \cite[Theorem 1]{DMSS}. Indeed, Section \ref{subsec-rcm-messages}
identifies $\cQ_{\rm m}$ in terms of the support of the \abbr{rcm} messages, whereby
Section \ref{subsec-d-even} completes this argument by combining this characterization
with the symmetric form $\Psi^{\rm sym}(\cdot)$ of the free energy density limit 
of the \abbr{rcm} in case $d \in 2 \N$.

\subsection{Properties of the \abbr{rcm} messages}\label{subsec-rcm-messages}

We start with the definition of our \abbr{RCM} messages.
\begin{dfn}[\abbr{RCM} messages]\label{dfn:rcm-msg2}
Let $\sT^\star$ denote the tail $\sigma$-algebra of the bonds on $\Tree_d^\star$. That is, 
for the subsets $\B_t^\star := \{0,1\}^{E(\Tree_d^\star(t))}$ of $\B_\infty^\star
:=\{0,1\}^{E(\Tree_d^\star)}$,
\[
\sT^\star := \bigcap_{t \ge 1} \sT^\star_t\,, \quad \text{where} \quad 
\sT^\star_t := \sigma(\bigcup_{r > t} \sT_t^{\star,r}) \quad \text{and} \quad 
\sT_t^{\star,r}:=\sigma\left(\B_r^\star \setminus \B_t^\star \right) \,.
\]
For $(u,v) \in E(\Tree_d)$ and $\varphi \in \cQ_\star$, 
the
\abbr{rcm} message from $u$ to $v$ is then
\beq\label{eq:rcm-msg-dfn}
s_{u \to v} (\varphi) := \varphi [ u \bij v^\star | \eta_{(u,v)}=0, \sT^\star ].
\eeq
\end{dfn}

\begin{rmk}
While not needed here, one can define the infinite volume 
random cluster Gibbs measures on $\Tree_d$, in presence of an external field, via the so called 
Dobrushin-Lanford-Ruelle condition, as done in \cite[Section III]{BBCK00}, 
and show that the messages of \eqref{eq:rcm-msg-dfn}
are indeed those 
{that characterize such Gibbs measures}.
\end{rmk}  

The \abbr{rcm} messages are limits of the \abbr{rcm} pre-messages which we define next 
(and later relate to the local functions $\htgs^{t,r}_{u \to w}$ on the finite graphs $\Graph_n$). 
{To this end, for finite graph 
$\Graph$, $w \in V(\Graph)$ and $t<r$, we denote hereafter}
$\red{\ueta}_{w,t}^{\star,r}:=\ul \eta_{\sfB_w^\star(r)\setminus \sfB_w^\star(t)}$ {and 
$\red{\ueta}_{w,t}^r :=\ul \eta_{\sfB_w (r)\setminus \sfB_w (t)}$ with}   
$\ueta_t^{\star,r}:=\ueta_{o,t}^{\star,r}$ and $\ueta_t^{r}:=\ueta_{o,t}^{r}$
the corresponding objects on the tree $\Tree_d$ rooted at $o$. 
\begin{dfn}[\abbr{rcm} pre-messages]\label{dfn:rcm-msg}
For $(u,v) \in E(\Tree_d(t))$ consider 
the (local) function 
\beq
\nonumber
\hts_{u \to v}^{t,r}(\ul y):= \varphi_{\Tree^{\star}_d(r)}^{\be,B} (u \bij v^\star| \eta_{(u,v)}=0, 
 \ueta_{t}^{\star,r}=\ul y), \quad \text{ of } \quad \ul y \in \B_r^\star \setminus \B_t^\star\,.
\eeq
Similarly,  for finite graph $\Graph$, $w \in V(\Graph)$ and $(u,v) \in \sfB_w(t)$ let
\beq
\nonumber
 \htgs_{u \to v}^{t,r}(w, \ul y) := \varphi_{\sfB_w(r)}^{\be, B} (u \bij v^\star | \eta_{(u,v)}=0, \ueta_{w,t}^{\star,r}=\ul y), \quad \ul y \in \{0,1\}^{\sfB_w^\star(r)\setminus \sfB_w^\star(t)},
\eeq
using hereafter $s_{\partial v}(\varphi)$ for the 
vector $(s_{u \to v}(\varphi))_{u \in \partial v}$ and similarly 
$\hts_{\partial v}^{t,r}$ and $\htgs_{\partial v}^{t,r}:=(\htgs_{u \to v}^{t,r}(v, \cdot))_{u \in \partial v}$.
\end{dfn}

Clearly $\sT_t^{\star,r} \uparrow \sT^\star_t$ and therefore the 
Doob's martingale $\{ {\hts}_{u \to v}^{t,r}(\ueta_t^{\star,r})  \}_r$ 
converges $\varphi$-a.e.~to $\varphi [u \bij v^\star| \eta_{(u,v)}=0, 
\sT^\star_{t} ]$.  Further, 
$\sT^\star_t \downarrow \sT^\star$ and hence the latter backward martingale
converges $\varphi$-a.e.~to $s_{u \to v}(\varphi)$. 
In conclusion,  fixing any $\varphi \in \cQ_\star$,  for 
$\varphi$-a.e.~bond configuration, 
\begin{equation}\label{eq:lim-pre-msg}
s_{\partial v}(\varphi) = \lim_{t \to \infty} \lim_{r \to \infty} {\hts}_{\partial v}^{t,r}(\ueta_t^{\star,r} ) \,, \qquad 
\forall v \in V(\Tree_d) \,.
\end{equation}

Now, for $\ddagger \in \{\free, 1\}$ and
$\nu_\ddagger^{\be,B}$ of Definition \ref{dfn:bethe-recursion}, let 
\beq\label{eq:b-ddagger}
b_\ddagger=b_\ddagger^{\be, B}:= \frac{q \nu_\ddagger^{\be,B}(1)-1}{q-1} \ge 0 \,,
\eeq 
setting hereafter $b_{\wired}:=b_1^{\be,B}$, ${\gamma := (e^\beta -1)/(e^\beta+q-1)}$, and
\begin{equation}\label{def:BP-tilde}
\wh{\sf BP} (s_1,\ldots,s_{d-1}; x):=\frac{e^x\displaystyle{\prod_{i=1}^{d-1}} \big(1+(q-1) \gamma s_i\big) 
- \prod_{i=1}^{d-1} \big(1-\gamma s_i\big)}{e^x\displaystyle{\prod_{i=1}^{d-1}} \big(1+(q-1) \gamma s_i 
\big) + (q-1)\prod_{i=1}^{d-1} \big(1-\gamma s_i \big) }.
\end{equation}

Utilizing these objects, we proceed
to characterize $\cQ_{\rm m}$ via the support of the \abbr{RCM} messages. 
\begin{lem}
\label{lem:prop-rcm-msg}
Fix $B >0$,  $\be \ge 0$ and $\varphi \in \cQ_\star$.
\begin{enumeratea}
\item For  $\varphi$-a.e.~bond configurations and any $(u,v) \in E(\Tree_d)$,
\begin{align}\label{eq:bp-lim-msg}
s_{u \to v}(\varphi) = \wh{\sf BP}((& s_{w \to u}(\varphi))_{w \in \partial u \setminus \{v\}}; B), \\
\label{eq:lem-prop-rcm-msg1}
b_{\free} \le & s_{u \to v}(\varphi)  \le b_{\wired} \,.
\end{align}
\item For $\ddagger \in \{\free, \wired\}$ and any $(u,v) \in E(\Tree_d)$,
we have $s_{u \to v}(\varphi_\ddagger^{\be,B})=b_\ddagger$,  for 
$\varphi_\ddagger^{\be, B}$-a.e.~configurations.  
\item If $\varphi$-a.e.~the random vector $s_{\partial o}(\varphi)$ 
is supported on $\{(b_\ddagger, b_\ddagger, \ldots, b_\ddagger)\}_{\ddagger \in \{\free, \wired\}}$, 
then $\varphi \in \cQ_{{\rm m}}$. 
\end{enumeratea}
\end{lem}

\begin{proof}
(a).  For any $(i,j) \in E(\Tree_d)$ let $\Tree_{i \to j}$ be the connected component of the sub-tree of $\Tree_d$ rooted at $i$ after deleting the edge $(i,j)$. Set $\Tree_{i \to j}^\star$ to be the graph obtained 
from $\Tree_{i \to j}$ by adding the edges from $v^\star$ to
$V(\Tree_{i \to j})\setminus\{i\}$ (so there is {\em no} edge between $i$ and $v^\star$ in 
$\Tree_{i \to j}^\star$). Denoting by $\Omega_{i \to j}$ the event that $i$ is connected 
to $v^\star$ using {\em only} the open bonds within $\Tree_{i \to j}^\star$ and by
$\gp(i)$ the (unique) parent of $i \ne o$ (the root) in $\Tree_d$, let
\[
\wh \sT_t^{r} :=\sigma (\sT_t^{\star,r}, \left\{\Omega_{i \to \gp(i)}, i \in  \partial \Tree_d(r)\right\} ).
\]
By the one-to-one correspondence between the power set of $\gD_\star(r)$ (\red{recall Definition \ref{dfn:btwn-f-w}}) and
$\sigma(\{\Omega_{i \to \gp(i)}, i \in \partial \Tree_d(r)\})$, the 
$\sigma$-algebra $\wh\sT_t^r$ is generated by the finite partition 
of $(\B_{r}^\star \setminus \B_t^\star) \times \gD_\star(r)$ 
to pairs $(\ueta_t^{\star,r}, \sC^r)$, \red{using the notation $\sC^r$ for an element from $\gD_\star(r)$}.  Thus,  
for any $B>0$,  $\be \ge 0$, $\varphi \in \cQ_\star$ and $r>t$
\beq\label{eq:prop-rcm-msg-1}
\varphi(\cdot| \wh\sT_t^r)(\ueta_t^{\star,r}=\ul y, \sC^r =\sC)=\varphi_\sC^{\be,B,r}(\cdot|\ueta_t^{\star,r}=\ul y)\,, \qquad
\forall \ul y \in \B_{r}^\star \setminus \B_t^\star \,, \;\; 
\forall \sC \in \gD_\star(r) \,.
\eeq
In particular, 
\[
\varphi [u \bij v^\star| \eta_{(u,v)}=0,  \wh\sT_t^r ] = \ws_{u \to v}^{t,r}(\ueta_t^{\star,r},\sC^r) \,, 
\qquad \forall (u,v) \in E(\Tree_d(t))\,,
\]
where
\beq\label{dfn:ws-tr}
\ws_{u \to v}^{t,r}(\ul y,\sC):= 
\varphi_{\sC}^{\be,B,r}(u \bij v^\star| \eta_{(u,v)}=0, \ueta_t^{\star,r}=\ul y) \,.
\eeq
As $\wh \sT_t^r \uparrow \sT_t^\star$,  fixing any $\varphi \in \cQ_\star$,  we get  
similarly to \eqref{eq:lim-pre-msg},  that $\varphi$-a.e.
\beq\label{eq:ws-tr-lim}
s_{u \to v}(\varphi) 
=
\lim_{t \to \infty} \lim_{r \to \infty} \ws_{u \to v}^{t,r}(\ueta_t^{\star,r},\sC^r) \,,
\qquad \forall (u,v) \in E(\Tree_d) 
\eeq
and by the continuity of the mapping $\wh{\sf BP}(\cdot;B)$
it suffices for \eqref{eq:bp-lim-msg} to show that for any $t<r$,
\beq\label{eq:bp-pre-msg}
\ws^{t,r}_{u \to v} = \wh{\sf BP} ((\ws^{t,r}_{w \to u})_{w \in \partial u \setminus \{v\}}; B)\,,
\qquad \forall (u,v) \in E(\Tree_d(t)).
\eeq
By Lemma \ref{lem:coupling-es}, using an argument similar to the proof of Corollary \ref{cor:fin-rc-potts} 
we find that
\begin{align}\label{eq:varpi-bar-hts}
\bar s^{t,r}_{u \to v}(\ul y,\sC)
:=\varpi^{\be,B,r}_{\sC}(\sigma_u =1|\eta_{(u,v)}=0, \ueta_t^{\star,r}=\ul y) 
&=  \big(1 -\f1q\big) \ws^{t,r}_{u \to v} (\ul y,\sC) + \f1q \,,
\end{align}
for any $\ul y \in \B_{r}^\star \setminus \B_t^\star$, $\sC \in \gD_\star(r)$.
Fixing such $(u,v)$, $\ul y$, and $\sC$,  the identity \eqref{eq:bp-pre-msg} thus amounts to
\beq\label{eq:prop-rcm-msg-2}
\bar s^{t,r}_{u \to v} = \ol{\sf BP}(
(\bar s^{t,r}_{w \to u})_{w \in \partial u \setminus \{v\}}),
\eeq
where $\ol{\sf BP}(\bar s_1,\ldots,\bar s_{d-1}):=\big(1-\f1q\big) 
\wh{\sf BP}\Big(\frac{q \bar s_1-1}{q-1},\ldots, \frac{q \bar s_{d-1}-1}{q-1};B\Big) +\f1q$.  Alternatively,
in view of \eqref{def:BP-tilde},
\begin{equation}\label{def:hat-BP}
\ol{\sf BP}(\bar s_1,\ldots,\bar s_{d-1}) =
\frac{e^B\displaystyle{\prod_{i=1}^{d-1}} \left(1+ (e^{\be}-1) \bar s_i \right) }{e^B
\displaystyle{\prod_{i=1}^{d-1}} \left(1+ (e^{\be}-1) \bar s_i \right) + (q-1)
\displaystyle{\prod_{i=1}^{d-1}} \Big(1+ \frac{e^\beta -1}{q-1} (1-\bar s_i) \Big) }.
\end{equation}
The spin marginal of $\varpi^{\be,B,r}_{\sC}(\cdot|\eta_{(u,v)}=0, \ueta_t^{\star,r}=\ul y)$ 
on the sub-tree $\Tree_{u \to v} \cap \Tree_d(t)$ is a Potts measure.  Further, with $\sC \in \sD_\star(r)$,
the boundary condition for the spins at $\Tree_{u \to v} \cap \partial \Tree_{d}(t)$, 
as determined by $(\ul y, \sC)$, must be a product measure, where each marginal 
spin is either uniform over $[q]$ or supported on $1$. Thus, restricting this 
Potts measure to $u$ and $\partial u \setminus \{v\}$, yields a Potts measure 
on {$q$} spins,  whose boundary marginals on $\partial u \setminus \{v\}$ are 
mutually independent and uniform over $\{2,\ldots,q\}$. Thus, upon expressing the \abbr{rhs} of 
\eqref{eq:prop-rcm-msg-2} in terms of \eqref{def:hat-BP}, the identity \eqref{eq:prop-rcm-msg-2}
follows by a 
direct computation.

Setting $\wt b_{\ddagger,t'} (u,v) := \varphi_{\ddagger, t'}^{\be, B}(u \bij v^\star|\eta_{(u,v)}=0)$
for $\ddagger \in \{\free, \wired\}$, $(u,v) \in E(\Tree_d(t))$ and
$\varphi_{\ddagger,t'}^{\be,B}$ as in \eqref{eq:f-w-rc-t},  we show in part (b) that 
$\wt b_{\ddagger,t'}(u,v) \to b_{\ddagger}$ as $t' \to \infty$.  Thus,
in view of \eqref{dfn:ws-tr} and \eqref{eq:ws-tr-lim},  it suffices for 
\eqref{eq:lem-prop-rcm-msg1} to fix such $(u,v)$ and show that for
$r> t' > t$, 
any
$\ul y \in \B_{r}^\star \setminus \B_{t'}^\star$ and $\sC \in \gD_\star(r)$,
\begin{equation}\label{eq:stoch-order-t-r}
\wt b_{\free,t'} (u,v)
 \le
\varphi_{\sC}^{\be,B,r}(u \bij v^\star|\eta_{(u,v)}=0, \ueta_{t'}^{\star,r} = \ul y)
\le
\wt b_{\wired,t'} (u,v)
\,.
\end{equation}
The  conditional measure in \eqref{eq:stoch-order-t-r} is merely the \abbr{rcm}
on $\Tree_{u \to v} \cap \wh \Tree^\star_d(t')$ with the induced boundary condition 
$\sC' (\ul y,\sC) \in \gD_\star(t')$ at $\Tree_{u \to v} \cap \sfK^\star(\partial \Tree_d(t'))$.
Any such $\sC'$ stochastically dominates the free boundary condition and 
is stochastically dominated by the corresponding wired boundary condition (see
Definition \ref{dfn:inf-vol-RC}), with \eqref{eq:stoch-order-t-r} thus a consequence
of the monotonicity of the \abbr{RCM} 
$\varphi^{\be,B}_{\Tree_{u \to v} \cap \wh \Tree_d^\star(t')}(\cdot)$.

\noindent
(b). The \abbr{rcm}-s $\varphi_\ddagger^{\be, B}$, for $\ddagger \in \{\free, \wired\}$, are both 
tail trivial (indeed, this follows as in the proof of \cite[Theorem 4.19(d)]{G-RC}).
Moreover, by definition $\varphi^{\be,B}_\ddagger \in \cQ_\star$ are invariant 
under automorphisms of $\Tree_d$. Hence, 
$\varphi^{\be,B}_{\ddagger}$-a.s.~the values of $s_{u \to v}(\varphi^{\be, B}_\ddagger)$ for 
$(u,v) \in E(\Tree_d)$, must equal the same non-random constant 
$\wt b_\ddagger:=\varphi^{\be,B}_\ddagger (w \bij v^\star | \eta_{(w,o)}=0)$,  where $w \in \partial o$.
Recall \eqref{eq:f-w-rc-limit} that 
$\wt b_{\ddagger,t} (w,o) \to \wt b_\ddagger$ as $t \to \infty$.
Now, for any finite graph $\Graph$ and $e \in E(\Graph)$, the conditional measure $\varphi_{\Graph}^{\be, B}(\cdot| \eta_e=0)$ is merely the \abbr{RCM} 
$\varphi_{\wt \Graph}^{\be, B}(\cdot)$, for the subgraph $\wt \Graph$ obtained upon
deleting the edge $e$ from $\Graph$. Hence, it follows from Lemma \ref{lem:coupling-es} that
\beq\label{eq:wt-b-ddagger}
\wt b_\ddagger =\lim_{t \to \infty} \wt b_{\ddagger, t}  
=\lim_{t \to \infty}  \f{q \mu^{\be, B}_{w \to o, \ddagger, t}(\sigma_w =1)-1}{q-1},
\eeq
where $\mu^{\be, B}_{w \to o, \ddagger, t}$ denotes the Potts measure on 
$\Tree_{w \to o} \cap \Tree_d(t)$ with free and $1$-boundary conditions at the spins
of $\Tree_{w \to o} \cap \partial \Tree_d (t)$, for $\ddagger=\free$ and $\ddagger=\wired$, respectively. Since $\Tree_{w \to o}$ is a $(d-1)$-ary tree (i.e.~every vertex has $(d-1)$ children), it thus follows 
from Definition \ref{dfn:bethe-recursion} of $\nu_\ddagger^{\be,B}$, that the \abbr{RHS} of 
\eqref{eq:wt-b-ddagger} is precisely $b_\ddagger^{\be,B}$ of \eqref{eq:b-ddagger}, as claimed.

\noindent
(c). 
Setting $\sA_\ddagger := \{(s_{\partial o}(\varphi))
 = (b_\ddagger, b_\ddagger, \ldots, b_\ddagger)\}$ and fixing 
$\varphi \in \cQ_\star$ with $\varphi(\sA_{\free}\cup \sA_{\wired})=1$, our claim 
that $\varphi \in \cQ_{\rm m}$ amounts to having for any $t \in \N$ that
$\varphi$-a.e.  
\beq\label{eq:rcm-mixture}
\varphi(\ul \eta_{\sfB_o^\star(t)}=\cdot \mid  \sT^\star) =  \varphi_{\free}^{\be,B}(\ul \eta_{\sfB_o^\star(t)}=\cdot) {\bf 1}_{\sA_{\free}}+  \varphi_{\wired}^{\be,B}(\ul \eta_{\sfB_o^\star(t)}=\cdot) {\bf 1}_{\sA_{\wired}}\,.
\eeq
 To this end, fixing $t<t'<r$, recall that by \eqref{eq:prop-rcm-msg-1} 
\[
\varphi(\ul \eta_{\red{\tsfB}_o^\star(t)}=\ul y \mid \wh \sT_{t'}^r) =\sum_{\sC \in \gD_\star(t)} 
\varphi_{\sC}^{\be,B,t}(\ul \eta_{\sfB_o^\star(t)}=\ul y)\rho_\varphi (\sC|\wh \sT_{t'}^r), \qquad 
\forall \ul y \in \B_o^\star(t),
\]
with $\rho_\varphi (\cdot\mid \wh \sT_{t'}^r)$ the conditional distribution induced on $\gD_\star(t)$ 
by the finitely generated 
$\wh \sT_{t'}^r$.  Recall the one-to-one correspondence between 
$\gD_\star(t)$ and subgraphs $C_\star \subset \{(u,v^\star), u \in \partial \Tree_d(t)\}$, with
\[
\rho_\varphi (\{u \in C_\star\} \mid \wh \sT_{t'}^r) = 
\varphi \big[u \bij v^\star| \eta_{(u,\gp(u))}=0, \eta_{(u,v^\star)}=0, \wh\sT_{t'}^r \big]
=: \tilde s_{u \to \gp(u)}^{t',r}(\ueta_{t'}^{\star,r},\sC^r), \quad 
\forall u \in \partial \Tree_d(t) \,. 
\]
Using standard and backward martingale convergence theorems, it then follows that $\varphi$-a.e.
\begin{align}\label{eq:rcm-cond}
\varphi(\ul \eta_{\sfB_o^\star(t)}=\cdot \mid \sT^\star) &=\sum_{\sC \in \gD_\star(t)} \varphi_{\sC}^{\be,B,t}(\ul \eta_{\sfB_o^\star(t)}=\cdot)\rho_\varphi (\sC|\sT^\star),\, \\
\rho_\varphi (\{u \in C_\star \}\mid \sT^\star) &= \tilde s_{u \to \gp(u)}(\varphi) :=
\lim_{t' \to \infty} \lim_{r \to \infty} \tilde s_{u \to \gp(u)}^{t',r} \,.
\label{eq:rcm-mgnl}
\end{align}
In view of \eqref{eq:rcm-cond}, we get \eqref{eq:rcm-mixture} and thereby complete the proof, 
upon showing that $\varphi$-a.e.
\beq\label{eq:rcm-mix-t}
\rho_\varphi (\cdot|\sT^\star) = \rho_{\varphi_{\free}^{\be,B}} (\cdot) {\bf 1}_{\sA_{\free}}
+  \rho_{\varphi_{\wired}^{\be,B}}(\cdot) {\bf 1}_{\sA_{\wired}} \,.
\eeq
Turning to prove \eqref{eq:rcm-mix-t}, recall that $\varphi \in \cQ_\star$ and that 
the \abbr{RCM} $\varphi_{\wh \Tree_d^\star(t')}^{\be, B}$ is monotonic. Hence,
\[
\rho_{\varphi_{\free,t'}^{\be,B}}(\cdot
) \st \rho_{\varphi}(\cdot|\wh \sT_{t'}^r) \st \rho_{\varphi_{\wired,t'}^{\be,B}}(\cdot)\,,
\]
which upon taking first $r \to \infty$ and then $t' \to \infty$, yields that $\varphi$-a.e. 
\[
\rho_{\varphi_{\free}^{\be,B}}(\cdot) \st \rho_{\varphi}(\cdot|\sT^\star) \st \rho_{\varphi_{\wired}^{\be,B}}(\cdot)\,.
\]
Given such stochastic ordering, \eqref{eq:rcm-mix-t} follows once all the corresponding
one-dimensional marginals match (see \cite[Proposition 4.6]{G-RC}). That is, in view of
\eqref{eq:rcm-mgnl} and our assumption that  $\varphi(\sA_{\free}\cup \sA_{\wired})=1$,  
it remains only to show that $\varphi$-a.e. on the event $\sA_\ddagger$
\[
\tilde s_{u \to \gp(u)}(\varphi) = \tilde s_{u \to \gp(u)}(\varphi^{\be,B}_{\ddagger}) \,,
\qquad \forall u \in \Tree_d \,.
\]
To this end, note that as in the proof of \eqref{eq:bp-lim-msg}, we also have that $\varphi$-a.e.
\[ 
\tilde s_{u \to \gp(u)}(\varphi) = \wh{\sf BP}((s_{w\to u}(\varphi))_{w \in \partial u \setminus \{\gp(u)\}};0), 
\qquad \forall u \in \Tree_d \,.
\]
Recalling from part (b) that $s_{w\to u}(\varphi^{\be,B}_\ddagger)=b_\ddagger$ for all $(w,u) \in E(\Tree_d)$,
is thus suffices to verify that on the event $\sA_\ddagger$, also $s_{w\to u}(\varphi)=b_\ddagger$
(outside a null set). Indeed, the latter is a direct consequence of \eqref{eq:bp-lim-msg} and the 
identity
\[
b_\ddagger^{\be,B} = \wh{\sf BP}(b^{\be,B}_\ddagger,\ldots, b^{\be,B}_\ddagger;B) \,,
\]
which in view of \eqref{eq:b-ddagger}, is merely the fact that $\nu_{\free}^{\be,B}$ and 
$\nu_1^{\be,B}$ are both fixed points of 
\eqref{eq:BP-recursion}.
\end{proof}

\begin{rmk}\label{rmk:rcm-msg-bd-graph}
With $\wt b_{\wired,\ell}(w,o) \downarrow b_{\wired}$ and 
$\wt b_{\free,\ell} (w,o) \uparrow b_{\free}$ (for $w \in \partial o$), 
we deduce from \eqref{eq:stoch-order-t-r} that for some
finite $\kappa_o=\kappa_o(\vep)$ and any $r >t + \kappa_o$,
\beq
\nonumber
b_{\free} - \vep \le \hts_{u \to v}^{t+\kappa_o,r} (\ul y) \le b_{\wired}+\vep, \qquad \forall (u,v) \in 
E(\Tree_d(t)),  \;
\ul y  \in \B_r^\star \setminus \B_{t+\kappa_o}^\star \,.
\eeq
Moreover, if $w \in V(\Graph)$ for a finite graph $\Graph$ such that $\sfB_w(r) \cong \Tree_d(r)$
for $r > t + \kappa_o$, then
\beq\label{eq:lem-prop-rcm-msg3}
b_{\free} -\vep \le \htgs_{u \to v}^{t+\kappa_o,r}(w, \ul y) \le b_{\wired}+\vep, \quad \forall
(u,v) \in \sfB_w(t), \; \ul y \in \{0,1\}^{\sfB_w^\star(r)\setminus \sfB_w^\star(t+\kappa_o)}\,.
\eeq
\end{rmk}
Denoting hereafter by $\Graph^-(w)$ the graph $\Graph$ without $w \in V(\Graph)$ 
and all edges to it, we next show that for $B >0$, when $r \gg t$,
uniformly over $\Graph$
the law induced on $\sfB_{w}^\star(t)$ by the \abbr{RCM}  on $\Graph$ does 
not depend {\em much} on the boundary conditions outside $\sfB_{w}^\star(r)$.
\begin{lem}\label{l:decay.corel}
Fix $B >0$, $\be \ge 0$, $1< t < r$ and $o\in V(\Graph)$ for a finite graph $\Graph$. 
Denoting by $\sS_{t,r}^c$ \red{the complement of $\sS_{t,r}$,  namely,} 
the set of 
$\ueta_{o,t}^{\star,r}$ with an open cluster,  not containing $v^\star$,  that intersects both 
$\partial \sfB_o(t)$ and $\partial \sfB_o(r)$, 
\beq\label{eq:decay-corel1}
 {\varphi_\Graph^{\be, B}(\ul\eta_{\sfB_o^\star(t)}=\cdot \mid \ul \eta_{\Graph^\star 
 \setminus \sfB^\star_o(t)} = \ul z)}
 = {\varphi^{\be, B}_{\sfB_o(r)} (\ul \eta_{\sfB_o^\star(t)}=\cdot 
 \mid \ueta_{o,t}^{\star,r} = \ul z_{\sfB^\star_o(r) \setminus\sfB_o^\star(t)})},
\eeq
whenever $\ul z_{\sfB_o^\star(r)\setminus \sfB_o^\star(t)} \in \sS_{t,r}$.
Further, for any $\vep>0$, some $r_0=r_0(\vep,q,B,t, |\sfB_o(t)|)<\infty$,  
\begin{align}
\label{eq:decay-corel2}
\sup_{r \ge r_0} \sup_{\Graph,\ul y} \big\{ \varphi^{\be, B}_\Graph(\sS^c_{t,r} 
\mid \ul \eta_{\Graph \setminus \sfB_o(t)} = \ul y) \big\} & \le \vep , \\
\label{eq:decay-corel3}
\sup_{r \ge r_o} \max_{\Graph,\ul y} \Big|\frac{\varphi_{\Graph}^{\be, B}(\ul\eta_{\sfB_o(t)}=\cdot \mid \ul \eta_{\Graph \setminus \sfB_o(t)} = \ul y)}{\varphi^{\be, B}_{\sfB_o(r)} (\ul \eta_{\sfB_o(t)}=\cdot \mid \ueta_{o,t}^r = \ul y_{\sfB_o(r) \setminus \sfB_o(t)})}-1\Big|& \le \vep\,.
\end{align}
In addition, \eqref{eq:decay-corel2} holds for $\varphi^{\be, B}_{\Graph^-(o)}(\cdot)$.
\end{lem}
\begin{proof} Note that \eqref{eq:decay-corel1} amounts to having
\begin{equation}\label{eq:eta.ratio.bound}
 \sX_\star(\ul y, \ul {\wh y}):=\frac{\varphi_\Graph^{\be, B}(\ul y)}{\varphi_\Graph^{\be, B}(\ul {\wh y})}
 \frac{\varphi_{\sfB_o(r)}^{\be, B}(\ul {\wh y}_{\sfB_o^\star(r)})}{\varphi_{\sfB_o(r)}^{\be, B}
 (\ul y_{\sfB_o^\star(r)})} 
 =1,
\end{equation}
whenever $\ul y_{\sfB_o^\star(r)\setminus \sfB_o^\star(t)} \in \sS_{t,r}$ and
$y_e=\wh y_e$ for all $e \notin \sfB^\star_o(t)$.
Turning to prove
\eqref{eq:eta.ratio.bound}, recall by Definition \ref{dfn:rcm_dfn} of the \abbr{rcm} on a finite graph, that
\beq\label{eq:eta.ratio.bound1}
\sX_\star(\ul y, \ul {\wh y}) =  q^{|{{\sf C}_{\Graph^\star}(\ul y)}|  - |{\sf C}_{\Graph^\star}(\ul {\wh y})| + 
|{\sf C}_{\sfB_o^\star (r)}(\ul {\wh y}_{\sfB_o^\star(r)})| - |{\sf C}_{\sfB_o^\star(r)} (\ul y_{\sfB_o^\star(r)})|}.
\eeq
The open cluster containing $v^\star$ appears once in each of these four counts. Thus, 
with $y_e = \wh y_e$ for all $e \notin \sfB^\star_o(t)$, the 
\abbr{RHS} of \eqref{eq:eta.ratio.bound1} is one, unless some other open cluster induced by 
$\ul y$ or by $\ul {\wh y}$ intersects both $\partial \sfB_o(t)$ and $\partial \sfB_o(r)$,
a scenario which is precluded for $\ul {\wh y}_{\sfB_o^\star(r)\setminus \sfB_o^\star(t)} =
\ul y_{\sfB_o^\star(r)\setminus \sfB_o^\star(t)} \in \sS_{t,r}$.

As for \eqref{eq:decay-corel2}, we can assume \abbr{wlog} that 
{$\ul y_{\Graph \setminus \sfB_o(t)}$ induces} $L \ge 1$ disjoint open clusters $C_i$ 
that intersect both $\partial \sfB_o(t)$ and $\partial \sfB_o(r)$. In particular, 
$L \le |\sfB_o(t)|$ and $|C_i| \ge r-t$. Since $\sS_{t,r}^c$ implies that all
edges from $v^\star$ to some $C_i$ must be closed, we thus deduce by a union over $i \le L$,
that 
\[
 \varphi^{\be, B}_\Graph(\sS^c_{t,r}\mid \ul \eta_{\Graph \setminus \sfB_o(t)} = \ul y) \le 
 \sum_{i=1}^L q e^{-B |C_i|} \le q |\sfB_o(t)| e^{-B(r-t)},
\]
which since $B>0$, completes the proof for $\varphi_\Graph^{\be,B}$. {With $t>1$, the same applies for $\varphi^{\be,B}_{\Graph^-(o)}$.

Turning to prove \eqref{eq:decay-corel3}, we move to the marginal \abbr{rcm}-s
$\wt \varphi^{\be,B}_{\Graph} (\cdot)$ of \eqref{eq:mar-rcm} and note that
as in our proof of \eqref{eq:decay-corel1}, 
it suffices to find
$\vep_r=\vep_r(q,B,t,|\sfB_o(t)|) \to 0$, such that if $y_e=\wh y_e$ for all $e \notin \sfB_o(t)$, then
\begin{equation}\label{eq:ratio-tilde}
 \sX (\ul y, \ul {\wh y}):=\frac{\wt \varphi_\Graph^{\be, B}(\ul y)}{\wt \varphi_\Graph^{\be, B}(\ul {\wh y})}
 \frac{\wt \varphi_{\sfB_o(r)}^{\be, B}(\ul {\wh y}_{\sfB_o (r)})}{\wt \varphi_{\sfB_o(r)}^{\be, B}
 (\ul y_{\sfB_o(r)})} 
 \le 1 + \vep_r \,.
\end{equation}
To this end, since $\ul y$ and $\ul {\wh y}$ agree outside $\sfB_o(r)$, the four
products over edges in \eqref{eq:ratio-tilde} 
match and perfectly cancel each other, with $\sX(\ul y, \ul {\wh y})$ thereby 
being the ratio of product of contributions from open clusters.
Further, as $\ul y$ coincides with $\ul {\wh y}$ outside $\sfB_o(t)$, the collections 
${\sf C}_{\Graph}(\ul y)$ and ${\sf C}_{\Graph}(\ul {\wh y})$ contain exactly the 
same open clusters $C$ which do not intersect $\sfB_o(t)$, and this applies also for 
${\sf C}_{\sfB_o(r)} (\ul y_{\sfB_o(r)})$ versus ${\sf C}_{\sfB_o(r)} (\ul {\wh y}_{\sfB_o(r)})$.
Similarly, a cluster $C \subset \sfB_o(r-1)$ appears in ${\sf C}_{\Graph}(\ul y)$ if and only if
it appears in ${\sf C}_{\sfB_o(r)} (\ul y_{\sfB_o(r)})$ (and the same holds for 
$\ul{\wh y}$). Denoting by ${\sf C}_+$ those elements of 
${\sf C}_{\Graph}(\ul y) \cup {\sf C}_{\sfB_o(r)} (\ul {\wh y}_{\sfB_o(r)})$, 
counted twice if needed, which intersect both $\partial \sfB_o(t)$ and $\partial \sfB_o(r)$,
and by ${\sf C}_-$ all such elements in 
 ${\sf C}_{\Graph}(\ul {\wh y}) \cup {\sf C}_{\sfB_o(r)} (\ul y_{\sfB_o(r)})$, we thus deduce that
\[
\sX(\ul y, \ul {\wh y}) 
=\frac{\displaystyle{\prod_{C\in {\sf C}_+}} \big(1+(q-1)e^{-B|C|}\big)}
{\displaystyle{\prod_{C\in {\sf C}_-}} \big(1+(q-1) e^{-B|C|} \big)} 
\le (1+(q-1) e^{-B(r-t)})^{2|\sfB_o(t)|} =: 1+\vep_r 
\]
(at most $2 |\sfB_o(t)|$ elements in ${\sf C}_+$, each of at least size $r-t$,
and product over ${\sf C}_-$ exceeding one).}
\end{proof}


\subsection{The \abbr{rcm} free energy density for $d \in 2 \N$: 
proof of Theorem \ref{thm:crit-1}}\label{subsec-d-even} 

Recall \cite[Proposition~4.2]{DMSS} that for $(\beta,B) \in \sfR_c$ and $d  \in 2\N$
the Potts free energy density limit 
\[
\wh\Phi=\wh\Phi(\be,B) := \max\{ \Phi(\nu_{\free}^{\be,B}), \Phi(\nu_1^{\be,B}) \}  \,,
\]
has the symmetric \abbr{rcm} formulation, 
\begin{align}
\wh\Phi = \log (\Psi^{{\rm sym}}(b_{\free},\ldots,b_{\free})) 
=\log (\Psi^{{\rm sym}}(b_{\wired},\ldots,b_{\wired})) 
& = \sup_{\underline{b}\in[b_{\free},b_{\wired}]^d} 
\{\log (\Psi^{{\rm sym}}(\underline{b})) \}, 
\label{eq:wh-Phi}
\end{align}
where for the symmetric group $\sfS_d$ of size $d$,
\[
\Psi^{{\rm sym}}(\underline{b}):=\f{\Psi^{{\rm vx}}(\underline{b})}{\Psi^{{\rm e},{\rm sym}}(\underline{b})}, \qquad \Psi^{{\rm e},{\rm sym}}(\underline{b}):= \frac1{d!} \sum_{\pi \in \sfS_d} 
 \Psi^{{\rm e}}(b_{\pi(1)},\ldots,b_{\pi(d)}), 
\]
while for any $\ul b=(b_1,b_2,\ldots, b_d) \in [0,1]^d$ and $d \in 2\N$,
\begin{align*}
\Psi^{{\rm vx}}(\ul b)&:=(1-\gamma)^{-d} \Big(e^B\prod_{i=1}^d \big(1+
(q-1) \gamma b_i \big) + (q-1)\prod_{i=1}^d \big(1-\gamma b_i \big) \Big),
\quad
\gamma := \frac{e^\beta -1}{e^\beta+q-1},
\\
\Psi^{{\rm e}}(\ul b)&:= (1-\gamma)^{-d/2}  \prod_{i=1}^{d/2} 
\Big(1+(q-1) \gamma b_{2i-1}b_{2i}
\Big).
\end{align*}
We prove Theorem \ref{thm:crit-1} by passing \abbr{wlog} to a locally weakly convergent 
sub-sequence of $\varphi_{n}^{\be,B}$ and procuring from 
$\gm_{\bond}(\cQ_\star \setminus \cQ_{\rm m}) >0$ some graphs $\Graph'_n \loc 
\Tree_d$ of free energy densities such that
${\varlimsup_n} \; \Phi'_n(\be,B) > \wh \Phi$. To do so, we rely on our next observation, 
that the supremum in \eqref{eq:wh-Phi} over 
$\ul b
$ 
with at least 
two coordinates strictly inside $(b_{\free}, b_{\wired})$, 
is smaller than $\wh \Phi$. 
\begin{lem}\label{l:Psi.msup.Lambda}
For $(\be,B) \in \sfR_c$ and any $\delta>0$ there exists $\vep_o\in (0, \delta)$ such that
\[
\sup_{\underline{b}\in \Lambda_{\delta}} \{ \log (\Psi^{{\rm sym}}(\underline{b})) \} < \wh\Phi-\vep_o,
\]
where
\beq\label{def:Lambda-del}
\Lambda_\delta:=\left\{\ul b\in[b_{\free},b_{\wired}]^d: \left|\left\{i: b_i\in[b_{\free}+\delta,b_{\wired}-\delta]\right\}\right|\geq  2\right\}.
\eeq
\end{lem}
\begin{proof}
Since $\Lambda_\delta$ is compact and $\Psi^{{\rm sym}}(\cdot)$ is continuous, the supremum over $\Lambda_\delta$ is achieved at some $\underline{b}'\in \Lambda_\delta$ {and it suffices} to 
show that $\log (\Psi^{{\rm sym}}(\underline{b}'))< \wh\Phi$. {Now, suppose that}
$\underline{b}^\circ\in[b_{\free},b_{\wired}]^d$ is such that $\log (\Psi^{{\rm sym}}(\underline{b}^\circ))= \wh\Phi$ and $b_i^\circ\in (b_{\free},b_{\wired})$ for some $i \in [d]$.  Since both $\Phi^{{\rm vx}}$ and $\Phi^{{\rm e}, {\rm sym}}$ are affine in each $b_j$, $j \in [d]$, and $\underline{b}^\circ$ is a maximizer, the function must be constant as we vary the $i$-th coordinate of $\ul b^\circ$.  This, in particular implies that if we replace $b_i^\circ$ with either $b_{\free}$ or $b_{\wired}$ then the new vector continues be a maximizer of $\log \Psi^{{\rm sym}}(\cdot)$. {In particular, if}
$\log (\Psi^{{\rm sym}}(\underline{b}'))= \wh\Phi$ for $\ul b'$ {with at least} two coordinates $i\ne j \in [d]$ such that $b'_i,b'_j \in (b_{\free}, b_{\wired})$, {then} by the preceding argument there exists 
$\ul b'' \in \{b_{\free}, b_{\wired}\}^d$, still a maximizer, with the property $2 \le |\{i: b''_i =b_{\free}\}| \le d-2$.   
%
However, the proof of \cite[Proposition~4.2]{DMSS} shows that such a $\underline{b}''$ can never be a maximizer, yielding a contradiction. 
\end{proof}

Combining Lemma \ref{lem:prop-rcm-msg}(c) with Lemma \ref{l:Psi.msup.Lambda}, we proceed to show that  if a local weak limit point of \abbr{RCM}-s
puts a positive mass on $\cQ_\star \setminus \cQ_{\rm m}$ then it must also put a positive mass on those $\varphi \in \cQ_\star$ for which 
$\log \Psi^{{\rm sym}}(\hts_{\partial o}^{t,r})$ is strictly smaller than $\wh \Phi$,  for large $t < r$, with non-negligible probability.}  
\begin{lem}\label{l:s.in.Lambda}
For $\vep >0 $ and $ t < r \in \N$ set 
\beq\label{dfn:F-vep-tr}
\sB_\vep^{t,r}:= \left\{ \log \Psi^{{\rm sym}}(\hts_{\partial o}^{t,r}) < \wh \Phi -\vep \right\} \,.
\eeq
If $\gm_{\bond}(\cQ_\star \setminus \cQ_{{\rm m}}) >0$ for a local weak limit point $\gm_{\bond}$ 
of $\{\varphi_n^{\be,B}\}$,  then there exist $\xi >0$,  $t' \in \N$, and $r'(t) < \infty$, 
such that $\int \varphi (\sB_\xi^{t,r}) d \gm_{\bond} \ge \xi$ for any $t \ge t'$ and $r \ge r'(t)$.
\end{lem}
\begin{proof}
With $\gm_{\bond}$ supported on the convex set $\cQ_\star$ (see Lemma \ref{lem:limit-point-properties}(d)), also $\bar \varphi(\cdot):= \int \varphi (\cdot) d\gm_{\bond} \in \cQ_\star$.  Further,  
for any non-negative $F \in C_b(\R^d)$ and $\varphi \in \cQ_\star$,  it follows from \eqref{eq:lim-pre-msg} that
\[
\lim_{t \to \infty}\lim_{r \to \infty} \varphi(F(\hts_{\partial o}^{t,r})) = \varphi(F(s_{\partial o}(\varphi))) \,,
\]
hence by definition of $\bar \varphi$ and Fatou's lemma,
\beq\label{eq:s-in-L0}
\bar \varphi(F(s_{\partial o}(\bar \varphi))) = 
 \lim_{t \to \infty}\lim_{r \to \infty} \bar \varphi(F(\hts_{\partial o}^{t,r}))= 
\lim_{t \to \infty} \lim_{r \to \infty} \int  \varphi(F(\hts_{\partial o}^{t,r})) d \gm_{\bond} 
\ge  \int  \varphi(F(s_{\partial o}(\varphi)) d \gm_{\bond} \,.
\eeq
Considering $F(\ul s) := \sum_{i=1}^d f_\vep(s_i)$ for continuous $f_\vep$ 
such that $0 \le f_\vep(\cdot) \uparrow {\bf 1}_{(b_{\free}, b_{\wired})}(\cdot)$ as $\vep \downarrow 0$, 
we deduce from \eqref{eq:s-in-L0} by monotone convergence {and Fatou's lemma},  that 
in view of \eqref{eq:lem-prop-rcm-msg1},
\beq\label{ineq-H0}
\sum_{u \in \partial o} \bar \varphi(s_{u \to o}(\bar \varphi) \ne \{b_{\free},b_{\wired}\}) 
\ge
\int \sum_{u \in \partial o} \varphi(s_{u \to o}(\varphi) \ne \{b_{\free},b_{\wired}\} ) d \gm_{\bond}
\ge \int \varphi( (\sA_{\free}\cup \sA_{\wired})^c ) d \gm_{\bond} \,.
\eeq
Recall Lemma \ref{lem:prop-rcm-msg}(c) that 
$\varphi( (\sA_{\free}\cup \sA_{\wired})^c )>0$ whenever $\varphi \in \cQ_\star \setminus \cQ_{\rm m}$.
Hence,  our assumption that $\gm_{\bond}(\cQ_\star \setminus \cQ_{{\rm m}})>0$ implies the strict positivity of the 
\abbr{rhs} of \eqref{ineq-H0} and thereby {the same holds for the \abbr{LHS}}. 
By 
definition of the local weak convergence,  the law $\gm_{\bond}$ must be invariant to our 
choice of the root of $\Tree_d$.  Hence,  with $\Tree_d$ a regular tree,  necessarily 
$\bar \varphi$ is invariant under any automorphism of $\Tree_d$.  In particular,
fixing $u_1 \in \partial o$ and setting  
\[
H_\vep:=\big\{\min_{\ddagger \in \{\free, \wired\}}|s_{u_1 \to o}(\bar \varphi) - b_{\ddagger}| > \vep\big\},
\]
we deduce that $\bar \varphi(H_0)>0$.  This in turn allows us to fix $\vep>0$ small enough so that $\bar\varphi(H_\vep) \ge \vep$. Next,  fix $\ell:=\lceil 3/\vep \rceil$ and a non-random path $o=:v_1,v_2, \ldots,v_{\ell+1}$ in $\Tree_d$
such that $v_{i+1} \in \partial \Tree_d(i)$.  Then,  fixing for each
$1 \le i \le \ell$,  some $w_{i} \in \partial v_i$,  $w_i \neq v_{i+1}$ (which as $d \ge 3$,  is always possible), 
we have by the automorphism invariance of $\bar\varphi$ that 
\beq\label{eq:lbd-pAi}
\bar \varphi(A_i)=\bar \varphi({\sf H}_\vep) \ge \vep \ge \frac{3}{\ell} \quad 
\text {for } \quad 
A_i:=\Big\{\min_{\ddagger \in \{\free, \wired\}} |s_{w_{i}\to v_i}(\bar \varphi)-b_\ddagger|>\vep\Big\} \,.
\eeq
The smooth map $\wh{\sf BP}(\cdot;B):[0,1]^{d-1} \to [0,1]$ of \eqref{def:BP-tilde} 
is coordinate-wise strictly increasing, having $(b_{\ddagger},\ldots,b_{\ddagger})$
as its fixed points (see \eqref{eq:bp-lim-msg} and Lemma \ref{lem:prop-rcm-msg}(b)).
Thus,  for some  
$\delta_0=\delta_0(\vep, \be, B,q,d)>0$,  it follows from \eqref{eq:bp-lim-msg} and
\eqref{eq:lem-prop-rcm-msg1} that 
$s_{v_{j}\to v_{j-1}}(\bar\varphi) \in (b_{\free}+\delta_0,b_{\wired}-\delta_0)$ on the event $A_j$.
Iterating this argument,  while reducing $\delta_0$ to 
$\delta=\delta(\delta_0,\ell,\be, B,q,d) \in (0, \vep)$,  it 
follows that on each $A_j$,  $j \le \ell$, 
\[
\{s_{v_{i+1}\to v_{i}} (\bar \varphi),  1 \le i < j \} \subset (b_{\free}+\delta, b_{\wired}-\delta) \,.
\]
So,  if $A_{i} \cap A_{j}$ occurs for $i<j$,  then
$s_{v_{i+1}\to v_{i}}(\bar\varphi),s_{w_i \to v_{i}}(\bar \varphi) \in (b_{\free}+\delta, b_{\wired}-\delta)$ 
and in particular $s_{\partial v_{i}} \in \Lambda_{\delta}$ 
of 
\eqref{def:Lambda-del}.  Namely,  any pair from $\{A_1,\ldots, A_\ell\}$ results with
$s_{\partial v_{i}}(\bar \varphi) \in \Lambda_{\delta}$ for some $i < \ell$.  Consequently,
by the union bound and the automorphism invariance property of $\bar \varphi$,
\begin{align}
\ell \bar\varphi(s_{\partial o}(\bar \varphi) \in \Lambda_{\delta}) &\ge 
\bar\varphi \big(
\bigcup_{i=1}^\ell \{ s_{\partial v_{i}}(\bar \varphi) \in \Lambda_{\delta} \} \big)
\geq \bar\varphi\big(\sum_{i=1}^\ell \red{\bf 1}_{A_i} \geq 2\big) \geq 
\frac{1}{\ell} \bar\varphi\big(\sum_{i=1}^\ell \red{\bf 1}_{A_i} - 2\big) \geq \frac{1}{\ell}
\label{eq:s-in-L3}
\end{align}
(where the right-most inequality is due to \eqref{eq:lbd-pAi}).  In view of Lemma \ref{l:Psi.msup.Lambda}
we deduce from \eqref{eq:s-in-L3} that for some $\vep_o>0$,
\[
 \bar \varphi(\log \Psi^{{\rm sym}}(s_{\partial o}(\bar \varphi)) < \wh \Phi -\vep_o) \ge \frac{1}{\ell^2}\,.
\]
With $\log \Psi^{{\rm sym}} \in C_b([0,1]^d)$,  it follows from \eqref{eq:lim-pre-msg} that 
$\bar\varphi$-a.e. ~$\log \Psi^{{\rm sym}}(\hts_{\partial o}^{t,r}) \to \log \Psi^{{\rm sym}}(s_{\partial o}(\bar \varphi))$ as $r \to \infty$ and then $t \to \infty$.  
Thus,  by the preceding 
\[
\lim_{t \to \infty} \lim_{r \to \infty} \bar\varphi(\sB_{\vep_o}^{t,r}) \ge \frac{1}{\ell^2} \,,
\]
yielding our claim for $\xi=\vep_o \wedge (2 \ell^2)^{-1}$.
\end{proof}

To prove Theorem \ref{thm:crit-1} we keep \red{removing} from a finite graph $\Graph$
some $w \in V(\Graph)$ \red{such that $\sfB_w(r) \cong \Tree_d(r)$ (for some $r \ge 2$),
and all the edges to $w$,  to get the} graph $\Graph^-(w)$, 
thereby re-connecting the neighbors $x_1, x_2,\ldots,x_d$ of $w$ to form
for $\pi \in \sfS_d$ the graph $\Graph^\pi(w)$ as 
$\Graph^-(w)$ with the additional edges $(x_{\pi(2i-1)},x_{\pi(2i)})$, $i=1\ldots,d/2$.
Key to the proof is thus to estimate
the ratios of contributions to the partition functions from \abbr{RCM}-s 
on $\Graph$ versus $\Graph^-(w)$, and on $\Graph^\pi(w)$ versus $\Graph^-(w)$, when
fixing certain bonds. To this end, recall $\be_e^\star$ of 
\eqref{dfn:Potts-star} and $p_e$, ${\sf C}(\cdot)$ of Definition \ref{dfn:rcm_dfn}, where the 
un-normalized Edwards-Sokal probability mass
\beq\label{eq:wh-varpi}
\wh \varpi_{\Graph^\star}^{\be,B}(\ul \sigma, \ul \eta) := \prod_{e \in E^\star} 
e^{\be^\star_e}\left[ (1-p_e)({1-\eta_e}) + p_e {\eta_e} \delta_e(\ul{\sigma}) \right] \cdot  \delta_{\sigma_{v^\star},1},  \quad \ul \sigma \in [q]^{V^\star},\;  \ul \eta \in \{0,1\}^{E^\star}
\eeq
(see Definition \ref{dfn:es_dfn}), has total mass matching the Potts partition function of $\Graph$,
as in \eqref{dfn:Potts-star}.  Namely, 
\beq\label{eq:wh-varpin}
 \sum_{\ul \sigma, \ul \eta}
\wh \varpi_{\Graph^\star}^{\be,B}(\ul \sigma, \ul \eta) = Z_{\Graph}(\be,B) \,.
\eeq
More generally, restricting \eqref{eq:wh-varpi} to bond values 
$\ul y_{W^\star} \in \{0,1\}^{E(W^\star)}$ on the edges of $W \subset \Graph$
(or, to only $\ul y_{W} \in \{0,1\}^{E(W)}$ on non-ghosted edges), 
yields the {\abbr{rcm}-restricted} partition functions
\begin{align}\label{eq:cZ-es}
\cZ_{\Graph, W^\star}(\ul y_{W^\star}) &:= {\frac{1}{q}}
 \sum_{\ul \eta: \ul \eta_{W^\star}= \ul y_{W^\star}} 
q^{|{\sf C}(\ul \eta)|} \prod_{e \in E^\star} e^{\be^\star_e} p_e^{\eta_e}(1-p_e)^{1-\eta_e} 
= \sum_{\ul \eta: \ul \eta_{W^\star} =\ul y_{W^\star}} 
\sum_{\ul \sigma} \wh \varpi_{\Graph^\star}^{\be,B}(\ul \sigma, \ul \eta) \,,\\
\label{eq:cZ-def}
\cZ_{\Graph, W}(\ul y_W)&:= {\frac{1}{q}}
\sum_{\ul \eta: \ul \eta_{W}= \ul y_{W}} 
q^{|{\sf C}(\ul \eta)|} \prod_{e \in E^\star} e^{\be^\star_e} p_e^{\eta_e}(1-p_e)^{1-\eta_e} 
= \sum_{\ul \eta: \ul \eta_{W} =\ul y_{W}} 
\sum_{\ul \sigma} \wh \varpi_{\Graph^\star}^{\be,B}(\ul \sigma, \ul \eta) \,.
\end{align}
For any $v \in V(\Graph)$ with $\partial v = (u_i)_{i=1}^d$, $\pi \in \sfS_d$ and
$\ul y \in \{0,1\}^{\sfB_v(r)\setminus \sfB_v(t)}$, $t<r$, we set 
 the functions 
\begin{align}\label{eq:bar-Psi1}
\Psi^{{\rm vx}}_{t,r,\Graph}(v, \ul y) &:= \varphi^{\be,B}_{\Graph^-(v)}
 \big[ \Psi^{{\rm vx}
 } (\red{\mathfrak{s}}^{t,r}_{\, \partial v}) \mid \ueta_{v,t}^r = \ul y \, \big],
  \quad
 \Psi^{{\rm e}, {\rm sym}}_{t,r,\Graph}
 (v, \ul y) 
  := \frac{1}{d!} \sum_{\pi \in \sfS_d}
  \Psi^{{\rm e}, \pi}_{t,r,\Graph}(v, \ul y) \,, \\
\Psi^{{\rm e}, \pi}_{t,r,\Graph}(v, \ul y)
&:=  \varphi^{\be, B}_{{\red{\tGraph}^-(v)}} \big[ 
\Psi^{{\rm e}}(\htgs^{t,r}_{u_{\pi(1)} \to v},\ldots,\htgs^{t,r}_{u_{\pi(d)} \to v}) \mid \ueta_{v,t}^r = \ul y \, \big].
\nonumber
\end{align}
Utilizing Lemma \ref{l:decay.corel} and \eqref{eq:wh-varpi}-\eqref{eq:cZ-def}, we proceed to
estimate the ratios of various {\abbr{rcm}-restricted} partition functions, in terms of these functions.
\begin{lem}\label{l:z.ratio.estimate}
Fix $B >0$ and $\be \ge 0$. \red{Fix $\vep \in (0,1)$, $t \in \N$, and a finite graph $\wh \Graph$. There exists some  
$r_o'=r_o'(\vep, t)$ such that} if
$o\in V(\wh \Graph)$ and $\sfB_o(r)\cong \Tree_d(r)$, $r \ge r_o'$, then 
for any $\sfB_o(r) \subset \Graph \subset \wh \Graph$, we get upon setting 
$\wh W:=\Graph \setminus \sfB_o(t)$ and $W_r:=\sfB_o(r) \setminus \sfB_o(t)$, that 
\begin{align}\label{eq:z.ratio.estimate1}
 \max_{\ul y_{\wh W}} \Big\{\big| \frac{\cZ_{\Graph,\wh W}(\ul y_{\wh W})}{
 \cZ_{\Graph^-(o),\wh W}(\ul y_{\wh W})} - 
 \Psi^{{\rm vx}}_{t,r,\wh \Graph}(o, \ul y_{W_r})  \big| \Big\} &\le \vep \,, \\
\max_{\pi \in \red{\sfS}_d} \max_{\ul y_{\wh W}} \Big\{
\big| \frac{\cZ_{\Graph^\pi(o),\wh W}(\ul y_{\wh W})}{\cZ_{\Graph^-(o),\wh W}(\ul y_{\wh W})} - 
\Psi^{{\rm e}, \pi}_{t,r,\wh \Graph}(o, \ul y_{W_r}) \big| \Big\} &\leq \vep.
\label{eq:z.ratio.estimate2}
\end{align}
\end{lem}
\begin{rmk}\label{rmk:z.ratio}
In view of \eqref{eq:cZ-def}, the bounds
\eqref{eq:z.ratio.estimate1}-\eqref{eq:z.ratio.estimate2} actually hold also for any  
subgraph $\wh{W}$ of $\Graph \setminus \sfB_o(t)$ which contains $\sfB_o(r) \setminus \sfB_o(t)$.
\end{rmk}
\begin{proof}  With $\Graph \cap \sfB_o(r) = \wh \Graph \cap \sfB_o(r) \cong \Tree_d(r)$ 
and $W_r = \wh W \cap \sfB_o(r)$, we first examine the special 
case $\Graph = \wh \Graph = \sfB_o(r)$. Specifically, using $\sfB^-_o(r):=\Graph^-(o)$ and 
$\sfB^\pi_o(r):=\Graph^\pi(o)$ when $\Graph=\sfB_o(r)$, we show that for any 
$\ul y \in \B_r^\star \setminus \B_t^\star$ and $\pi \in  \red{\sfS}_d$, 
\begin{align}\label{eq:claim-cZ}
\wh \cZ_r(\ul y)&:=\f{\cZ_{\sfB_o(r), W_r^\star}(\ul y)}{\cZ_{\sfB^-_o(r), W_r^\star}(\ul y)} = \Psi^{{\rm vx}}(\htgs_{\partial o}^{t,r}(\ul y)), \\
\wh \cZ^\pi_r(\ul y)&:=\f{\cZ_{\sfB^\pi_o(r), W_r^\star}(\ul y)}{\cZ_{\sfB^-_o(r), W_r^\star}(\ul y)} =
\Psi^{{\rm e}}(\htgs^{t,r}_{u_{\pi(1)} \to o}(\ul y),\ldots,\htgs^{t,r}_{u_{\pi(d)} \to o}(\ul y)).
\label{eq:claim-cZ-pi}
\end{align}
Indeed, it follows 
from \eqref{eq:cZ-es} that 
\begin{align*}
\cZ_{\sfB^-_o(r), W_r^\star}(\ul y) = \!\!\!\!\!\!\!
\sum_{\{\ul \eta' : \ueta_{o,t}^{\star,r} = \ul y\}} 
\sum_{\ul \sigma'} \wh \varpi_{\sfB^-_o(r)^\star}^{\be,B}(\ul \sigma', \ul \eta').
\end{align*}
Further, for any $e \in E^\star$ and $\ul \sigma \in [q]^{V^\star}$,
\[
\sum_{\eta_e \in \{0,1\}} e^{\be^\star_e} \left[ (1-p_e)({1-\eta_e}) + p_e {\eta_e} \delta_e(\ul{\sigma})  
\right] 
= 1 + (e^{\be^\star_e} -1 ) \delta_e(\ul \sigma)  \,.
\]
Thus, with
$\ul \eta_{\sfB_o^\star(r)} = (\ul \eta^{(1)}, \ul \eta')$ for
$\ul \eta^{(1)} :=(\eta_e)_{e \in \sfB_o(1) \cup \{(o,v^\star)\}}$,
we similarly get from \eqref{eq:wh-varpi}-\eqref{eq:cZ-es} that 
\begin{align*}
\cZ_{\sfB_o(r), W_r^\star}(\ul y) &= \!\!\!
\sum_{\{\ul \eta' : \ul \eta'_{W_r^\star}= \ul y\}} 
\sum_{\ul \sigma'} \sum_{\sigma_o,\ul \eta^{(1)}} \wh \varpi_{\sfB^\star_o(r)}^{\be,B}((\sigma_o,\ul \sigma'), (\ul \eta^{(1)},\ul \eta'))
= \!\!\!
\sum_{\{\ul \eta' : \ueta_{o,t}^{\star,r} = \ul y\}} 
\sum_{\ul \sigma'} 
f(\sigma'_{\partial o}) \wh \varpi_{\sfB^-_o(r)^\star}^{\be,B}(\ul \sigma', \ul \eta'),
\end{align*}
for some function $f(\cdot)$, {where by} enumerating over the $[q]$-valued $\sigma_o$,  
{one verifies that}
\[
f(\sigma_1,\ldots,\sigma_d)=\sum_{k=1}^q e^{B\delta_{k,1}}\prod_{i=1}^d 
\big(1+(e^{\be}-1) \delta_{\sigma_i,k} \big) \,.
\] 
Consequently,  
\beq\label{eq:claim-cZ-1}
\wh \cZ_r(\ul y) = 
\varpi_{\sfB^-_o(r)^\star}^{\be,B}
 \Big[f(\sigma_{\partial o})\big| \ueta_{o,t}^{\star,r}= \ul y \Big].
\eeq
By the same reasoning, now with $\ul \eta^{(1)} := \ul \eta_{\sfB_o^\pi(1) \setminus \sfB_o^-(1)}$,
we find that, for any $\pi \in  \red{\sfS}_d$, 
\begin{align}\label{eq:claim-cZpi-1}
 \wh \cZ^\pi_r(\ul y) &= 
\varpi_{\sfB^-_o(r)^\star}^{\be,B}
 \Big[f^\pi(\sigma_{\partial o})\big| \ueta_{o,t}^{\star,r} = \ul y \Big],\\
f^\pi(\sigma_1,\ldots,\sigma_d) & :=\prod_{i=1}^{d/2} 
\big(1+(e^{\be}-1) \delta_{\sigma_{\pi(2i-1)},\sigma_{\pi(2i)}} \big) \,. \notag
\end{align}
Since $\sfB_o(r) \cong \Tree_d(r)$, 
the spins $(\sigma_u, u \in \partial o)$ are mutually
independent under $\varpi_{\sfB^-_o(r)^\star}^{\be,B}(\cdot \mid \ueta_{o,t}^{\star,r})$, with 
each of them uniformly distributed on $\{2,\ldots,q\}$. Thus, upon setting for $i=1,\ldots,d$,
\[
b_i := \frac{1}{1-1/q} \varpi_{\sfB^-_o(r)^\star}^{\be,B}
 \big(\sigma_{u_i}=1\big| \ueta_{o,t}^{\star,r}= \ul y
 \big) - \frac{1/q}{1-1/q}  \,,
\]
a straightforward computation shows that the \abbr{rhs} of \eqref{eq:claim-cZ-1} 
and \eqref{eq:claim-cZpi-1} are given by $\Psi^{\rm vx} (\ul b)$ and 
$\Psi^{{\rm e}}(b_{\pi(1)},\ldots,b_{\pi(d)})$, respectively.
Finally, the identification $\ul b = \htgs^{t,r}_{\partial o}$ follows from Lemma \ref{lem:coupling-es},
analogously to the derivation of \eqref{eq:varpi-bar-hts}.

Armed with \eqref{eq:claim-cZ}, we proceed to prove \eqref{eq:z.ratio.estimate1}.
To this end, writing $\ul y = \ul y_{W^\star}$ and $\ul y^o = \ul y_{W}$, we
get by following the derivation of \eqref{eq:claim-cZ-1},
that for any $\Graph$ with $|\partial o|=d$ and $W_r \subset W \subset \Graph \setminus \sfB_o(t)$,
\begin{align*}
\f{\cZ_{\Graph,W^\star}(\ul y)}{\cZ_{\Graph^-(o),W^\star}(
\ul  y)} & = 
\varpi_{\Graph^- (o)^\star}^{\be,B}
 \Big[f(\sigma_{\partial o})\big| \ul \eta_{W^\star}= \ul y \Big] \text{ and  } 
 \f{\cZ_{\Graph,W}(\ul y^o)}{\cZ_{\Graph^-(o),W}(
\ul  y^o)}
= 
\varpi_{\Graph^- (o)^\star}^{\be,B}
 \Big[f(\sigma_{\partial o})\big| \ul \eta_{W}= \ul y^o\Big] \,.
\end{align*}
Consequently, 
\begin{align}\label{eq:W-to-W-star}
\frac{\cZ_{\Graph,W}(\ul y^o)}{\cZ_{\Graph^-(o),W}(\ul y^o)} 
& =  \sum_{\ul y \setminus \ul y^o}  
\f{\cZ_{\Graph, W^\star}(\ul y)}{\cZ_{\Graph^-(o), W^\star}(
\ul  y)}  \varphi^{\be,B}_{\Graph^-(o)}({\ul \eta_{ W^\star}=\ul y} \mid {\ul \eta_{W}=\ul y^o}).
\end{align}
Further, from \eqref{eq:wh-varpi}-\eqref{eq:cZ-es} we have that for 
some $\bar c = \bar c (q,\be,B,d) >0$,
\begin{equation}\label{eq:vertex.removal.Z}
\frac{\cZ_{\Graph^-(o),{W}^\star}(\ul y)}{\cZ_{\Graph, W^\star}(\ul y)}
= \frac{1}{q} \varphi^{\be,B}_{\Graph}(\ul\eta^{(1)} 
\equiv 0\mid \ul \eta_{W^\star}=\ul y) \ge \bar c \,.
\end{equation}
Comparing \eqref{eq:claim-cZ} with \eqref{eq:vertex.removal.Z} for $(\sfB_o(r),W_r)$,
we deduce that
\beq\label{eq:claim-cZ-ratio}
\frac{\cZ_{\Graph,W^\star}(\ul y)}{\cZ_{\Graph^-(o),{W}^\star}
(\ul y)} =
\frac{\varphi^{\be,B}_{\sfB_o(r)}(\ul\eta^{(1)}\equiv 0\mid \ueta_{o,t}^{\star,r}=\ul y_{W_r^\star})}{\varphi^{\be,B}_{\Graph}(\ul\eta^{(1)}\equiv 0\mid {\ul \eta_{W^\star}=\ul y})} \Psi^{{\rm vx}}(\htgs_{\partial o}^{t,r} (\ul y_{W_r^\star})) \,,
\eeq
which, in view of \eqref{eq:decay-corel1} equals to 
$\Psi^{{\rm vx}}(\htgs_{\partial o}^{t,r} (\ul y_{W_r^\star}))$ whenever $\ul y_{W_r^\star} \in \sS_{t,r}$.
Utilizing this observation when plugging \eqref{eq:claim-cZ-ratio} into \eqref{eq:W-to-W-star}, then 
using the uniform lower bound of \eqref{eq:vertex.removal.Z} and the triangle inequality, we find that
for $r \ge r_0(\vep \bar c/6,t)$ of Lemma \ref{l:decay.corel},
\begin{align}
\Big| \frac{\cZ_{\Graph,W}(\ul y_W)}{\cZ_{\Graph^-(o),W}(\ul y_W)}  - 
\varphi^{\be,B}_{\Graph^-(o)} (\Psi^{{\rm vx}}(\htgs_{\partial o}^{t,r}) 
\mid {\ul \eta_{W}=\ul y_W})
\Big|  & \le (2/\bar c)
 \varphi^{\be,B}_{\Graph^-(o)}(\sS_{t,r}^c \mid {\ul \eta_{W}=\ul y_W}) \le \f{\vep}{3} \,,
\label{eq:claim-cZ-ratio-1}
\end{align}
where in the last inequality we have employed \eqref{eq:decay-corel2} for $\varphi^{\be, B}_{\Graph^-(o)}$. 

Analogously to the derivation of \eqref{eq:vertex.removal.Z}, except for now using  
\eqref{eq:cZ-def}, we get for any $W_r \subset W \subset \wh{W}$, 
\[
\frac{\cZ_{\Graph^-(o),W}(\ul y_W)}{\cZ_{\Graph,W}(\ul y_W)}
= \frac{1}{e^B +q-1} \varphi^{\be,B}_{\Graph}(\ul\eta_{\sfB_o(1)}  
\equiv 0\mid \ul \eta_{W}=\ul y_W)  \,.
\]
Hence, by \eqref{eq:decay-corel3}, if $\wh \Graph \supset \Graph$ matches
$\Graph$ on their respective $r$ balls around $o$ and $r \ge r_0(\vep/9,t)$, then 
\beq\label{eq:claim-cZ-ratio-2}
\max_{\ul y_{\wh W}} \Big|\frac{\cZ_{\wh \Graph,W_r}(\ul y_{W_r})}{\cZ_{\wh \Graph^-(o),W_r}(\ul y_{W_r})}  \frac{\cZ_{\Graph^-(o),\wh W}(\ul y_{\wh W})}{\cZ_{\Graph,\wh W}(\ul y_{\wh W})} -1 \Big| \, 
\le \frac{\vep}{3}.
\eeq
We combine \eqref{eq:claim-cZ-ratio-2} with 
\eqref{eq:claim-cZ-ratio-1} at $(\wh\Graph,W_r)$, and recall \eqref{eq:bar-Psi1} 
that $\Psi^{{\rm vx}}_{t,r,\wh \Graph}(o,\cdot) = 
\varphi^{\be,B}_{\wh \Graph^-(o)} (\Psi^{{\rm vx}}(\htgs_{\partial o}^{t,r}) 
| {\ul \eta_{W_r}})$, to arrive at \eqref{eq:z.ratio.estimate1}.

Upon changing $f$ to $f^{\pi}$ and $\ul \eta^{(1)}$ to $\ul \eta_{\sfB_o^\pi(1) \setminus \sfB_o^-(1)}$
(as in the derivation of \eqref{eq:claim-cZpi-1}),
the proof of \eqref{eq:z.ratio.estimate2} out of \eqref{eq:claim-cZ-pi}
is the same, hence omitted.
\end{proof}

Hereafter we fix the sequence $\{\Graph_n\}$, using
for $v \in [n]$ the shorthand $\varphi_{n}^v:= \varphi_{\Graph_n^-(v)}^{\be, B}$ and define
\begin{align}\label{eq:bar-Psi2}
\Psi^{{\rm sym}}_{t,r,n}(v, \ul y) &:= \Psi^{{\rm vx}}_{t,r,n}(v, \ul y)/ 
\Psi^{{\rm e}, {\rm sym}}_{t,r,n}(v, \ul y),
\end{align}
for $\Psi^{{\rm vx}}_{t,r,n} := \Psi^{{\rm vx}}_{t,r,\Graph_n}$ and 
$\Psi^{{\rm e}, {\rm sym}}_{t,r,n} := \Psi^{{\rm e},{\rm sym}}_{t,r,\Graph_n}$ of \eqref{eq:bar-Psi1}.
For any $ t < r$, $v \in [n]$ and $\vep>0$, we further define, analogously to \eqref{dfn:F-vep-tr}, 
the bond events
\beq\label{dfn:barB-vep-tr}
 \cA_{\vep,v,n}^{t,r} := \Big\{ \ul y : \log  
\Psi_{t,r,n}^{{\rm sym}}(v,\ul y_{\sfB_v(r)\setminus \sfB_v(t)}) 
< \wh \Phi -\vep \Big\} \,. 
\eeq
Aiming to later utilize Lemma \ref{l:z.ratio.estimate}, 
we first adapt Lemma \ref{l:s.in.Lambda} to lower bound the average over $v$
of $\varphi^{\be,B}_{\Graph_n} (\cA_{\vep,v,n}^{t,r})$, when
$\varphi^{\be,B}_n$ converges locally weakly to $\gm_{\bond}$.
Out of this we produce vertex subsets $\sV_n$, having well-spaced, tree-like 
neighborhoods in $\Graph_n$, at which with non-negligible probability
a positive fraction of the events \eqref{dfn:barB-vep-tr} hold. 
\begin{lem}\label{c.lowermean}
Suppose $\gm_{\bond}(\cQ_\star \setminus \cQ_{{\rm m}}) >0$.\\ 
The following then holds
for some $\delta_o >0$, $t'$, large $r'(t)<\infty$, and small enough $\zeta(\delta_o,r)>0$.\\
(a). If $\Graph_n \loc \Tree_d$ and $\{\varphi_n:=\varphi_{\Graph_n}^{\be,B}\}$ 
converges locally weakly to $\gm_{\bond}$, then for $t \ge t'$ and $r \ge r'(t)$, 
\[
\liminf_{n \to \infty} \frac{1}{n} \sum_{v=1}^n \varphi_n
(\cA_{\delta_o,v,n}^{t,r}) \ge 4\delta_o \,.
\]
(b). \red{There exists $n'(\zeta)$ so that for any} $n \ge n'(\zeta)$, there exist $\sV_n \subset [n]$ of size $|\sV_n|=\zeta n$, such that 
\begin{align}\label{eq:top}
\inf_{v\ne v' \in \sV_n} \{ {\rm dist}_{\Graph_{n}} (v,v') \} > 3r, \quad & \quad
\sfB_v(3r) \cong \Tree_d(3r), \;\; \forall v \in \sV_n \,, \\
\label{eq:Qk.property}
\varphi_n \Big(\f1{|\sV_n|}\sum_{v \in \sV_n} {\bf 1}_{\cA^{t,r}_{\delta_o,v,n}} &\geq 
\delta_o \Big)\geq \delta_o.
\end{align}
\end{lem}
\begin{proof} (a). Note that for some finite $C=C(q,d,\be,B)$, uniformly over $n$ 
and $v \in [n]$ with $|\partial v| \le d$, 
\beq
\nonumber
C^{-1} \varphi_n(\ul \eta_{{\Graph_n^-(v)}}) \le
\varphi_n^v(\ul \eta_{\Graph_n^-(v)}) \le C \varphi_n(\ul \eta)
\qquad \forall
\ul \eta \in \{0,1\}^{E_n^\star} \,.
\eeq
Then, from $\Graph_n \loc \Tree_d$ and Lemma \ref{l:s.in.Lambda},
we have that for any $\vep \le \xi/(2C)$, $t \ge t'$ and $r \ge r'(t)$,
\beq\label{eq:mod-bd}
\liminf_{n \to \infty} \frac{1}{n} \sum_{v=1}^n {\bf 1}_{\{\red{\tsfB}_v(r) \cong \red{\tTree}_d(r)\}} 
\big( \varphi_n^v ( \sB_{2C \vep,v,n}^{t,r} ) - \vep \big)
 \ge C^{-1} \int \varphi (\sB_{2 C \vep}^{t,r}) d \gm_{\bond} - \vep \ge \vep \,,
\eeq
where $\sB_{\vep,v,n}^{t,r} := \{ \log \Psi^{{\rm sym}}(\htgs_{\partial v}^{t,r}) < \wh \Phi -\vep \}$.
Further, for some $c=c(q,d,\be,B)>0$, 
\beq\label{eq:Theta}
\Psi^{{\rm vx}}(\ul b),\Psi^{{\rm e}, {\rm sym}}(\ul b) \in [1,c^{-1}], 
\qquad \forall \ul b \in [0,1]^d, 
\eeq
and we now set for any $v \in [n]$, 
\beq\label{dfn:Y-phi}
Y_{v,n}(\ul y) := \varphi_n^v \left[\Psi^{{\rm e}, {\rm sym}}(\htgs_{\partial v}^{t,r}) {\bf 1}_{\sB^{t,r}_{2 C \vep,v,n}} \mid \ueta_{v,t}^r=\ul y \right] -  c \vep \Psi_{t,r,n}^{{\rm e}, {\rm sym}}(v,\ul y).
\eeq
Recall from \eqref{eq:lem-prop-rcm-msg3} that 
$\htgs_{\partial v}^{t,r} \in [b_{\free}-\vep^4, b_{\wired}+\vep^4]^d$
whenever $t > \kappa_o(\vep^4)$ and $\sfB_v(r) \cong \Tree_d(r)$.
For such $v$ we have by \eqref{eq:wh-Phi} and the Lipschitz continuity of 
$\log \Psi^{{\rm sym}}(\cdot)$, that for some $\vep'>0$ and all $\vep \le \vep'$, 
\beq\label{eq:Phi-ubd}
\sup_{r  > t > \kappa_o(\vep^4)} \sup_{\ul y \in \B_r^\star \setminus \B_t^\star} 
\{ \log \Psi^{{\rm sym}}\big( \htgs_{\partial v}^{t,r}(\ul y) \big) \} \le \wh \Phi+ \vep^3
\eeq
\red{(see Definition \ref{dfn:rcm-msg} for $\htgs_{\partial v}^{t,r}(\cdot)$).}
In view of \eqref{eq:bar-Psi1}, we have for any $t>\kappa_o(\vep^4)$ that if $\sfB_v(r) \cong \Tree_d(r)$
and $Y_{v,n}(\ul y) \ge 0$, then
\begin{align*}
\Psi^{{\rm vx}}_{t,r,n}& (v,\ul y)  = {\varphi_n^v} \Big[\Psi^{{\rm vx}}(\htgs_{\partial v}^{t,r}) 
{\bf 1}_{\sB^{t,r}_{2{C} \vep,v,n}} | \ueta_{v,t}^r = \ul y \Big]
+ 
{\varphi_n^v} \left[\Psi^{{\rm vx}}(\htgs_{\partial v}^{t,r}) 
{\bf 1}_{(\sB^{t,r}_{2 {C} \vep,v,n})^c} \mid \ueta_{v,t}^r = \ul y \right]\notag\\
& < e^{\wh \Phi -2 {C} \vep} 
{\varphi_n^v} \left[\Psi^{{\rm e},{\rm sym}}(\htgs_{\partial v}^{t,r}) 
{\bf 1}_{ \sB^{t,r}_{2 {C} \vep,v,n}}\mid \ueta_{v,t}^r = \ul y \right]+ e^{\wh \Phi +\vep^3} 
{\varphi_n^v} \left[\Psi^{{\rm e},{\rm sym}}(\htgs_{\partial v}^{t,r}) 
{\bf 1}_{(\sB^{t,r}_{2 {C} \vep,v,n})^c} \mid \ueta_{v,t}^r = \ul y \right]\notag\\
&= e^{\wh \Phi -2 {C} \vep} \left(
\Psi^{{\rm e}, {\rm sym}}_{t,r,n}(v,\ul y) + (e^{2 {C} \vep+\vep^3}-1) 
{\varphi_n^v} \big[\Psi^{{\rm e},{\rm sym}}(\htgs_{\partial v}^{t,r}) 
{\bf 1}_{(\sB^{t,r}_{2{C} \vep,v,n})^c} \mid \ueta_{v,t}^r = \ul y \, \big] \right) \notag\\
& \le e^{\wh \Phi -2 {C} \vep}(1+ (e^{2 {C} \vep+\vep^3}-1)(1-c\vep)) \Psi^{{\rm e}, {\rm sym}}_{t,r,n}
(v,\ul y) \le e^{\wh \Phi -c\vep^2} \Psi^{{\rm e}, {\rm sym}}_{t,r,n}(v,\ul y),
\end{align*}
where the first inequality uses the definition 
of $\sB^{t,r}_{2 {C} \vep,v,n}$ and the uniform bound \eqref{eq:Phi-ubd}, the second one holds
for $Y_{v,n}(\ul y) \ge 0$ (see \eqref{dfn:Y-phi}),
and the last one applies when $\vep \in (0,\vep'')$. 
That is, in view of \eqref{eq:bar-Psi2}, for any $\vep \le \vep' \wedge \vep''$ if $r>t>\kappa_o(\vep^4)$
and $\delta \le c \vep^2$, then
\beq\label{eq:wh-sF-tR-ineq-2}
\{ \sfB_v(r) \cong \Tree_d(r) \text { and } Y_{v,n}(\ul y) \ge 0 \} \Longrightarrow 
\{ \log \Psi_{t,r,n}^{{\rm sym}}(v,\ul y) < \wh \Phi -\delta \} \,.
\eeq
As $Y_{v,n} \le 1/c$, clearly ${\varphi_n^v}  (Y_{v,n} \ge 0) \ge c {\varphi_n^v} [Y_{v,n}]$. 
Further, by \eqref{dfn:Y-phi}, \eqref{eq:bar-Psi1} and \eqref{eq:Theta},
\begin{align}\label{eq:wh-sF-tR-ineq-1}
{\varphi_n^v}  [Y_{v,n}(\ueta_{v,t}^r)] & =  {\varphi_n^v}  \big[\Psi^{{\rm e}, {\rm sym}}(\htgs_{\partial v}^{t,r}) {\bf 1}_{\sB^{t,r}_{2{C} \vep,v,n}}\big] 
- c \vep  {\varphi_n^v}  \big[ \Psi^{{\rm e}, {\rm sym}}(\htgs_{\partial v}^{t,r})\big] 
 \ge {\varphi_n^v}  ({\sB^{t,r}_{2{C} \vep,v,n}}) - \vep \,. 
\end{align}
Combining \eqref{dfn:barB-vep-tr}, \eqref{eq:wh-sF-tR-ineq-2} and \eqref{eq:wh-sF-tR-ineq-1}, 
we find that
\[
\varphi_n (\cA_{\delta,v,n}^{t,r}) \ge  C^{-1} {\bf 1}_{\{\sfB_v(r) \cong \Tree_d(r)\}} 
{\varphi_n^v}  (Y_{v,n} (\ueta_{v,t}^r) \ge 0) \ge 
\frac{c}{C} {\bf 1}_{\{\sfB_v(r) \cong \Tree_d(r)\}} \big( {\varphi_n^v}  ( \sB_{2{C} \vep,v,n}^{t,r} ) - \vep \big) \,.
\]
The preceding inequality, together with \eqref{eq:mod-bd}, completes our proof 
(with $\delta_o := \delta \wedge { c \vep/(4C)}$). 

\noindent
(b). We establish the existence of $\cV_n$ via a probabilistic construction based on the auxiliary
\abbr{iid}, uniform over $[n]$, samples $\{w_i, i \le 2 \zeta n\}$. Specifically,
set $\sU:=\sU_0\setminus \bar \sU$ where
\[
\sU_0:=\{i:   \sfB_{w_i}(3r)\cong \Tree_d(3r)\}, \quad \bar\sU:= \{i \in \sU_0: \exists j \ne i \text{ such that }  {\rm dist}_{\Graph_{n}} (w_i, w_j) \le 3 r\}\,.
\]
Observe that $n^{-1} \E|\sU_0| \to 2\zeta$ since $\Graph_{n} \loc \Tree_d$. Moreover,
\[
\E|\bar\sU| \le \sum_{i \ne j} \P(w_j \in \sfB_{w_i}(3r) \cong \Tree_d(3r) ) \le 
2 \zeta |\Tree_d(3r)| \, \E|\sU_0 | \,.
\]
Thus, $\E|\sU| \ge (1- \sqrt{\zeta}) \E |\sU_0|$ for any $\zeta \le \zeta_o(r)$, \red{where $\zeta_o(r)$ is some sufficiently small constant depending only on $r$}, in which case
\beq\label{eq:sU-bd}
\liminf_{n \to \infty} \P\big(|\sU| \ge 2\zeta n (1-3 \sqrt{\zeta}) \big) 
\ge \liminf_{n \to \infty} \f{\E|\sU|}{2\zeta n} - (1 - 3 \sqrt{\zeta})  \ge 2 \sqrt{\zeta}. 
\eeq
With $q_n(v):=\varphi_n(\cA_{\delta_o,v,n}^{t,r})$, recall from part (a) that
the \abbr{iid} $[0,1]$-valued $\{q_n(w_i)\}$ satisfy 
\[
\liminf_{n \to \infty} \E [q_n(w_i)] = 
\liminf_{n \to \infty} \frac{1}{n} \sum_{v=1}^n q_n(v) \ge 4 \delta_o \,.
\]
Hence, by Chebychev's inequality
\beq\label{eq:cheb}
\lim_{n \to \infty} \P ( \frac{1}{2\zeta n} \sum_{i=1}^{2 \zeta n} q_n(w_i) \ge 3 \delta_o ) = 1 
\eeq
and in view of \eqref{eq:sU-bd}, at any $n \ge n'(\zeta)$ we have with probability 
$\sqrt{\zeta}$ that $|\sU| \ge 2\zeta n (1-3\sqrt{\zeta})$ and the event in \eqref{eq:cheb} holds.
For $3 \sqrt{\zeta} \le \delta_o \le \frac{1}{2}$, it then follows that  
\[
\frac{1}{|\sU|} \sum_{i \in \sU} q_n(w_i) \ge
\frac{1}{2\zeta n} \sum_{i \in \sU} q_n(w_i) \ge 
\frac{1}{2\zeta n} \sum_{i=1}^{2 \zeta n} q_n(w_i) - 3 \sqrt{\zeta} \ge 3 \delta_o -  3 \sqrt{\zeta} 
\ge 2 \delta_o \,.
\]
Take for $\sV_n$ the $[\zeta n]$ elements of largest 
$q_n(\cdot)$ in $\{w_i : i \in \sU\}$. Then, \eqref{eq:top} holds and moreover
\[
\frac{1}{|\sV_n|} \sum_{v \in \sV_n} \varphi_n(\cA_{\delta_o,v,n}^{t,r}) \ge 2 \delta_o \,,
\]  
from which \eqref{eq:Qk.property} immediately follows.
\end{proof}

\begin{proof}[Proof of Theorem \ref{thm:crit-1}] Fix $(\be,B) \in \sfR_c$, $B>0$,
$d \in 2 \N$, $d \ge 3$ and $\Graph_n \loc \Tree_d$, further assuming \abbr{wlog} that 
$\varphi_n=\varphi_{\Graph_n}^{\be, B}$ converges locally weakly to $\gm_{\bond}$.  
In view of \cite[Theorem 1]{DMSS}, postulating further that 
$\gm_{\bond}(\cQ_\star \setminus \cQ_{\rm m})>0$, it suffices to produce 
$\Graph'_n \loc \Tree_d$ such that ${\varlimsup_n} \; \Phi'_n(\be,B) > \wh \Phi$. 
To this end, for $\delta_o>0$ of Lemma \ref{c.lowermean} and $c>0$ from \eqref{eq:Theta}, set
\beq\label{eq:vep-star-def}
\vep_\star:=- \frac{1}{3} \log\Big( 1- (1-e^{-\delta_o}) \frac{\delta_o c^2}{3} \Big).
\eeq
Similarly to the derivation of \eqref{eq:Phi-ubd}, with $\log \Psi^{\rm vx}(\cdot)$ and $\log \Psi^{{\rm e},{\rm sym}}(\cdot)$ Lipschitz on $[0,1]^d$, 
taking $\kappa_o=\kappa_o(\vep_\star^2)$ as in Remark \ref{rmk:rcm-msg-bd-graph} 
(and shrinking $\delta_o$ as needed for $\vep_\star$ to be small enough), guarantees by
\eqref{eq:wh-Phi}, \eqref{eq:bar-Psi1}, and \eqref{eq:bar-Psi2} that for $r>t>\kappa_o$ and any
finite graph $\Graph_n$,
\beq\label{eq:Psi.sym.Ubound}
\sup_{\ul y \in \B_r \setminus \B_t} \log \Psi_{t,r,n}^{{\rm sym}}(v, \ul y) \le \wh \Phi + \vep_\star
\eeq
provided $v \in V(\Graph_n)$ is such that $\sfB_v(r) \cong \Tree_d(r)$.
Hereafter we further take $r>r'(t)$ and $t>t'$ as needed in 
Lemma \ref{c.lowermean}, possibly 
increasing $r'(t)$ till Lemma \ref{l:z.ratio.estimate} holds with $\vep=\frac{\vep_\star}{4}$.
Then, for $\sV_n$ of Lemma \ref{c.lowermean}(b), set the disjoint union 
$W_n :=\cup_{v\in \sV_n} (\sfB_v(r)\setminus\sfB_v(t))$ and
for $\cA_v=\cA_{\delta_o,v,n}^{t,r}$ of \eqref{dfn:barB-vep-tr} let
\beq\label{dfn:cY}
\mathcal{Y}_n :=\Big\{\ul y \in \{0,1\}^{E(W_n)}: \frac1{|\sV_n|}\sum_{v\in \sV_n} {\bf 1}_{\cA_{v}}\geq \delta_o  \Big\}.
\eeq
Our graph decomposition starts 
at $\Graph_{n}^{(0)}=\Graph_{n}$ and 
iteratively has $\Graph_n^{(\ell)} :=\Graph_{n}^{(\ell-1),\pi_\ell}(v_\ell)$ for 
\begin{align*}
v_\ell & := \argmaxx_{v \in \sV_n \setminus \{v_1, \ldots, v_{\ell-1}\}} \Big\{ \sum_{\ul y \in \mathcal{Y}_n} \cZ_{\Graph_{n}^{(\ell-1)},W_n}(\ul y) {\bf 1}_{\cA_v}(\ul y) \Big\}, \\
\pi_\ell & := \argmaxx_{\pi \in \red{\sfS}_d} \Big\{ \sum_{\ul y \in \mathcal{Y}_n} 
\cZ_{\Graph_{n}^{(\ell-1),\pi}(v_\ell),W_n}(\ul y) \Big\} .
\end{align*}
We will show in the sequel that for any $\ell \le \ell_n := \frac{\delta_o}{3} \zeta n$,
\beq\label{eq:delta-cZ}
\Delta\cZ_\ell:= \log\Big( \sum_{\ul y \in \mathcal{Y}_n} \cZ_{\Graph_n^{(\ell-1)}, W_n}(\ul y) \Big) - \log\Big( \sum_{\ul y \in \mathcal{Y}_n} \cZ_{\Graph_n^{(\ell)}, W_n}(\ul y) \Big) \le \wh \Phi - \vep_\star .
\eeq
We claim that $\Graph_n' :=\Graph_n^{(\ell_n)}$ then have too large free densities $\Phi'_n(\be,B)$. 
Indeed, by \eqref{eq:wh-varpi}, \eqref{eq:wh-varpin}, \red{\eqref{eq:cZ-def}} and the fact that the Potts measure is the spin marginal of the Edwards-Sokal measure, 
$$
Z_{\red{\Graph_n'}}(\be,B) = \sum_{\ul y \red{ \in \{0,1\}^{E(W_n)}}} \cZ_{\red{\Graph_n'}, W_n}(\ul y). 
$$
Thus, by the \abbr{lhs} of
\eqref{eq:Qk.property},  \red{\eqref{dfn:cY}} and \eqref{eq:delta-cZ},
\begin{align*}
\log Z_{\Graph'_{n}}(\be,B) \ge \log \bigg(\sum_{\ul y \in \cY_n} \cZ_{\Graph^{(\ell_n)}_{n},W_n}(\ul y)\bigg) &
= \log \bigg(\sum_{\ul y \in \cY_n} \cZ_{\Graph_{n},W_n}(\ul y)\bigg) - \sum_{\ell=1}^{\ell_n} \Delta\cZ_\ell \\
  &\ge \log Z_{\Graph_{n}}(\be,B) +\log\delta_o -  (\wh\Phi- \vep_\star)\ell_n\,.
\end{align*}
With $\Graph_n \loc \Tree_d$ we have from \cite[Theorem 1]{DMSS} that $\Phi_n(\be,B) \to \wh \Phi$,
so by the preceding, 
\beq\label{eq:cZ-lbd-limit}
\liminf_{n \to \infty} \frac1{n}\log Z_{\Graph'_{n}}(\be,B) \geq \wh\Phi (1-   \frac{\delta_o}{3} \zeta
) + \frac{\delta_o}{3} \zeta \vep_\star .
\eeq
Note that $|V(\Graph'_n)|=n-\ell_n=(1-\frac{\delta_o}{3} \zeta) n$ and the degree in $\Graph'_n$ of any 
$w \in V(\Graph'_n)$ matches its degree in $\Graph_n$. Moreover, 
the sets $\sfB_{v_\ell}(1)$ on which $\Graph'_n$ and $\Graph_n$ differ, 
are $2r$-separated on $\Graph'_n$ and if $(u,u') \in E(\Graph'_n)$ within 
such a changed set $\sfB_{v}(1)$, then ${\rm dist}_{\Graph_n}(u,u') \le 2$.
Consequently, any cycle of length $k$ in $\Graph_n'$ must have been obtained 
from a cycle of length at most $k/(1+(2r)^{-1})$ in $\Graph_n$.
The last two observations together imply that $\Graph_n' \loc \Tree_d$ as $n \to \infty$. 
Hence, by \cite[Theorem 1]{DMSS} 
\[
\limsup_{n \to \infty} \frac1{n}\log Z_{\Graph_n'} (\be,B)\leq \wh\Phi (1-   \frac{\delta_o}{3} \zeta),
\]
in contradiction to \eqref{eq:cZ-lbd-limit}.

We thus complete the proof of Theorem \ref{thm:crit-1} upon establishing \eqref{eq:delta-cZ}.
To this end, we shall apply Lemma~\ref{l:z.ratio.estimate} with $\wh{\Graph}=\Graph_n$,
$\Graph=\Graph^{(\ell-1)}_n$, and $o=v_\ell \in \cV_n$ (so $\sfB_o(r) \cong \Tree_d(r)$ in view of \eqref{eq:top}), where as $\inf_{i<\ell} {\rm dist}_{\Graph_n}(v_\ell,v_i) > 3r$ (see \eqref{eq:top}), 
indeed $\Graph^{(\ell-1)}_n$ coincides with $\Graph_n$ on its $r$-ball around $v_\ell$. Further, 
Remark \ref{rmk:z.ratio} allows us to do so for $\wh{W}=W_n$ (which by 
construction contains the relevant annuli $\sfB_{v_\ell}(r) \setminus \sfB_{v_\ell}(t)$ while 
excluding $\sfB_{v_\ell}(t)$). Specifically, using the shorthand
\[
\wh \cZ_\ell(\ul y):=\cZ_{\Graph_{n}^{(\ell)},W_n}(\ul y), \quad 
\wh \cZ^-_\ell(\ul y):=\cZ_{\Graph_{n}^{(\ell-1),-}(v_\ell),W_n}(\ul y),
\quad \wh \cZ^\pi_\ell(\ul y):=\cZ_{\Graph_{n}^{(\ell-1),\pi}(v_\ell),W_n}(\ul y),
\]
and 
$\Gamma_\ell := \sum_{\ul y\in \mathcal{Y}_n} \Psi^{{\rm e}, {\rm sym}}_{t,r,n}(v_\ell, \ul y) \wh \cZ_\ell^-(\ul y)$, we have by Lemma~\ref{l:z.ratio.estimate} and our choice of $r$ that,
\begin{align}\label{eq:cY.partion.comp.bound}
\Delta \cZ_\ell 
&= \log\bigg( \sum_{\ul y \in \mathcal{Y}_n} \frac{\wh \cZ_{\ell-1}(\ul y)}{\wh \cZ_\ell^-(\ul y)}\wh \cZ_\ell^-(\ul y) \bigg) - \log\bigg( \max_\pi \sum_{\ul y\in \mathcal{Y}_n} \frac{ \wh \cZ_\ell^\pi(\ul y) }{\wh \cZ_\ell^-(\ul y)} \wh \cZ_\ell^-(\ul y) \bigg)\nonumber \\
&\leq\vep_\star +\log\bigg( \sum_{\ul y\in \mathcal{Y}_n} \Psi^{{\rm vx}}_{t,r,n}(v_\ell, \ul y) 
\wh \cZ_\ell^-(\ul y) \bigg) -\log\bigg( \max_\pi \sum_{\ul y\in \mathcal{Y}_n} 
\Psi^{{\rm e},\pi}_{t,r,n}(v_\ell, \ul y) \wh \cZ_\ell^-(\ul y) \bigg)\nonumber\\
 & \leq \vep_\star +\log\bigg( \sum_{\ul y\in \mathcal{Y}_n} \Psi^{{\rm vx}}_{t,r,n}(v_\ell, \ul y) \wh 
 \cZ_\ell^-(\ul y) \bigg) -\log \Gamma_\ell 
 \,.
 \end{align}
To simplify the \abbr{RHS} of \eqref{eq:cY.partion.comp.bound}, note that by our choice of $v_\ell$,
\begin{align}
\sum_{\ul y\in \mathcal{Y}_n} \wh\cZ_{\ell-1}(\ul y) {\bf 1}_{\cA_{v_\ell}}(\ul y) &\ge \f{1}{|\sV_n|} \sum_{v \in \sV_n \setminus \{v_1,\ldots, v_{\ell-1}\}} \sum_{\ul y\in \mathcal{Y}_n} \wh\cZ_{\ell-1}(\ul y) {\bf 1}_{\cA_v}(\ul y)\notag \\
&\ge \f{1}{|\sV_n|}  \sum_{\ul y\in \mathcal{Y}_n} \wh\cZ_{\ell-1}(\ul y) \bigg(\sum_{v \in \sV_n } {\bf 1}_{\cA_v}(\ul y) - (\ell-1)\bigg) \ge \f{2\delta_o}{3}\sum_{\ul y\in \mathcal{Y}_n} \wh\cZ_{\ell-1}(\ul y),
\label{eq:partition-fn-ratio-bd}
\end{align}
with the right-most inequality due to \eqref{dfn:cY} and the size of $\sV_n$. Moreover, by Lemma \ref{l:z.ratio.estimate} and \eqref{eq:Theta} 
\begin{align*}
\sum_{\ul y\in \mathcal{Y}_n} \Psi^{{\rm e}, {\rm sym}}_{t,r,n}(v_\ell, \ul y) \wh\cZ_\ell^{-}(\ul y) {\bf 1}_{\cA_{v_\ell}}(\ul y) 
\geq \frac{3c}{4} \sum_{\ul y\in \mathcal{Y}_n}  \wh\cZ_{\ell-1}(\ul y) 
{\bf 1}_{\cA_{v_\ell}}(\ul y)
\geq \frac{\delta_o c}{2} \sum_{\ul y\in \mathcal{Y}_n} \wh\cZ_{\ell-1}(\ul y)
\geq 
\frac{\delta_o c^2}{3} \Gamma_\ell
\end{align*} 
where in the penultimate step we have used \eqref{eq:partition-fn-ratio-bd}. Substituting the latter
inequality in~\eqref{eq:cY.partion.comp.bound}, upon recalling the definitions of $\cA_{v_\ell}$ and
$\Psi^{{\rm sym}}_{t,r,n}(\cdot, \ul y)$, and using \eqref{eq:Psi.sym.Ubound} we get
\begin{align*}\label{eq:cY.partion.comp.bound2}
e^{\Delta \cZ_\ell -\vep_\star} &\leq 
 \f{1}{\Gamma_\ell} 
 \sum_{\ul y\in \mathcal{Y}_n} \exp\big(\wh\Phi+\vep_\star - \delta_o{\bf 1}_{\cA_{v_\ell}}(\ul y)
\big)\Psi^{{\rm e},{\rm sym}}_{t,r,n}(v_\ell, \ul y) \wh \cZ_\ell^-(\ul y)\\
&\leq e^{\wh \Phi +\vep_\star} \big( 1-\frac{1-e^{-\delta_o}}{\Gamma_\ell}
\sum_{\ul y\in \mathcal{Y}_n} \Psi^{{\rm e},{\rm sym}}_{t,r,n}(v_\ell, \ul y) \wh \cZ_\ell^-(\ul y) {\bf 1}_{\cA_{v_\ell}}(\ul y)
\big)\nonumber\\
&\leq e^{\wh \Phi +\vep_\star} 
\big( 1-(1-e^{-\delta_o}) \frac{\delta_o c^2}{3} \big)= e^{\wh \Phi -2\vep_\star}
\end{align*}
(with the last equality by our choice \eqref{eq:vep-star-def} of $\vep_\star$).
This is \eqref{eq:delta-cZ}, which thus completes our proof.
\end{proof}

\appendix
\section{Infinite volume Potts measures and Bethe fixed points}\label{sec:app}
{We prove here} various properties of the {\em phase diagram} associated to the
Potts measures {on $\Tree_d$. To this end, for any $(i,j) \in E(\Tree_d)$ let}
$\Tree_{i \to j}$ be the tree rooted at $i$ obtained {upon deleting $(i,j)$ from $E(\Tree_d)$ and with
$\Tree_{i \to j}(t)$ denoting for $t \in \N$} the ball $\sfB_i(t)$ in $\Tree_{i \to j}$, we first prove
Lemma \ref{lem:marginals}. 


\begin{proof}[Proof of Lemma \ref{lem:marginals}]
{Summing in \eqref{eq:mg-pairs} over $\sigma_j$ results with \eqref{eq:mg-single}.}
The proof of \eqref{eq:mg-pairs} in case $q=2$ is well known (e.g. see \cite[Lemma 4.1]{DM}), 
and the extension for $q \ge 3$ {(which we now give),} follows a similar scheme. 
{With $\mu^{\be,B}_{\ddagger}$ a translation invariant measure, it suffices to} consider
$i=o$ and $j=1$ a specific neighbor of the root $o$ of $\Tree_d$. {Fixing $t \in \N$,} 
as there are no cycles in $\Tree_d(t)$,  we have under $\mu_{\ddagger, t}^{\be, B}$ of 
Definition \ref{dfn:potts-bounary-B-0}, that for some 
{$\widehat \nu_t^1,\widehat \nu_t^o \in \cP([q])$} 
\[
\mu_{\ddagger, t}^{\be, B} (\sigma_o,\sigma_1) \propto e^{\beta \delta_{\sigma_o,\sigma_1}} \widehat \nu^1_t(\sigma_1) \widehat \nu^o_t(\sigma_o)\,, \qquad \forall \sigma_o,\sigma_1 \in [q].
\]
{Further,} both $T_{o \ra 1}$ and $T_{1 \ra o}$ are infinite $(d-1)$-ary trees (i.e.~each vertex has $(d-1)$ children). {Hence,} by induction we deduce from Definitions \ref{dfn:potts-bounary-B-0} 
and \ref{dfn:bethe-recursion}, that $\widehat \nu_t^o$ and $\widehat \nu_t^1$ are the probability measures obtained by running the ${\sf BP}$ recursion $t$ and $(t-1)$ times, respectively, starting from 
$\nu_{\ddagger}$ {(i.e. starting at the uniform measure if $\ddagger=\free$ and at Dirac at $1$ 
when $\ddagger=1$). By definition} the limit of these recursions is $\nu_{\ddagger}^{\be, B}$
{and \eqref{eq:mg-pairs}} follows. 
\end{proof}

\begin{rmk}\label{rmk:infinite-vol-exist}
{By a similar reasoning one finds also} that for any $\be, B \ge 0$ and $\ddagger \in \{\free, 1\}$, 
all finite dimensional marginals of $\mu_{\ddagger, t}^{\be, B}$ converge as $t \to \infty$. Furthermore, it can be checked that these marginals are consistent, {thus implying} the existence of $\mu_{\ddagger}^{\be, B}$ for $\ddagger \in \{\free, 1\}$. The existence of $\mu_{i}^{\be, 0}$ for all $i \in [q]$ 
{is proved similarly (now starting the ${\sf BP}$ recursion at Dirac at $i$).}
\end{rmk}
Next we prove that $\nu_{\free}$ and $\nu_1$, when viewed as functions of $\be$ and $B$, are continuously differentiable except on $\partial \sfR_{\neq}^{+}$ and $\partial \sfR_{\neq}^{\free}$, respectively.

\begin{proof}[Proof of Lemma \ref{lem:smoothness-fixed-pt}]
(i) 
{Since}
$\nu_1^{\be, B}$ is a fixed point of the ${\sf BP}$ recursion of \eqref{eq:BP-recursion}, starting from the probability measure on $[q]$ that is Dirac at $1$, {necessarily} $\nu_1^{\be,B}(\sigma)=\nu_1^{\be,B}(\sigma')$ for all $\sigma, \sigma' \in [q]\setminus \{1\}$ {and as} $\nu_1^{\be, B} \in \cP([q])$, 
{it suffices} to show that $(\be,B) \mapsto \nu_1^{\be,B}(1)$ is continuously differentiable
{at any fixed} $(\be_0,B_0) \notin \partial \sfR_{\ne}^{\free}$. Further, {after some algebra we
deduce from  \eqref{eq:BP-recursion} that}
\[
r_1(\be,B):=\log \frac{\nu_{1}^{\be,B}(1)}{\nu_1^{\be,B}(2)} =
   \log \frac{(1-q)\nu_{1}^{\be,B}(1)}{1-\nu_1^{\be,B}(1)} \,, 
\]
satisfies the fixed point equation 
\beq\label{eq:fixed-pt}
F(r; \be, B)=r,
\eeq
where
\[
F(r; \be, B):= B+ (d-1) \log \left( \f{e^{\be+r}+q-1}{e^r +e^\be +q-2}\right).
\]
Using this representation it suffices to show that $(\be, B) \mapsto r_1(\be,B)$ is continuously differentiable,
{and this follows by} an application of the implicit function theorem,  once we {have verified} that  
\beq\label{eq:implicit-cond}
\f{\partial }{\partial r} F(r; \be, B) \ne 1, \qquad \text{ for } (r, \be, B) = (r_1(\be_0,B_0), \be_0,B_0).
\eeq
To obtain \eqref{eq:implicit-cond} we borrow results from \cite{DMS}. From the proof of \cite[Lemma 4.6]{DMS} it follows that for $d \ge 3$ there is some $\be_- \in (0, \infty)$ such that for $\be < \be_-$ there does not exist any solution to the equation 
\beq\label{eq:implicit-cond-1}
\frac{\partial}{\partial r} F(r; \be, B) = \frac{\partial}{\partial r} F(r; \be, 0)=1.
\eeq
In contrast, for $\be \ge \be_-$ there exist solutions $\rho_{-}(\beta) \le \rho_+(\be)$ of the equation \eqref{eq:implicit-cond-1}, with equality if and only if $\be=\be_-$. Thus, to complete the proof of \eqref{eq:implicit-cond} we need to show that if $(\be_0, B_0) \in \sfR_{\neq} \setminus \partial \sfR_{\neq}^{\free}$ (and therefore $\be_0 > \be_-$, see \cite[Lemma 4.6]{DMS}) then $r_1(\be_0,B_0) \ne \rho_{\pm}(\beta_0)$.

If possible, let us assume that $r_1(\be_0,B_0) = \rho_+(\be_0)$. Denoting
\beq\label{eq:B-pm}
B_\pm(\be):= \rho_\mp (\be) - F(\rho_\mp(\be); \be, 0),
\eeq
as $r_1(\be,B)$ is a fixed point of the equation \eqref{eq:fixed-pt}, we note that $B_0=B_-(\be_0)$. From the proof of \cite[Theorem 1.11(a)]{DMS} we have that the map $B \mapsto \be_{\free}(B)$ is the inverse of the map $\be \mapsto B_-(\be)$. So, $\be_0=\be_{\free}(B_0)$ implying that $(\be_0, B_0) \in \partial \sfR_{\ne}^{\free}$. As $(\be_0, B_0) \notin \partial \sfR^{\free}_{\neq}$ we arrive at a contradiction. 

To rule out the other possibility that $r_1(\be_0,B_0) = \rho_-(\be_0)$ we again proceed by contradiction. As 
$\be_0 >\be_-$ we have that $\rho_-(\be_0) < \rho_+(\be_0)$.  Now note
\[
B_+(\be) - B_-(\be) = \int_{\rho_-(\be)}^{\rho_+(\be)} \left[\f{\partial}{\partial r} F(r;\be, 0) -1\right] dr.
\]
It can be checked that $\frac{\partial^2}{\partial r^2}F(r; \be,0)$ is negative for sufficiently large $r$ and has a single change of sign. As $\rho_{\pm}(\be)$  are the solutions of \eqref{eq:implicit-cond-1} one therefore have that $\f{\partial}{\partial r} F(r;\be,0) > 1$ for $r \in (\rho_-(\be), \rho_+(\be))$. Hence, it follows from above that $B_-(\be_0) < B_+(\be_0)$. 
Thus
\beq\label{eq:implicit-cond-2}
F(\rho_+(\be_0); \be_0, B_0) = \rho_+(\be_0) + B_0 - B_-(\be_0) = \rho_+(\be_0) + B_+(\be_0) - B_-(\be_0) > \rho_+(\be_0),
\eeq
where in the penultimate step we have used the fact that the assumption $r_1(\be_0,B_0)=\rho_-(\be_0)$ implies that $B_0=B_+(\be_0)$.  

Noting that $\lim_{r \to \infty} F(r; \be_0 ,B_0) <\infty$, \eqref{eq:implicit-cond-2} implies that there exists some $r_\star \in (\rho_+(\be_0),\infty)$ such that $r_\star = F(r_\star; \be_0,B_0)$. However, this contradicts the fact that $r_1(\be_0,B_0)$ is the largest root of that fixed point equation \eqref{eq:fixed-pt}. 

(ii) {As in part (i), it now suffices} to show that $(\be,B) \mapsto r_{\free}(\be,B)$ is continuously differentiable for $(\be, B) \in \sfR_{\neq}\setminus \partial \sfR_{\ne}^+$, 
where $r_{\free}(\be,B) := \log \frac{(1-q) \nu_{\free}^{\be,B}(1)}{1-\nu_{\free}^{\be,B}(1)}$ 
{for $\nu_{\free}^{\be,B}(\cdot)$ of} Definition \ref{dfn:bethe-recursion}.
{To this end, fixing} $(\be_0,B_0) \in \sfR_{\neq}\setminus \partial \sfR_{\ne}^+$
we {only} need to show that \eqref{eq:implicit-cond-1} {fails at} $(r,\be,B)=(r_{\free}(\be_0,B_0),\be_0,B_0)$ {which analogously to part (i) be a direct consequence of having that}
$r_{\free}(\be_0,B_0) \ne \rho_{\pm}(\be_0)$.

Proceeding to do this task, the same argument as {in part (i)} shows that 
$r_{\free}(\be_0,B_0) \ne \rho_-(\be_0)$, but a slightly different argument is needed 
{for ruling} out the other choice. {Specifically,} we claim that  
$r_{\free}(\be_0,B_0)=\rho_+(\be_0)$ would imply 
$\rho_-(\be_0) >0$. To see this, from the definition of $B_{\pm}(\be)$ we find that $B_0=B_-(\be_0)$. From the proof of \cite[Theorem 1.11]{DMS} we also have that the map $B \mapsto \be_{\free}(B)$ is the inverse of the map $\be \mapsto B_-(\be)$. Moreover, the assumption $(\be_0,B_0) \in \sfR_{\ne}\setminus \partial \sfR_{\ne}^+$ implies that $B_0 \in [0,B_+)$. Therefore, 
\beq\label{eq:be_+}
\be_0=\be_{\free}(B_0) < \be_+(B_0) \le \be_+(0) =:\be_+,
\eeq
where the first inequality follows from the fact that $\be_{\free}(B) < \be_+(B)$ for $B \in [0,B_+)$ and the second inequality follows from the fact that the map $\be_+(\cdot)$ is decreasing in $B$. Using \cite[Eqn.~(4.10)]{DMS} we {thus} deduce that $\rho_-(\be_0) >0$ whenever 
$r_{\free}(\be_0,B_0)=\rho_+(\be_0)$. Now, similarly to part (i), 
{if both} $\rho_-(\be_0) >0$ and $r_{\free}(\be_0,B_0) = \rho_+(\be_0)$, {then} 
there {exists a solution} $r_\star \in [0, r_{\free}(\be_0,B_0))$ of \eqref{eq:fixed-pt}. From Definition \ref{dfn:bethe-recursion} we {know} that $r_{\free}(\be_0,B_0)$ is the smallest 
nonnegative solution of \eqref{eq:fixed-pt}, {so having arrived} at a contradiction, 
the proof of the lemma is thus complete.  
\end{proof}

\begin{proof}[Proof of Proposition \ref{prop:non-unique-regime}]
From \cite[Theorem 1.11(a)]{DMS} we have the existence of the smooth curves $\be_{\free}(B)$ and $\be_+(B)$ with the desired properties. It also follows from there that for $(\be, B) \in [0,\infty)^2 \setminus \sfR_{\neq}$ one has $\nu_{\free}^{\be, B} = \nu_1^{\be, B}$. This together with Lemma \ref{lem:marginals} implies that all two dimensional marginals of $\mu_{\free}^{\be, B}$ and $\mu_{1}^{\be, B}$ coincide.
{Further, as} described in Remark \ref{rmk:infinite-vol-exist}, any finite dimensional marginal of $\mu_{\free}^{\be, B}$ then coincides with that of $\mu_{1}^{\be, B}$, {thereby yielding our} claim \eqref{eq:f-eq-1-B-ge-0} {from which} \eqref{eq:f-eq-all-B-eq-0} also follows {(since at} $B=0$ the Potts measures are invariant with respect to permutations on $[q]$). Since $r_{\free}^{\be, B} < r_1^{\be, B}$ for $(\be, B) \in \sfR_{\ne}$ our claim that $\nu_{\free}^{\be, B}(1) < \nu_1^{\be, B}(1)$ follows from the fact that map $r \mapsto e^r/(e^r+q-1)$ is strictly increasing on $\R$. 

{Now, for the existence of} smooth $B \to \be_c(B)$ of the required properties,
{it suffices to verify that}
\begin{align}\label{eq:partial-deri-ineq}
\f{\partial}{\partial \be}\Phi(\nu_{\free}^{\be, B}) < \frac{\partial}{\partial \be}\Phi(
\nu_1^{\be, B}), \qquad & \text{ for all } (\be, B) \in {\rm Int}(\sfR_{\neq})\,,\\
\label{eq:val-bdry-1}
\Phi(\nu_{\free}^{\be,B}) > \Phi(\nu_1^{\be,B}), \qquad & \text{ for } (\be, B) \in \partial \sfR_{\neq}^{\free},
\\
\label{eq:val-bdry-2}
\Phi(\nu_{\free}^{\be,B}) < \Phi(\nu_1^{\be,B}), \qquad & \text{ for } (\be, B) \in \partial \sfR_{\neq}^{+}.
\end{align}
{Indeed, \eqref{eq:partial-deri-ineq}-\eqref{eq:val-bdry-2} imply} that \eqref{eq:free-v-1} holds
{for some} $B \mapsto \be_c(B)$ such that $\be_{\free}(B) < \be_c(B) < \be_+(B)$ for $B \in [0, B_+)$, 
{with smoothness of $\be_c(\cdot)$ due to} the implicit function theorem
{(where}
\(
\f{\partial}{\partial \be} \Phi(\nu_{\free}^{\be, B}) \ne \f{\partial}{\partial \be} \Phi(\nu_1^{\be, B})
\)
{throughout} ${\rm Int}(\sfR_{\ne})$ thanks to \eqref{eq:partial-deri-ineq}). 

Turning to prove \eqref{eq:partial-deri-ineq}, {we set} $\bar \nu := \frac{1-\nu}{q-1}$
{and note} that the map
\[
\nu \mapsto \frac{\bar \nu (1+\nu - \bar \nu)}{\bar \nu^2 +(q-1)\nu^2}
\] 
is strictly decreasing on $[\f{1}{q},1)$ (having positive numerator and denominator, 
whose derivatives are strictly negative and positive, respectively, on the interval $(\f1q, 1)$). Using Lemma \ref{lem:marginals} and the fact that $\nu_\ddagger(i) =\nu_\ddagger(2)$ for $\ddagger \in \{\free, 1\}$ 
and all $i \in [q]\setminus \{1\}$, we also have at any $i \in \partial o$,
\[
\mu_{\ddagger}^{\be, B}(\sigma_o=\sigma_i) = \f{e^\be \left[ \nu_\ddagger(1)^2 + (q-1) \nu_\ddagger(2)^2\right]}{e^\be \left[ \nu_\ddagger(1)^2 + (q-1) \nu_\ddagger(2)^2\right]+ 2(q-1) \nu_\ddagger(1) \nu_\ddagger(2)+ (q-1) (q-2) \nu_\ddagger(2)^2}\,.
\]
With $\frac{1}{q} \le \nu_{\free}^{\be, B}(1) < \nu_1^{\be, B}(1)$ at any $(\be, B) \in \sfR_{\neq}$,
combining the last two observations results with
\[
\sum_{i \in \partial o} \mu_{\free}^{\be, B}(\sigma_o=\sigma_i) < \sum_{i \in \partial o} \mu_{1}^{\be, B}(\sigma_o=\sigma_i),  \qquad \qquad \forall \, (\be, B) \in \sfR_{\ne}
\]
from which \eqref{eq:partial-deri-ineq} follows upon using \eqref{eq:phi-derivative}. 

Moving next to the proof of \eqref{eq:val-bdry-1}, for any $\be \in [\be_-, \be_{\free}]$ let 
\[
\Psi_-(\be):=\Phi^{\be, B_-(\be)}(r_1(\be, B_-(\be)))- \Phi^{\be, B_-(\be)} (r_{\free}(\be, B_-(\be)))\,,
\]
with the convention that $\Phi^{\be, B}(r)= \Phi^{\be, B}(r(\nu)):= \Phi^{\be, B}(\nu)$ 
for $r(\nu):= \log(\nu(1)/\nu(2))$ and any $\nu \in \cP([q])$ such that $\nu(2)=\cdots=\nu(q)$.
Fix any $(\be_0, B_0) \notin \partial \sfR_{\neq}^+$. By Lemma \ref{lem:smoothness-fixed-pt}(ii), proceeding as in the steps leading to \eqref{eq:phi-derivative}, and using the definition of $r(\nu)$ we obtain that
\beq\label{eq:der-phi-beB-f}
\f{\partial}{\partial \be} \Phi^{\be_0, B_0}(r_{\free}(\be_0, B_0))= h_1(r_{\free}(\be_0, B_0); \be_0) \text{ and } \f{\partial}{\partial B} \Phi^{\be_0, B_0}(r_{\free}(\be_0, B_0))= h_2(r_{\free}(\be_0, B_0); \be_0), 
\eeq
where
\[
h_1(r;\be):= \f{d}{2} \cdot \f{e^\be(e^{2r}+q-1)}{e^r(e^{\be+r}+q-1)+ (q-1)(e^r+e^\be +q-2)}
\]
and
\[
 h_2(r;\be):= \f{e^r(e^{\be+r}+q-1)}{e^r(e^{\be+r}+q-1)+ (q-1)(e^r+e^\be +q-2)}. 
\]
By \eqref{eq:B-pm} {we see that} $\rho_+(\be_0)$ is a fixed point of \eqref{eq:fixed-pt} for $(\be,B)=(\be_0, B_-(\be_0))$. Recall that $\rho_+(\be)$ is the largest solution of \eqref{eq:implicit-cond-1} {and since} $r_1(\be, B)$ is the largest {solution of} \eqref{eq:fixed-pt}, we deduce that $\rho_+(\be_0)=r_1(\be_0, B_-(\be_0))$ for $\be_0 \in [\be_-, \be_{\free}]$. From \eqref{eq:implicit-cond-1} it can be further checked that $\rho_+(\be)$ is the $\log$ of the largest solution of a quadratic equation with coefficients smooth in $\be$ that does not admit a double root for $\be >\be_-$. Therefore, the map $\be \mapsto \rho_+(\be)$ is differentiable on $(\be_-, \be_{\free}]$. Hence, repeating a computation 
{as in the} steps leading to \eqref{eq:phi-derivative}, we find that \eqref{eq:der-phi-beB-f} continues to hold when $B_0$ is replaced by $B_-(\be_0)$ and $r_{\free}(\be_0, B_0)$ is replaced by $r_1(\be_0, B_-(\be_0))$. Moreover, by \eqref{eq:B-pm} we have that $\partial_\be B_-(\be_0)=-\partial_\be F(\rho_+(\be_0);\be_0,0)$ for $\be_0 \in (\be_-, \be_{\free}]$. Thus, using the chain rule of differentiation, we find that
\begin{multline}\label{eq:Psi-der}
\Psi_-'(\be_0)= \Delta h_1(\be_0) - \Delta h_2(\be_0) \cdot \f{\partial}{\partial \be}F(\rho_+(\be_0); \be_0, 0) \\
= \int_{r_{\free}(\be_0, B_-(\be_0))}^{r_1(\be_0, B_-(\be_0))} \left(\f{\partial}{\partial r} h_1(r; \be_0) - \f{\partial}{\partial r} h_2(r, \be_0) \cdot \f{\partial}{\partial \be}F(\rho_+(\be_0); \be_0, 0)\right) dr,
\end{multline}
where  
\[
\Delta h_i(\be_0):= h_i(r_1(\be_0, B_-(\be_0)); \be_0) - h_i(r_{\free}(\be_0, B_-(\be_0)); \be_0), \quad \text{ for } i=1,2. 
\]

Next denoting 
\[
g_{\be_0}(r):= \f{\f{\partial}{\partial r} h_1(r;\be_0)}{\f{\partial}{\partial r} h_2(r;\be_0)}  = d \cdot \f{e^{2r} +(q-2)e^r -(q-1)}{e^{2r} + 2(e^{\be_0}+q-2)e^r + (q-1)(1+e^{-\be_0}(q-2))}
\]
we find that $\sgn(g'_{\be_0}(r))= \sgn(2e^{\be_0}+q-2)\cdot \sgn(Q_{\be_0}(e^r))$, where $Q_{\be_0}(t):=t^2 +2 (q-1)e^{-\be_0} + (q-1)(1+(q-2)e^{-\be_0})$. Since the roots of $Q_{\be_0}$ are either negative or complex conjugates of each other we deduce that $g'_{\be_0}(r) >0$ for all $r \in \R$. This entails that
\beq\label{eq:der-gminus}
g_{\be_0}(r) < g_{\be_0}(r_1(\be_0, B_-(\be_0))) = \frac{\partial}{\partial \be} F(\rho_+(\be_0); \be_0, 0) \text{ for } r \in [0, r_1(\be_0, B_-(\be_0))),
\eeq
where to obtain the right most equality we use the fact that $\rho_+(\be_0)=r_1(\be_0, B_-(\be_0))$ satisfies \eqref{eq:implicit-cond-1}. So, noticing that $\f{\partial}{\partial r}h_2(\cdot;\be_0)$ is strictly positive, by \eqref{eq:der-gminus}, we find that the integrand in the \abbr{RHS} of \eqref{eq:Psi-der} is strictly negative for $r\in[r_{\free}(\be_0, B_-(\be_0), r_1(\be_0,B_-(\be_0)))$, {and consequently} 
that $\Psi_-'(\be_0) <0$ for $\be_0 \in (\be_-, \be_{\free}]$. {By} definition of the non-uniqueness regime, {necessarily} $\Psi_-(\be_-)=0$ {and} \eqref{eq:val-bdry-1} {follows}.

The proof of \eqref{eq:val-bdry-2} is similar. Indeed, noting that $r_{\free}(\be, B)= \rho_-(\be)$ for any $(\be, B) \in \partial \sfR_{\ne}^+$, denoting $\Psi_+(\be):=\Phi^{\be, B_+(\be)}(r_1(\be, B_+(\be)))- \Phi^{\be, B_+(\be)} (r_{\free}(\be, B_+(\be)))$ for any $\be\in [\be_-, \be_+]$, and proceeding as in the steps leading to \eqref{eq:Psi-der} we find that \eqref{eq:Psi-der} continues to hold for $\Psi_+'(\be_0)$ 
{at} any $\be_0 \in (\be_-,\be_+]$, provided we replace {there}
$\rho_+(\be_0)$ and $B_-(\be_0)$ by $\rho_-(\be_0)$ and $B_+(\be_0)$, respectively. Moreover, a same reasoning as in \eqref{eq:der-gminus} yields that
\[
g_{\be_0}(r) > g_{\be_0}(r_{\free}(\be_0, B_+(\be_0))) = \frac{\partial}{\partial \be} F(\rho_-(\be_0); \be_0, 0) \text{ for } r \in (r_{\free} (\be_0, B_+(\be_0)),\infty). 
\]
Repeating now the rest of the argument in the proof of \eqref{eq:val-bdry-1} 
we obtain \eqref{eq:val-bdry-2} and thereby complete the proof of the proposition.
\end{proof}

\subsection*{Acknowledgements} \red{We thank the anonymous referees for helpful comments that improved our presentation.  Research of AB was partially supported by DAE Project no.~RTI4001 via ICTS, the Infosys Foundation via the Infosys-Chandrashekharan Virtual Centre for Random Geometry, an Infosys–ICTS Excellence Grant, and a MATRICS grant (MTR/2019/001105) from Science and Engineering Research Board.  Research of AD was partially supported by NSF grants DMS-1954337 and DMS-2348142.  Research of AS was partially supported by Simons Investigator Grant.}

\subsection*{Declarations} The authors have no competing interests to declare that are relevant to the content of this article. Data sharing is not applicable -- no new data is generated.

\end{document} 

\appendix

\section{Comments on the proof of Theorem \ref{thm:main-1}}

\begin{enumerate}

\item Some (standard) definitions are not yet included in the draft. Please let me know if some of the notations are unclear.


\item I have checked that Lemma \ref{lem:connect-RC-Potts-T_d} holds when $B=0$. \red{Needs to check that it continues to hold for $B >0$.}

\item I have worked out the proof of Lemma \ref{lem:limit-point-properties} on paper. Needs to be typed up.



\item Proof of Lemma \ref{lem:RC-local-limit} follows from the scanned notes. Have verified that there are no gaps in the proof at the moment.

\item To complete the proof of Theorem \ref{thm:main-1} we need to find the conditional distribution of the spin variables given the bond variables. When $B=0$ this is done. \red{Needs to do it for $B >0$. I checked it earlier. It seemed okay then. But not entirely confident about this.}
\end{enumerate}